\documentclass[11pt]{article}

\usepackage{amsmath,amssymb,amsfonts,enumitem,amsthm,accents,pstricks,pst-node,pstricks-add,mathrsfs,xy}
\xyoption{all}

\numberwithin{equation}{section}

\newtheorem{thm}{Theorem}[section]
\newtheorem{lmm}[thm]{Lemma}
\newtheorem{prp}[thm]{Proposition}
\newtheorem{crl}[thm]{Corollary}
\theoremstyle{definition}
\newtheorem{dfn}[thm]{Definition}
\newtheorem{eg}[thm]{Example}
\newtheorem{rmk}[thm]{Remark}

\def\BE#1{\begin{equation}\label{#1}}
\def\EE{\end{equation}}
\def\eref#1{(\ref{#1})}

\def\lra{\longrightarrow}
\def\Lra{\Longrightarrow}

\def\lhra{\ensuremath{\lhook\joinrel\relbar\joinrel\rightarrow}}

\def\blr#1{\big\langle{#1}\big\rangle}
\def\ov#1{\overline{#1}}

\def\wt#1{\widetilde{#1}}
\def\sf#1{\textsf{#1}}
\def\ri#1{{#1}^{\circ}}
\def\sm#1{\begin{small}{#1}\end{small}}
\def\tn#1{\textnormal{#1}}
\def\wh#1{\widehat{#1}}

\def\al{\alpha}
\def\be{\beta}
\def\de{\delta}
\def\ep{\epsilon}

\def\io{\iota}
\def\ka{\kappa}
\def\la{\lambda}
\def\na{\nabla}
\def\om{\omega}

\def\th{\theta}
\def\ve{\varepsilon}

\def\ze{\zeta}

\def\Ga{\Gamma}
\def\De{\Delta}
\def\La{\Lambda}
\def\Om{\Omega}
\def\Th{\Theta}

\def\fc{\mathfrak c}
\def\bI{\mathbb I}

\def\C{\mathbb C}
\def\cC{\mathcal C}

\def\fD{\mathfrak D}

\def\bI{\mathbb I}
\def\cK{\mathcal K}

\def\cL{\mathcal L}

\def\cN{\mathcal N}
\def\cO{\mathcal O}
\def\P{\mathbb P}
\def\cP{\mathcal P}
\def\R{\mathbb R}
\def\cR{\mathcal R}
\def\fR{\mathfrak R}
\def\Q{\mathbb Q}

\def\X{\mathbf X}

\def\cZ{\mathcal Z}
\def\Z{\mathbb Z}

\def\ne{\textnormal{e}}
\def\fI{\mathfrak i}

\def\Aux{\tn{Aux}}

\def\Dom{\textnormal{Dom}}

\def\id{\textnormal{id}}
\def\Im{\textnormal{Im}}

\def\nd{\textnormal{d}}
\def\hor{\textnormal{hor}}

\def\top{\textnormal{top}}
\def\PD{\textnormal{PD}}
\def\pr{\textnormal{pr}}
\def\Re{\textnormal{Re}}
\def\Symp{\textnormal{Symp}}

\def\supp{\textnormal{supp}}

\def\ver{\textnormal{ver}}

\def\i{\infty}
\def\w{\wedge}

\def\eset{\emptyset}
\def\prt{\partial}
\def\1{\mathbf 1}
\def\hb{\hbar}

\def\bu{\bullet}

\begin{document}

\title{The Smoothability of Normal Crossings Symplectic Varieties}
\author{Mohammad F.~Tehrani, Mark McLean, and 
Aleksey Zinger\thanks{Partially supported by NSF grants 0846978 and 1500875}}

\date{\today}

\maketitle

\begin{abstract}
\noindent
Our previous paper introduces topological notions of normal crossings symplectic divisor
and variety and establishes that they are equivalent, in a suitable sense,
to the desired geometric notions.
Friedman's $d$-semistability condition  is well-known to be 
an obstruction to the smoothability of a normal crossings variety
in a one-parameter family with a smooth total space
in the algebraic geometry category.
We show that the direct analogue of this condition is the only obstruction
to such smoothability in the symplectic topology category.
Every smooth fiber of the families of smoothings we describe 
provides a multifold analogue of the now classical (two-fold) symplectic sum construction;
we thus establish an old suggestion of Gromov in a strong~form. 
\end{abstract}

\tableofcontents

\section{Introduction}
\label{intro_sec}

\noindent
Flat one-parameter families of smoothings are an important tool in algebraic geometry 
and raise considerable interest in related areas of symplectic topology.
The Gross-Siebert program~\cite{GS0} for a direct proof of mirror symmetry 
has highlighted in particular the significance of log smooth degenerations 
to log smooth algebraic varieties. 
A central part of this program is the study of Gromov-Witten invariants 
(which are fundamentally symplectic topology invariants) under such degenerations.
Such a study is undertaken from an algebro-geometric perspective in \cite{AC,Ch,GS}.
The almost complex analogue of the log smooth category provided by 
the exploded manifold category of~\cite{Brett0} underlines 
a similar study of Gromov-Witten invariants in~\cite{Brett}.
Log smooth varieties include varieties with \sf{normal crossings} (or \sf{NC}) singularities,
i.e.~singularities of the form $z_1\!\ldots\!z_N\!=\!0$ in complex coordinates.
Purely symplectic topology notions of an \sf{NC symplectic variety} and of 
a
\sf{one-parameter family of smoothings} of such a variety
are introduced in~\cite{SympDivConf} and in this paper, respectively.
It is straightforward to show that 
the direct analogue of the well-known triple point condition of algebraic geometry 
is an obstruction
for an NC symplectic variety to admit a one-parameter family of smoothings.
The main construction of this paper produces such a family for
every NC variety satisfying this direct analogue and thus establishes
the necessity {\it and} sufficiency of this condition.
A non-central fiber of such a family is a representative of the deformation equivalence class of 
the multifold symplectic sum construction on the central fiber envisioned
in~\cite[p343]{GrBook}.\\

\noindent
For a symplectic submanifold~$V$ in a symplectic manifold $(X,\om)$,
the normal bundle 
\BE{cNXVsymp_e}
\cN_XV\equiv \frac{TX|_V}{TV}\approx TV^{\om}
\equiv \big\{v\!\in\!T_xX\!:\,x\!\in\!V,\,\om(v,w)\!=\!0~\forall\,w\!\in\!T_xV\big\}
\lra V\EE
of~$V$ in~$X$ inherits a fiberwise symplectic form~$\om|_{\cN_XV}$ from~$\om$.
The space of complex structures on the fibers of~\eref{cNXVsymp_e}
compatible with (resp.~tamed~by)  $\om|_{\cN_XV}$ is non-empty and contractible;
we call such complex structures \sf{$\om$-compatible} (resp.~\sf{$\om$-tame}).
The now classical symplectic sum construction, indicated in~\cite[p343]{GrBook} 
and carried out in~\cite{Gf,MW}, smooths out the union of two symplectic manifolds 
$(X_1,\om_1)$ and $(X_2,\om_2)$ glued along a common compact smooth symplectic 
divisor $V\!\equiv\!X_{12}$ such~that
\BE{Gfcond_e}  c_1(\cN_{X_1}X_{12})+c_1(\cN_{X_2}X_{12})=0\in H^2(X_{12};\Z)\EE
into a new symplectic manifold $(X_{\#},\om_{\#})$.
From the complex geometry point of view, this construction replaces
the nodal singularity $z_1z_2\!=\!0$ in~$\C^n$, i.e.~the union of 
the two coordinate hyperplanes, by a smoothing $z_1z_2\!=\!\la$ with $\la\!\in\!\C^*$.\\

\noindent
In this paper, we describe a multifold version of the symplectic sum construction of~\cite{Gf,MW};
it in particular smooths out the union of several symplectic manifolds 
identified along transversely intersecting smooth divisors with a {\it single} 
smoothing parameter~$\la$.
From the complex geometry point of view, this construction replaces
the singularity $z_1\!\ldots\!z_N\!=\!0$ in~$\C^n$, i.e.~the union of 
the $N$ coordinate hyperplanes, by a smoothing $z_1\!\ldots\!z_N\!=\!\la$ with $\la\!\in\!\C^*$. 
An inverse degeneration construction,
which includes a multifold version of the symplectic cut procedure of~\cite{L},
is described in~\cite{SympCutMulti}.
The precise relation between the smoothing/sum construction of the present paper
and the degeneration/cut construction of~\cite{SympCutMulti} is the subject of~\cite{SympSumCut}.\\

\noindent
The topological restriction~\eref{Gfcond_e} is equivalent to the existence of an isomorphism
\BE{cOGfbdl_e}\cN_{X_1}X_{12}\otimes_{\C}\cN_{X_2}X_{12}
\approx X_{12}\!\times\!\C\EE
of complex line bundles over~$X_{12}$.
The topological type of~$X_{\#}$ in~\cite{Gf} depends only on 
the homotopy class of such an isomorphism.
With such a choice fixed, the construction of~\cite{Gf} involves choosing 
an $\om_1$-compatible almost complex structure on $\cN_{X_1}X_{12}$,
an $\om_2$-compatible almost complex structure on $\cN_{X_2}X_{12}$,
and a representative for the above homotopy class.
Because of these choices, the resulting symplectic manifold $(X_{\#},\om_{\#})$
is determined by $(X_1,\om_1)$, $(X_2,\om_2)$, and the choice of the homotopy class
only up to symplectic deformation equivalence.
Since the symplectic deformations of the~tuple
\BE{Gfconf_e}\big((X_1,X_2,X_{12}),(\om_1,\om_2)\big)\EE
do not affect  the deformation equivalence class of $(X_{\#},\om_{\#})$,
it would have been sufficient to carry out the symplectic sum construction
of~\cite{Gf} only on a path-connected set of representatives for
each deformation equivalence class of the tuples~\eref{Gfconf_e}.
This change in perspective turns out to be very useful for smoothing out 
NC symplectic varieties, 
including unions of several symplectic manifolds glued along transversally intersecting
smooth divisors.\\

\noindent
The one-parameter family $z_1\!\ldots\!z_N\!=\!\la$ of smoothings of
the union of the $N$ coordinate hyperplanes~$\C^N_i$ in~$\C^N$ involves
compatible complex structures on~$\C^N_i$ that preserve all
coordinate  subspaces $\C^N_I\!\subsetneq\!\C^N$.
For each $i\!=\!1,\ldots,N$, the union of the codimension~2 coordinate subspaces~$\C_{ij}^N$
with $j\!\neq\!i$ is a  \sf{simple crossings} (or \sf{SC}) K\"ahler
\sf{divisor} in~$\C^N_i$.
Analogues of this notion and of the related notion of an \sf{SC variety} 
in the symplectic category are introduced in~\cite{SympDivConf}
and reviewed in Section~\ref{SCdfn_subs} of the present paper; 
see Definitions~\ref{SCD_dfn} and~\ref{SCC_dfn}.
In the terminology of Definition~\ref{SCC_dfn},
the~tuple~\eref{Gfconf_e} is a \sf{2-fold SC symplectic configuration} 
and the~tuple
\BE{Gfconf_e2}\big(X_1\!\cup_{X_{12}}\!X_2,(\om_1,\om_2)\big)\EE
is the \sf{associated NC~symplectic variety}.
As noted at the end of Section~\ref{SCdfn_subs},
an SC symplectic variety~$X_{\eset}$ comes with a natural complex line 
bundle~$\cO_{X_{\prt}}(X_{\eset})$ over its \sf{singular locus}~$X_{\prt}$;
see~\eref{PsiDfn_e}.
We call it the \sf{normal bundle} of~$X_{\prt}$ in~$X_{\eset}$;
it reduces to the left-hand side of~\eref{cOGfbdl_e} in the setting of~\cite{Gf}.
By Theorem~\ref{SympSum_thm}, an SC symplectic variety~$X_{\eset}$ 
is smoothable in a one-parameter family 
if and only if the line bundle $\cO_{X_{\prt}}(X_{\eset})$ is trivial.
Furthermore, the possible families of smoothings are again classified by
the homotopy classes of its trivializations.
We give two examples in Section~\ref{SCCeg_subs}.
In~\cite{SympNC}, we extend Theorem~\ref{SympSum_thm} to arbitrary NC symplectic varieties
and give more elaborate examples of the associated smoothings.\\

\noindent
Theorem~\ref{SympSum_thm} leads to and has further potential for 
very different applications in symplectic topology.
First and foremost, it includes a new surgery construction for symplectic manifolds
and thus opens the possibility of generating new symplectic manifolds. 
Furthermore, it fits naturally with a decade-long program to develop 
decomposition formulas for Gromov-Witten invariants under one-parameter families
of almost K\"ahler (or projective) degenerations;
approaches to this problem appear in \cite{AC,Ch,GS,Brett}.
An immediate consequence of Theorem~\ref{SympSum_thm}, along with \cite[Theorem~2.17]{SympDivConf}, 
is that the decomposition formulas arising from~\cite{Brett} include 
splitting formulas for Gromov-Witten invariants of 
the $N$-fold symplectic sums of Theorem~\ref{SympSum_thm}.
Since the decomposition formulas of~\cite{Brett} have connections with tropical geometry,
Theorem~\ref{SympSum_thm} may have applications in this field as well.
It should also have applications in the theory of singularities,
as an isolated singularity can often be studied by smoothing~it 
and then applying symplectic techniques as in \cite{Mi,Se}.
Theorem~\ref{SympSum_thm} provides a purely topological condition 
for the smoothability of a singularity symplectically after a sequence of blowups
that turns it into a simple crossings form
(even though it may not be smoothable algebraically).\\

\noindent
Surgery constructions on 4-dimensional symplectic manifolds along {\it pairwise}
positively intersecting immersed surfaces are described in
\cite{SymingtonThesis,Symington3}.
While these are called $N$-fold symplectic sum constructions,
this terminology agrees with ours (which is consistent with
algebraic geometry and \cite[p343]{GrBook}) only for $N\!=\!3$.
In particular, the setting of \cite[Theorem~2.7]{Symington3}
is a specialization of the $N\!=\!3$ case of the setting of our Theorem~\ref{SympSum_thm}.
The output of \cite[Theorem~2.7]{Symington3} is then symplectically deformation equivalent
to the smooth fibers of the one-parameter family provided by Theorem~\ref{SympSum_thm}.
The perspectives taken in~\cite{Symington3} and the present paper are fundamentally different
as~well.
The viewpoint taken in the former is that of surgery on 4-dimensional manifolds;
our viewpoint is that of smoothing a variety in a one-dimensional family with 
a smooth total space.
The configurations in~\cite{Symington3} with $N\!\ge\!4$ correspond to varieties, such~as
\BE{Gfconf_e4}\big\{(x,y,z,w)\!\in\!\C^4:\,xy\!=\!0,\,zw\!=\!0\big\},\EE
that do not even admit such smoothings.
The total space of the natural one-parameter smoothing of~\eref{Gfconf_e4},
i.e.~with~$0$ replaced by~$\la\!\in\!\C$, is singular at the origin.
On the other hand, the total space of this family is smooth in 
the logarithmic category central to the mirror symmetry program of~\cite{GS0}
and in the exploded manifold category of~\cite{Brett}.
Unfortunately, symplectic topology analogues of the singularity described by~\eref{Gfconf_e4}
are yet to be defined.\\

\noindent
Notions of symplectic regularizations for an SC divisor
$\{V_i\}_{i\in S}$ in~$X$ and a configuration~$\X$ are introduced
in \cite[Sections~2.2,2,3]{SympDivConf} and recalled 
in Sections~\ref{SCDregul_subs} and~\ref{SCCregul_subs} of the present paper;
see Definitions~\ref{SCDregul_dfn} and~\ref{SCCregul_dfn}.
Such regularizations provide essential auxiliary data for
the multifold symplectic smoothing/sum construction of Theorem~\ref{SympSum_thm},
just as they did in the $N\!=\!2$ case addressed in \cite{Gf,MW}.
By \cite[Theorem~2.17]{SympDivConf}, the space $\Symp^+(\X)$ of symplectic structures 
on~$\X$ is weakly homotopy equivalent to the space $\Aux(\X)$ of pairs consisting of 
a symplectic structure on~$\X$ and a compatible regularization.
In Section~\ref{SCCprpPf_subs}, we show that a given trivialization of
the complex line bundle~$\cO_{X_{\prt}}(X_{\eset})$  can be homotoped 
to be compatible with a given symplectic regularization for~$\X$
in a suitable sense
and that any two compatible trivializations are homotopic to each other through    
compatible trivializations.
While the projection map from $\Aux(\X)$
to $\Symp^+(\X)$ need not be surjective in general (in contrast to the $N\!=\!2$ case), 
this is not an issue for typical applications in symplectic topology.\\

\noindent
In Section~\ref{SympSumPf_sec}, we show that the triviality of $\cO_{X_{\prt}}(X_{\eset})$
is sufficient for the existence of a one-parameter
family of smoothings of the symplectic variety associated to an SC symplectic configuration, 
up to symplectic deformation equivalence.
By Proposition~\ref{InducedIsom_prp}, this condition is necessary and in fact 
every one-parameter family of smoothings determines a homotopy class
of trivializations of~$\cO_{X_{\prt}}(X_{\eset})$.
By Proposition~\ref{PhiExt_prp}, the homotopy class determined by 
a one-parameter family constructed as in Section~\ref{SympSumPf_sec} 
is the homotopy class used in its construction.
Appendix~\ref{conn_app} collects some basic facts concerning connections
on vector bundles.
Appendix~\ref{RemainPf_sec} provides a more intrinsic perspective 
on the smoothability criterion of Theorem~\ref{SympSum_thm}.\\

\noindent
We would like to thank E.~Lerman for pointing out related literature and
B.~Parker and D.~Sullivan for related discussions.

\section{Main theorem}
\label{MainThm_sec}

\noindent
We begin by introducing the most commonly used notation.
If $N\!\in\!\Z^{\ge0}$ and $I\!\subset\!\{1,\ldots,N\}$, let 
$$[N]=\{1,\ldots,N\}, \qquad
\C_I^N=\big\{(z_1,\ldots,z_N)\!\in\!\C^n\!:\,z_i\!=\!0~\forall\,i\!\in\!I\big\}.$$
Denote by $\cP(N)$ the collection of subsets of~$[N]$ and
by $\cP^*(N)\!\subset\!\cP(N)$ the collection of nonempty subsets.
If $\cN\!\lra\!V$ is a vector bundle, $\cN'\!\subset\!\cN$, and $V'\!\subset\!V$, we define
\BE{cNrestrdfn_e} \cN'|_{V'}=\cN|_{V'}\cap\cN'\,.\EE
Let $\bI\!=\![0,1]$.

\subsection{Preliminaries}
\label{SCdfn_subs}

\noindent
We first recall the notions of simple crossings (or \sf{SC}) symplectic divisor and
variety introduced, described in more detail, and illustrated with 
examples in \cite[Section~2.1]{SympDivConf}.
We then define a natural complex line bundle~$\cO_{X_{\prt}}(X_{\eset})$ over
the \sf{singular locus} $X_{\prt}$ of an SC symplectic variety~$X_{\eset}$.\\

\noindent
Let $X$ be a (smooth) manifold. 
For any submanifold $V\!\subset\!X$, let
$$\cN_XV\equiv \frac{TX|_V}{TV}\lra V$$
denote the normal bundle of~$V$ in~$X$.
For a collection $\{V_i\}_{i\in S}$ of submanifolds of~$X$ and $I\!\subset\!S$, let
$$V_I\equiv \bigcap_{i\in I}\!V_i\subset X\,.$$
Such a collection is called \sf{transverse} if any subcollection $\{V_i\}_{i\in I}$ 
of these submanifolds intersects transversely, i.e.~the homomorphism
\BE{TransVerHom_e}
T_xX\oplus\bigoplus_{i\in I}T_xV_i\lra \bigoplus_{i\in I}T_xX, \qquad
\big(v,(v_i)_{i\in I}\big)\lra (v\!+\!v_i)_{i\in I}\,,\EE
is surjective for all $x\!\in\!V_I$. 
Each subspace $V_I\!\subset\!X$ is then a submanifold of~$X$ 
and the homomorphisms 
\BE{cNorient_e2}\begin{split}
\cN_XV_I\lra \bigoplus_{i\in I}\cN_XV_i\big|_{V_I}\quad&\forall~I\!\subset\!S,\qquad
\cN_{V_{I-i}}V_I\lra \cN_XV_i\big|_{V_I} \quad\forall~i\!\in\!I\!\subset\!S,\\
&\bigoplus_{i\in I-I'}\!\!\cN_{V_{I-i}}V_I\lra \cN_{V_{I'}}V_I \quad\forall~I'\!\subset\!I\!\subset\!S
\end{split}\EE
induced by inclusions of the tangent bundles are isomorphisms.\\

\noindent
If $X$ is an oriented manifold,
a transverse collection $\{V_i\}_{i\in S}$ of oriented submanifolds of~$X$
of even codimensions  induces an orientation of each submanifold $V_I\!\subset\!X$
with $|I|\!\ge\!2$, which we call \sf{the intersection orientation of~$V_I$}.
If $V_I$ is zero-dimensional, it is a discrete collection of points in~$X$
and the homomorphism~\eref{TransVerHom_e} is an isomorphism at each point $x\!\in\!V_I$;
the intersection orientation of~$V_I$ at $x\!\in\!V_I$
then corresponds to a plus or minus sign, depending on whether this isomorphism
is orientation-preserving or orientation-reversing.
For convenience, we call the original orientations of 
$X\!=\!V_{\eset}$ and $V_i\!=\!V_{\{i\}}$ \sf{the intersection orientations}
of these submanifolds~$V_I$ of~$X$ with $|I|\!<\!2$.\\

\noindent
Suppose $(X,\om)$ is a symplectic manifold and $\{V_i\}_{i\in S}$ is a transverse collection 
of submanifolds of~$X$ such that each $V_I$ is a symplectic submanifold of~$(X,\om)$.
Each $V_I$ then carries an orientation induced by $\om|_{V_{I}}$,
which we will call the \sf{$\om$-orientation}.
If $V_I$ is zero-dimensional, it is automatically a symplectic submanifold of~$(X,\om)$;
the $\om$-orientation of~$V_I$ at each point $x\!\in\!V_I$ corresponds to the plus sign 
by definition.
By the previous paragraph, the $\om$-orientations of~$X$ and~$V_i$ with $i\!\in\!I$
also induce intersection orientations on all~$V_I$.

\begin{dfn}\label{SCD_dfn}
Let $(X,\om)$ be a symplectic manifold.
An \sf{SC symplectic divisor} 
in~$(X,\om)$ is a finite transverse collection 
$\{V_i\}_{i\in S}$ of closed submanifolds of~$X$ of codimension~2 such that 
$V_I$ is a symplectic submanifold of~$(X,\om)$ for every $I\!\subset\!S$
and the intersection and $\om$-orientations of~$V_I$~agree.
\end{dfn}

\noindent
The intersection and symplectic orientations of~$V_I$ agree if $|I|\!<\!2$.
Thus, an SC symplectic divisor $\{V_i\}_{i\in S}$ with $|S|\!=\!1$ is 
a smooth symplectic divisor in the usual sense. 
If $(X,\om)$ is a 4-dimensional symplectic manifold, 
a finite transverse collection  $\{V_i\}_{i\in S}$ of closed symplectic submanifolds of~$X$ 
of codimension~2 is an SC symplectic divisor if all points of the pairwise intersections
$V_{i_1}\!\cap\!V_{i_2}$ with $i_1\!\neq\!i_2$ are positive;
these are the cases considered in~\cite{SymingtonThesis,Symington3}.

\begin{dfn}\label{SCdivstr_dfn}
Let $X$ be a manifold and $\{V_i\}_{i\in S}$ be a finite transverse collection of 
closed submanifolds of~$X$ of codimension~2.
A \sf{symplectic structure on $\{V_i\}_{i\in S}$ in~$X$} is a symplectic form~$\om$ 
on~$X$ such that $V_I$ is a symplectic submanifold of $(X,\om)$ for all $I\!\subset\!S$.
\end{dfn}

\noindent
For $X$ and $\{V_i\}_{i\in S}$ as in Definition~\ref{SCdivstr_dfn}, 
we denote by $\Symp(X,\{V_i\}_{i\in S})$ the space of all symplectic structures 
on $\{V_i\}_{i\in S}$ in~$X$ and by 
$$\Symp^+\big(X,\{V_i\}_{i\in S}\big)\subset \Symp\big(X,\{V_i\}_{i\in S}\big)$$
the subspace of the symplectic forms~$\om$ such that $\{V_i\}_{i\in S}$
is an SC symplectic divisor in~$(X,\om)$.

\begin{dfn}\label{TransConf_dfn1}
Let $N\!\in\!\Z^+$.
An \sf{$N$-fold transverse configuration} is a tuple $\{X_I\}_{I\in\cP^*(N)}$
of manifolds such that $\{X_{ij}\}_{j\in[N]-i}$ is a transverse collection 
of submanifolds of~$X_i$ for each $i\!\in\![N]$ and
$$X_{\{ij_1,\ldots,ij_k\}}\equiv \bigcap_{m=1}^k\!\!X_{ij_m}
=X_{ij_1\ldots j_k}\qquad\forall~j_1,\ldots,j_k\in[N]\!-\!i.$$
\end{dfn}

\begin{dfn}\label{TransConf_dfn2}
Let $N\!\in\!\Z^+$ and $\X\!\equiv\!\{X_I\}_{I\in\cP^*(N)}$ be an $N$-fold transverse configuration
such that $X_{ij}$ is a closed submanifold of~$X_i$ of codimension~2
for all $i,j\!\in\![N]$ distinct.
A \sf{symplectic structure on~$\X$} is a~tuple 
$$(\om_i)_{i\in[N]}\in 
\prod_{i=1}^N\Symp\big(X_i,\{X_{ij}\}_{j\in[N]-i}\big)$$
such that $\om_{i_1}|_{X_{i_1i_2}}\!=\!\om_{i_2}|_{X_{i_1i_2}}$ for all $i_1,i_2\!\in\![N]$.
\end{dfn}

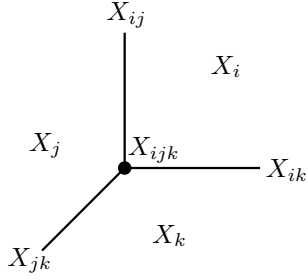
\begin{figure}
\begin{pspicture}(-3,-2)(11,2)
\psset{unit=.3cm}
\psline[linewidth=.1](15,-2)(21,-2)\psline[linewidth=.1](15,-2)(15,4)
\psline[linewidth=.1](15,-2)(11.34,-5.66)\pscircle*(15,-2){.3}
\rput(19.5,2.5){\sm{$X_i$}}\rput(11.5,-1){\sm{$X_j$}}\rput(17,-5){\sm{$X_k$}}
\rput(22.2,-2.1){\sm{$X_{ik}$}}\rput(15.1,4.8){\sm{$X_{ij}$}}
\rput(10.8,-6.1){\sm{$X_{jk}$}}\rput(16.3,-1.2){\sm{$X_{ijk}$}}
\end{pspicture}
\caption{A 3-fold simple crossings configuration and variety.}
\label{3conf_fig}
\end{figure}

\noindent
For an $N$-fold transverse configuration as in Definition~\ref{TransConf_dfn1}, let
\begin{gather}\label{Xesetdfn_e}
X_{\eset}=\bigg(\bigsqcup_{i=1}^NX_i\bigg)\bigg/\!\!\sim, \quad
X_i\ni x\sim x\in X_j~~\forall~x\in X_{ij}\subset X_i,X_j,~i\neq j\,,\\
\label{Xprtdfn_e}
X_{\prt}\equiv\bigcup_{I\in\cP(N),|I|=2}\hspace{-.25in}\!X_I\subset X_{\eset}\,.
\end{gather}
The SC variety~$X_{\eset}$ associated to a 3-fold SC configuration is shown
in Figure~\ref{3conf_fig}.
For $k\!\in\!\Z^{\ge0}$, we call a tuple $(\om_i)_{i\in[N]}$ a \sf{$k$-form on~$X_{\eset}$}
if $\om_i$ is a $k$-form on~$X_i$ for each $i\!\in\![N]$ and 
$$\om_i\big|_{X_{ij}}=\om_j\big|_{X_{ij}} \qquad\forall~i,j\!\in\![N].$$

\vspace{.2in}

\noindent
For $\X$ as in Definition~\ref{TransConf_dfn2}, 
let $\Symp(\X)$ denote the space of all symplectic structures 
on $\X$ and
\BE{Sympdfn_e2}
\Symp^+\big(\X\big)= \Symp\big(\X\big)
\cap \prod_{i=1}^N\Symp^+\big(X_i,\{X_{ij}\}_{j\in[N]-i}\big)\,.\EE
Thus, if $(\om_i)_{i\in[N]}$ is an element of $\Symp^+(\X)$,
then $\{X_{ij}\}_{j\in[N]-i}$ is an SC symplectic divisor in $(X_i,\om_i)$
for each $i\!\in\![N]$.

\begin{dfn}\label{SCC_dfn}
Let $N\!\in\!\Z^+$.
An \sf{$N$-fold simple crossings} (or \sf{SC}) \sf{symplectic configuration} 
is a~tuple 
\BE{SCCdfn_e}\X=\big((X_I)_{I\in\cP^*(N)},(\om_i)_{i\in[N]}\big)\EE
such that $\{X_I\}_{I\in\cP^*(N)}$ is an $N$-fold transverse configuration,
$X_{ij}$ is a closed submanifold of~$X_i$ of codimension~2
for all $i,j\!\in\![N]$ distinct, and
$(\om_i)_{i\in[N]}\in\Symp^+(\X)$.
The \sf{SC symplectic variety associated~to} such a tuple~$\X$ 
is the pair~$(X_{\eset},(\om_i)_{i\in[N]})$.
\end{dfn}

\noindent
Suppose $(X,\om)$ is a compact symplectic manifold 
and $V\!\subset\!X$ is a smooth symplectic divisor, 
i.e.~$|S|\!=\!1$ in the notation of Definition~\ref{SCD_dfn}.
Fix an identification~$\Psi$ of a tubular neighborhood~$D_X^{\ep}V$ of~$V$ in~$\cN_XV$ 
with a tubular neighborhood of~$V$ in~$X$ (i.e.~a \sf{regularization} of~$V$ in~$X$
in the sense of Definition~\ref{smreg_dfn}) and
an $\om$-tame complex structure~$\fI$ on~$\cN_XV$. 
Let
\begin{gather}
\label{cOXVdfn_e}
\cO_X(V)= \big(\Psi^{-1\,*}\pi_{\cN_XV}^*\cN_XV|_{\Psi(D_X^{\ep}V)} \sqcup 
(X\!-\!V)\!\times\!\C\big)\big/\!\!\sim\,\lra X,\\
\notag
\Psi^{-1\,*}\pi_{\cN_XV}^*\cN_XV|_{\Psi(D_X^{\ep}V)} 
  \ni\big(\Psi(v),v,cv\big)\sim\big(\Psi(v),c\big)\in (X\!-\!V)\!\times\!\C.
\end{gather}
This is a complex line bundle over $X$ with $c_1(\cO_X(V))\!=\!\PD_X([V]_X$),
where $[V]_X$ is the homology class in~$X$ represented by~$V$.
The space of pairs~$(\Psi,\fI)$ involved in explicitly constructing this line bundle
is contractible.\\

\noindent
Suppose $\X$ is an SC symplectic configuration as in~\eref{SCCdfn_e}.
If $i,j,k\!\in\![N]$ are distinct, the inclusion
$(X_{jk},X_{ijk})\!\lra\!(X_j,X_{ij})$ induces an isomorphism
\BE{restrisom_e}\cN_{X_{jk}}X_{ijk}\!\equiv\!\frac{TX_{jk}|_{X_{ijk}}}{TX_{ijk}}
\stackrel{\approx}{\lra} 
\frac{TX_j|_{X_{ijk}}}{TX_{ij}|_{X_{ijk}}}\!\equiv\! \cN_{X_j}X_{ij}\big|_{X_{ijk}}\EE
of rank~2 real vector bundles over~$X_{ijk}$;
this is a special case of the second isomorphism in~\eref{cNorient_e2} for \hbox{$X\!=\!X_j$}.
Thus, the rank~2 real vector bundles $\cN_{X_j}X_{ij}|_{X_{ijk}}$ and $\cN_{X_k}X_{ik}|_{X_{ijk}}$
are canonically identified with~$\cN_{X_{jk}}X_{ijk}$.
Let 
$$\Psi_{ij;j}\!: \cN_{ij;j}'\lra X_j, \qquad i,j\!\in\![N],\,i\!\neq\!j,$$
be a collection of
identifications of tubular neighborhoods of~$X_{ij}$ in~$\cN_{X_j}X_{ij}$
and in~$X_j$ so~that
\BE{restrisom_e3a}\Psi_{ij;j}\big|_{\cN_{ij;j}'\cap\cN_{X_{jk}}X_{ijk}}= 
\Psi_{ik;k}\big|_{\cN_{ik;k}'\cap\cN_{X_{jk}}X_{ijk}}\EE
for all $i,j,k\!\in\![N]$ with $k,j\!\neq\!i$.\\

\noindent
Since the intersection and $\om_j$-orientations of 
$$X_{ijk}=X_{ij}\cap X_{jk}\subset X_j$$ 
agree, the isomorphism~\eref{restrisom_e} is orientation-preserving
with respect to the orientations induced by $(\om_j|_{X_{jk}})|_{\cN_{X_{jk}}X_{ijk}}$
and $\om_j|_{\cN_{X_j}X_{ij}}$.
Thus, we can choose a collection of $\om_j$-tame complex structures~$\fI_{ij;j}$ on
the vector bundles~$\cN_{ij;j}$ so~that
\BE{restrisom_e3b}
\fI_{ij;j}\big|_{\cN_{X_{jk}}X_{ijk}}=\fI_{ik;k}\big|_{\cN_{X_{jk}}X_{ijk}}\EE
for all $i,j,k\!\in\![N]$ with $k,j\!\neq\!i$.\\

\noindent
For $i,j\!\in\![N]$ distinct, let $\cO_{X_j}(X_{ij})$ be the complex line bundle
over~$X_j$ constructed as in~\eref{cOXVdfn_e} using the identification~$\Psi_{ij;j}$
and the complex structure~$\fI_{ij;j}$.
By~\eref{restrisom_e3a} and~\eref{restrisom_e3b}, there are canonical identifications
\BE{cOrestr_e}\cO_{X_j}(X_{ij})\big|_{X_{ijk}}=\cO_{X_{jk}}(X_{ijk})=
\cO_{X_k}(X_{ik})\big|_{X_{ijk}}\EE
for all $i,j,k\!\in\![N]$ with $j,k\!\neq\!i$.
For each $i\!\in\![N]$,
\begin{gather}
\label{cOXXidfn_e}
\cO_{X_i^c}(X_i)\equiv\bigg(\bigsqcup_{j\in[N]-\{i\}}\hspace{-.2in}\cO_{X_j}(X_{ij})\bigg)
\!\!\bigg/\!\!\!\sim\,\lra 
X_i^c\!\equiv\bigcup_{j\in[N]-\{i\}}\hspace{-.2in}X_j\subset X_{\eset}\,,\\
\notag
\cO_{X_j}(X_{ij})\big|_{X_{ijk}} \ni u\sim u\in  \cO_{X_k}(X_{ik})\big|_{X_{ijk}}
\quad\forall\,i,j,k\!\in\![N],\,j,k\!\neq\!i,
\end{gather}
is thus a well-defined complex line bundle.
Let $\cO_{X_{\prt}}(X_i)\!=\!\cO_{X_i^c}(X_i)|_{X_{\prt}}$.\\

\noindent
We call the complex line bundle
\BE{PsiDfn_e}
\cO_{X_{\prt}}(X_{\eset})\equiv\bigotimes_{i=1}^N\cO_{X_{\prt}}(X_i)\EE
the \sf{normal bundle} of the singular locus~$X_{\prt}$ in~$X_{\eset}$.
The space of the collections of pairs~$(\Psi_{ij;j},\fI_{ij;j})$ 
involved in explicitly constructing this line bundle is contractible.
By~\eref{cOrestr_e},
$$\cO_{X_{\prt}}(X_{\eset})\big|_{X_{ij}}=
\cN_{X_i}X_{ij}\otimes\cN_{X_j}X_{ij}\otimes
\bigotimes_{k\in[N]-\{i,j\}}\hspace{-.22in}\cO_{X_{ij}}(X_{ijk})
\quad\forall\,i,j\!\in\![N],\,i\!\neq\!j.$$
In the $N\!=\!2$ case, this line bundle is the left-hand side of~\eref{cOGfbdl_e}.

\subsection{Statement}
\label{SympSumThm_subs}

\noindent
We now describe the setup for our smoothing/sum construction in the symplectic
topology category.
Theorem~\ref{SympSum_thm} provides a necessary and sufficient topological condition
for when it can be carried out.

\begin{dfn}\label{SimpFibr_dfn}
If $(\cZ,\om_{\cZ})$ is a symplectic manifold and $\De\!\subset\!\C$
is a disk centered at the origin,
a smooth surjective map $\pi\!:\cZ\!\lra\!\De$ is a \textsf{nearly regular symplectic fibration}~if 
\begin{enumerate}[label=$\bullet$,leftmargin=*]

\item $\cZ_0\!\equiv\!\pi^{-1}(0)\!=\!X_1\!\cup\!\ldots\!\cup\!X_N$ for some SC symplectic divisor
$\{X_i\}_{i\in[N]}$ in~$(\cZ,\om_{\cZ})$, 
\item  $\pi$ is a submersion outside of the submanifolds $X_I$ with $|I|\!\ge\!2$,

\item for every $\la\!\in\!\De\!-\!\{0\}$, 
the restriction~$\om_{\la}$ of~$\om_{\cZ}$ to $\cZ_{\la}\equiv\pi^{-1}(\la)$ is nondegenerate.

\end{enumerate} 
\end{dfn}

\vspace{.1in}

\noindent
We call a nearly regular symplectic fibration as in Definition~\ref{SimpFibr_dfn} 
a \sf{one-parameter family of smoothings} of the SC variety $(X_{\eset},(\om_i)_{i\in[N]})$
associated to the SC symplectic configuration~\eref{SCCdfn_e} with
$$X_I=\bigcap_{i\in I}X_i\subset X_{\eset}\!=\!\cZ_0\subset\cZ
\qquad \forall~I\!\in\!\cP^*(N)\,.$$
We call (the deformation equivalence class of)
an SC symplectic variety $(X_{\eset},(\om_i)_{i\in[N]})$ \sf{smoothable}
if some SC~symplectic variety $(X_{\eset},(\om_i')_{i\in[N]})$ deformation equivalent
to $(X_{\eset},(\om_i)_{i\in[N]})$ admits  a one-parameter family of smoothings.
Theorem~\ref{SympSum_thm} below provides a necessary and sufficient topological condition 
for the smoothability of an~SC symplectic variety.

\begin{thm}\label{SympSum_thm}
Let $\X$ be an $N$-fold SC symplectic configuration as in~\eref{SCCdfn_e}.
The associated SC symplectic variety $(X_{\eset},(\om_i)_{i\in[N]})$ is smoothable
if and only~if the normal bundle $\cO_{X_{\prt}}(X_{\eset})$ of its singular locus
is trivializable.
Furthermore, the germ of the deformation equivalence class of 
the smoothing $(\cZ,\om_{\cZ},\pi)$ provided by the proof of this statement 
is determined by a homotopy class of trivializations of~$\cO_{X_{\prt}}(X_{\eset})$.
If in addition $X_{\prt}$ is compact,  
the deformation equivalence class of a smooth fiber~$(\cZ_{\la},\om_{\la})$
is also determined by a homotopy class of these trivializations.
\end{thm}

\begin{rmk}\label{SmClassif_rmk}
In a future paper, we expect to show that the deformation equivalence class of {\it any}
one-parameter family of smoothings of $(X_{\eset},(\om_i)_{i\in[N]})$ 
corresponds to a homotopy class of trivializations of~\eref{PsiDfn_e}.
This is equivalent to every such smoothing being equivalent to 
a smoothing as constructed in Section~\ref{SympSumPf_sec}.
\end{rmk}

\noindent
By standard \v{C}ech cohomology considerations and~\eref{PsiDfn_e}, 
the complex line bundle $\cO_{X_{\prt}}(X_{\eset})$ is trivializable
if and only if
\BE{SympSumCond_e} \sum_{i=1}^N c_1\big(\cO_{X_{\prt}}(X_i)\big)=0 \in 
\check{H}^2(X_{\prt};\Z)\,.\EE
By \cite[Corollary~6.9.5]{Spanier}, the \v{C}ech and singular cohomologies
of $X_{\prt}$ (as well as of all other spaces in this paper) are canonically
isomorphic.\\

\noindent
If $N\!=\!1$, \eref{SympSumCond_e} imposes no condition.
In this case,
we can take $(\cZ,\om_{\cZ})$ to be the product symplectic manifold 
$(X_1,\om_1)\!\times\!(\De,\om_{\C})$,
where $\om_{\C}$ is the standard symplectic form on~$\C$.
The $N\!=\!2$ case of Theorem~\ref{SympSum_thm} is the symplectic sum
construction of~\cite{Gf,MW} for the SC symplectic variety~\eref{Gfconf_e2}.
It glues two symplectic manifolds $(X_1,\om_1)$ and $(X_2,\om_2)$ 
along normal circle bundles of a common symplectic divisor $(X_{12},\om_{12})$
if 
\begin{equation*}\begin{split}
c_1\big(\cO_{X_{\prt}}(X_1)\big)+c_1\big(\cO_{X_{\prt}}(X_2)\big)
&\equiv c_1\big(\cN_{X_1}X_{12}\big)+c_1\big(\cN_{X_2}X_{12}\big)\\
&=0\in H^2(X_{12};\Z)\equiv H^2(X_{\prt};\Z),
\end{split}\end{equation*}
i.e.~\eref{Gfcond_e} is satisfied.\\

\noindent
In general, the condition~\eref{SympSumCond_e} implies that
\BE{SympSumCond_e1}
c_1\big(\cN_{X_i}X_{ij}\big)+ c_1\big(\cN_{X_j}X_{ij}\big)
+\sum_{k\in[N]-\{i,j\}}\hspace{-.22in}\big[X_{ijk}\big]_{X_{ij}}=0
\qquad\forall~i,j\in[N],~i\neq j.\EE
The latter implies the former if at most one of the restriction homomorphisms
\BE{SympSumCond_e1b}H^1(X_{ij};\Z)\lra \bigoplus_{k\in[N]-\{i,j\}}\hspace{-.22in}H^1(X_{ijk};\Z),
\qquad i,j\in[N],~i\neq j,\EE
is not surjective, but not in general; see Example~\ref{NonSurj_eg}.
In the most basic case of the $N\!=\!3$ situation of Theorem~\ref{SympSum_thm}
with $X_{ij}$ and~$X_{ik}$ being symplectic surfaces in a 4-dimensional manifold~$X_i$
intersecting transversely and positively at a single point,
the conditions~\eref{SympSumCond_e} and~\eref{SympSumCond_e1} reduce to the simple
condition on the self-intersection numbers of these surfaces stated in
\cite[Theorem~2.7]{Symington3}.\\

\noindent
The algebro-geometric analogue of~\eref{SympSumCond_e},
$$\cO_{X_{\prt}}(X_{\eset})\approx\cO_{X_{\prt}}\,,$$
is called \sf{$d$-semistability} in \cite[Definition~(1.13)]{F}.
It is well-known to be an obstruction to the existence of a one-parameter family 
of smoothings of~$X_{\eset}$ in the algebraic geometry category; 
see \cite[Corollary~(1.12)]{F}.
As shown in~\cite{PP}, it is not the only obstruction  in the algebraic category, 
even in the $N\!=\!2$ case.
The algebro-geometric analogue of~\eref{SympSumCond_e1},
$$\cN_{X_i}X_{ij}\otimes \cN_{X_j}X_{ij} \otimes
\bigotimes_{k\in[N]-\{i,j\}}\hspace{-.22in}\cO_{X_{ij}}(X_{ijk})
\approx \cO_{X_{ij}} \qquad\forall~i,j\in[N],~i\neq j\,,$$
is known as \sf{the triple point condition}; see \cite[Proposition~2.4.3]{P}.\\

\noindent
As in the $N\!=\!2$ case of Theorem~\ref{SympSum_thm} addressed in~\cite{Gf},
the construction of \hbox{$\pi\!:(\cZ,\om_{\cZ})\!\lra\!\De$} involves some auxiliary data
for~$\X$ and a compatible choice of a trivialization of the complex line bundle~\eref{PsiDfn_e}
in a given homotopy class. 
We call the former regularizations and recall their definition in 
Sections~\ref{SCDregul_subs} and~\ref{SCCregul_subs}.
Proposition~\ref{SCC_prp}, proved in Section~\ref{SCCprpPf_subs}, ensures that 
each homotopy class of trivializations of~\eref{PsiDfn_e} contains a representative
compatible with a given regularization for~$\X$.
The main part of the proof of Theorem~\ref{SympSum_thm} is carried out in
Section~\ref{SympSumPf_sec}, where the chosen auxiliary data for~$\X$
and a compatible trivialization of~\eref{PsiDfn_e} are used to construct a one-parameter family
\hbox{$\pi\!:\cZ\!\lra\!\De$} of smoothings of $(X_{\eset},(\om_i')_{i\in[N]})$.
By Proposition~\ref{InducedIsom_prp} proved in Section~\ref{SympSumConv_subs},
every one-parameter family $\pi\!:\cZ\!\lra\!\De$ 
of smoothings of $(X_{\eset},(\om_i')_{i\in[N]})$ 
determines a homotopy class of trivializations of~\eref{PsiDfn_e}.
By Proposition~\ref{PhiExt_prp} proved in Section~\ref{TrivExt_subs},
the homotopy class determined by the family constructed in Section~\ref{SympSumPf_sec}   
is the homotopy class used to construct~it.

\subsection{Examples}
\label{SCCeg_subs}

\noindent
We now give two examples.
The first one describes a 3-fold case of Theorem~\ref{SympSum_thm}.
The second example shows that condition~\eref{SympSumCond_e1} is in general weaker
than condition~\eref{SympSumCond_e}.

\begin{eg}\label{P2cut_eg}
Let $\wh\P^2$ be the blowup of $\P^2$ at a point~$p$,
$E,\bar{L}\!\subset\!\wh\P^2$ be the exceptional divisor and
the proper transform of a line through~$p$, respectively, and $P\!=\!E\!\cap\!\bar{L}$.
We take $X_1,X_2,X_3\!=\!\wh\P^2$ and identify 
$E\!\subset\!X_1$ with $\bar{L}\!\subset\!X_2$, 
$E\!\subset\!X_2$ with $\bar{L}\!\subset\!X_3$, 
and $E\!\subset\!X_3$ with $\bar{L}\!\subset\!X_1$;
see Figure~\ref{P2cut_fig}.
By adjusting the size of the blowup as in \cite[Section~7.1]{MS1}, 
we can ensure that the identifications can be made symplectically.
Since 
$$\blr{c_1(\cN_{\wh\P^2}E),E}=-1 \qquad\hbox{and}\qquad
\blr{c_1(\cN_{\wh\P^2}\bar{L}),\bar{L}}=0,$$
the resulting 3-fold configuration $((X_I)_{I\in\cP^*(3)},(\om_i)_{i\in[3]})$ 
satisfies~\eref{SympSumCond_e1}
for all $i,j\!\in\![3]$ distinct.
Since all three homomorphisms~\eref{SympSumCond_e1b} are surjective in this case, 
this configuration thus satisfies~\eref{SympSumCond_e}.
The singular locus~$X_{\prt}$ of the NC symplectic variety~$X_{\eset}$ 
 to be smoothed out consists of 3 copies of~$\P^1$ 
with one point in common.
Since $X_{\prt}$ is simply connected, there is only one homotopy class
of trivializations of~\eref{PsiDfn_e}.
The symplectic deformation equivalence class (of a smooth fiber) of
the corresponding 3-fold symplectic sum is~$\P^2$.  
This is illustrated in the second diagram in Figure~\ref{P2cut_fig}
from the symplectic cut perspective of~\cite{SympCutMulti}
applied in a toric setting (the big triangle corresponds to~$\P^2$).
\end{eg}

\begin{figure}
\begin{pspicture}(0,-2)(11,2)
\psset{unit=.3cm}
\psline[linewidth=.1](15,-2)(22,-2)\psline[linewidth=.1](15,-2)(15,5)
\psline[linewidth=.1](15,-2)(10.5,-6.5)\pscircle*(15,-2){.3}
\rput(19.5,2.5){\sm{$\wh\P^2$}}\rput(11.5,-1){\sm{$\wh\P^2$}}\rput(17,-5){\sm{$\wh\P^2$}}
\rput(21.5,-2.9){\sm{$E$}}\rput(21.5,-1.1){\sm{$\bar{L}$}}
\rput(14.2,4.5){\sm{$\bar{L}$}}\rput(15.7,4.5){\sm{$E$}}
\rput(10,-5.7){\sm{$E$}}\rput(12,-6.1){\sm{$\bar{L}$}}
\rput(15.8,-1.2){\sm{$P$}}
\psline[linewidth=.1](35,-6)(45,-6)\psline[linewidth=.1](35,-6)(35,4)
\psline[linewidth=.1](45,-6)(35,4)
\psline[linewidth=.1](35,-2)(38,-2)\psline[linewidth=.1](38,1)(38,-2)
\psline[linewidth=.1](42,-6)(38,-2)\pscircle*(38,-2){.3}
\rput(40.6,-3.4){\sm{$E$}}\rput(39.3,-4.4){\sm{$\bar{L}$}}
\rput(37.2,-.7){\sm{$E$}}\rput(38.5,-.7){\sm{$\bar{L}$}}
\rput(35.8,-2.8){\sm{$E$}}\rput(35.8,-1.2){\sm{$\bar{L}$}}
\end{pspicture}
\caption{The NC variety of Example~\ref{P2cut_eg} and 
a toric representation of the corresponding symplectic sum.}
\label{P2cut_fig}
\end{figure}
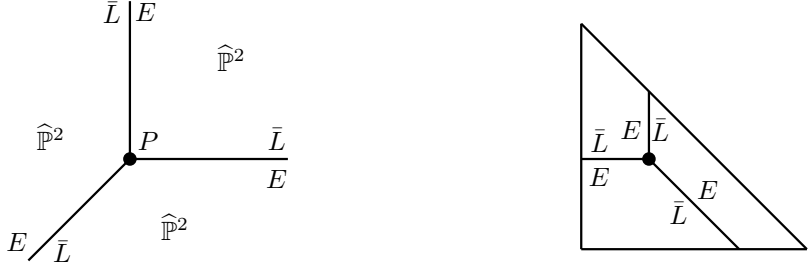

\begin{eg}\label{NonSurj_eg}
Let $X_1\!=\!\P^3$ with its standard symplectic form, 
$X_{12}\!\approx\!\P^2$ be a linear subspace, and 
$X_{13}\!\subset\!\P^3$ be a cubic surface transverse to~$X_{12}$.
The intersection~$X_{123}$ of~$X_{12}$ and~$X_{13}$ is then
a plane cubic, i.e.~a genus~1 curve.
For $i\!=\!2,3$, define 
\begin{gather*}
\cL_{1i}=\cN_{X_1}X_{1i}\otimes\cO_{X_{1i}}(X_{123})\approx\cO_{\P^3}(4)\big|_{X_{1i}}\lra X_{1i}\,,\\
X_i=\P\big(\cL_{1i}\!\oplus\!\cO_{X_{1i}}\big), \qquad
X_{1i}^{\i}=\P\big(\cL_{1i}\!\oplus\!0\big)\approx X_{1i}\,;
\end{gather*}
see the left diagram in Figure~\ref{NonSurj_fig}.
There are canonical isomorphisms
\BE{cL1213_e}\cL_{12}|_{X_{123}}\approx\big(\cN_{X_1}X_{12}\!\otimes\!\cN_{X_1}X_{13}\big)|_{X_{123}}
\approx\big(\cN_{X_1}X_{13}\!\otimes\!\cN_{X_1}X_{12}\big)|_{X_{123}}
\approx\cL_{13}|_{X_{123}}\,.\EE
The homotopy classes of all isomorphisms $\cL_{12}|_{X_{123}}\!\approx\!\cL_{13}|_{X_{123}}$
in the category of complex (not holomorphic) line bundles correspond to the homotopy classes
of continuous functions $X_{123}\!\lra\!S^1$, i.e.~to the elements~of
$$H^1(X_{123};\Z)\approx\Z^2\,.$$
Any such isomorphism~$\rho$ induces an identification 
$$\psi_{23}^{\rho}\!: \P(\cL_{12}\!\oplus\!\C)\big|_{X_{123}}\lra 
\P(\cL_{13}\!\oplus\!\C)\big|_{X_{123}}\,,$$
which can be assumed to be holomorphic by pushing the holomorphic structure forward;
the zero element of $H^1(X_{123};\Z)$ corresponds to the identification induced
by~\eref{cL1213_e}.
The SC variety
$$X_{\eset}^{\rho} = X_1\!\cup\!X_2\!\cup\!X_3,\qquad
X_{23}\equiv
\P(\cL_{12}\!\oplus\!\cO_{X_{12}})\big|_{X_{123}}\!\sim_{\psi_{23}^\rho}\!\!
\P(\cL_{13}\!\oplus\!\cO_{X_{13}})\big|_{X_{123}}\,,$$
is then K\"ahler and satisfies the vanishing condition in~\eref{SympSumCond_e1} 
over~$X_{12}$ and~$X_{13}$.
If $\rho$ corresponds to a nonzero element of $H^1(X_{123};\Z)$, 
then~\eref{SympSumCond_e} is not satisfied over $X_{12}\!\cup\!X_{13}$ 
because the connecting homomorphism~$\de$ in the exact sequence
$$H^1(X_{12};\Z)\oplus H^1(X_{13};\Z)\lra H^1(X_{123};\Z)
\stackrel\de\lra H^2\big(X_{12}\!\cup\!X_{13};\Z\big)$$
is injective ($X_{12}$ and $X_{13}$ are simply connected).
\begin{figure}
\begin{pspicture}(-.5,-2)(11,2)
\psset{unit=.3cm}
\psline[linewidth=.1](15,-2)(22,-2)\psline[linewidth=.1](15,-2)(15,5)
\psline[linewidth=.1](15,-2)(10.7,-6.3)\pscircle*(15,-2){.3}
\rput(19.5,2.5){\sm{$\P^3$}}\rput(12.5,2){\sm{$X_2$}}\rput(19,-4.5){\sm{$X_3$}}
\rput(22.2,-1.2){\sm{$X_{13}$}}\rput(22.2,-3){\sm{$X_{13}^{\i}$}}
\rput(16,4.8){\sm{$\P^2$}}\rput(13.8,4.8){\sm{$X_{12}^{\i}$}}
\rput{45}(12.1,-3.9){\sm{$X_{23}$}}\rput{45}(13.2,-5){\sm{$X_{23}$}}
\rput(16.4,-1.2){\sm{$X_{123}$}}\rput(16,-3){\sm{$X_{123}^{\i}$}}
\psline[linewidth=.05](10.7,-6.3)(17.7,-6.3)\psline[linewidth=.05](10.7,-6.3)(10.7,.7)
\pscircle*(10.7,-6.3){.2}\rput(10,-7){\sm{$X_{123}^0$}}
\psline[linewidth=.1](35,-2)(42,-2)\psline[linewidth=.1](35,-2)(35,5)
\psline[linewidth=.1](35,-2)(30.7,-6.3)\pscircle*(35,-2){.3}
\rput(39.5,2.5){\sm{$\P^3$}}\rput(31.5,2){\sm{$\wh{X}_2$}}\rput(39,-4.5){\sm{$X_3$}}
\rput(42.2,-1.2){\sm{$X_{13}$}}\rput(42.2,-3){\sm{$X_{13}^{\i}$}}
\rput(36,4.8){\sm{$\P^2$}}\rput(33.8,4.8){\sm{$X_{12}^{\i}$}}
\rput{45}(32.1,-3.9){\sm{$X_{23}$}}\rput{45}(33.2,-5){\sm{$X_{23}$}}
\rput(36.4,-1.2){\sm{$X_{123}$}}\rput(36,-3){\sm{$X_{123}^{\i}$}}
\psline[linewidth=.05](30.7,-6.3)(37.7,-6.3)\psline[linewidth=.05](28.7,-4.3)(28.7,2.7)
\psline[linewidth=.05](30.7,-6.3)(28.7,-4.3)
\pscircle*(30.7,-6.3){.2}\rput(30,-7){\sm{$X_{123}^0$}}
\end{pspicture}
\caption{The two NC varieties in Example~\ref{NonSurj_eg}.}
\label{NonSurj_fig}
\end{figure}
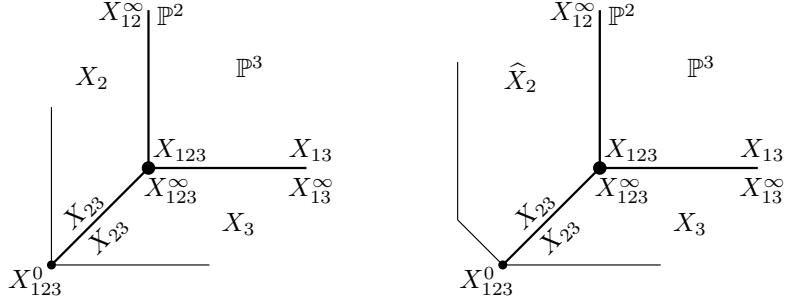
Let 
$$X_{123}^0=\P\big(0\!\oplus\!\cO_{X_{123}}\big)\,, ~
X_{123}^{\i}=\P(\cL_{12}\!\oplus\!0)|_{X_{123}}\subset 
\P(\cL_{12}\!\oplus\!\cO_{X_{12}})|_{X_{123}}\,,$$
and $\pi\!:X_{23}\!\lra\!X_{123}$ be the projection map.
Using 
$$\cN_{X_i}X_{23}=\pi^*\cN_{X_{1i}}X_{123}
=\pi^*\big(\cO_{\P^3}(7\!-\!2i)|_{X_{123}}\big)
\qquad\hbox{for}~~ i=2,3,$$
we find that 
$$\cN_{X_2}X_{23}\otimes \cN_{X_3}X_{23}\otimes \cO_{X_{23}}(X_{123})
=\pi^*\big(\cO_{\P^3}(4)|_{X_{123}}\big)\otimes \cO_{X_{23}}(X_{123}^{\i})
\neq \cO_{X_{23}}\,.$$
In order to achieve \eref{SympSumCond_e1} over all smooth strata of~$X_{\prt}$, 
we replace~$X_2$ by its blowup~$\wh{X}_2$ along~$X_{123}^0$;
see the right diagram in Figure~\ref{NonSurj_fig}.
The proper transform of~$X_{23}$ is still~$X_{23}$, but its normal bundle in $\wh{X}_2$~is
$$\cN_{X_2}X_{23}\otimes \cO_{X_{23}}(-X_{123}^0)
=\cN_{X_2}X_{23}\otimes \cO_{X_{23}}(-X_{123}^{\i}) 
\otimes \pi^*\big(\cO_{\P^3}(-4)|_{X_{123}}\big)\,.$$
Thus, the modified 3-fold SC symplectic configuration satisfies
the vanishing condition in~\eref{SympSumCond_e1} over all smooth strata of~$X_{\prt}$.
\end{eg}

\section{Regularizations}
\label{SCD_sec}

\noindent
In \cite[Sections~2.2,2,3]{SympDivConf}, we introduced
notions of symplectic regularizations for an SC divisor
$\{V_i\}_{i\in S}$ in~$X$ and an SC configuration~$\X$;
we recall them in Sections~\ref{SCDregul_subs} and~\ref{SCCregul_subs}.
Such regularizations provide essential auxiliary data for
the symplectic smoothing/sum construction of Theorem~\ref{SympSum_thm},
just as they did in the $N\!=\!2$ case addressed in \cite{Gf,MW}.
By \cite[Theorem~2.17]{SympDivConf},
the space  $\Symp^+(\X)$ defined in Section~\ref{SCdfn_subs}
is weakly homotopy equivalent to the space $\Aux(\X)$ of pairs consisting
 of an element of $\Symp^+(\X)$  and a compatible regularization.
Proposition~\ref{SCC_prp}, which is established in Section~\ref{SCCprpPf_subs},
adjusts this weak homotopy equivalence property to incorporate 
trivializations of~\eref{PsiDfn_e}; 
see also \eref{SCCproj_e} and Remark~\ref{SCC_rmk}.

\subsection{SC divisors}
\label{SCDregul_subs}

\noindent
If $B$ is a manifold, possibly with boundary, and $k\!\in\!\Z^{\ge0}$,
we call a family $(\om_t)_{t\in B}$ of $k$-forms on~$X$ \sf{smooth} 
if the $k$-form~$\wt\om$ on~$B\!\times\!X$ given~by
$$\wt\om_{(t,x)}(v_1,\ldots,v_k)=
\begin{cases}
\om_t|_x(v_1,\ldots,v_k),&\hbox{if}~v_1,\ldots\!v_k\!\in\!T_xX;\\
0,&\hbox{if}~v_1\!\in\!T_tB;
\end{cases}$$
is smooth. 
Smoothness for families of other objects is defined similarly.\\

\noindent
We  call $\pi\!:(L,\rho,\na)\!\lra\!V$ a \sf{Hermitian line bundle} if
$V$ is a manifold, $L\!\lra\!V$ is a smooth complex line bundle,
$\rho$ is a Hermitian metric on~$L$, 
and $\na$ is a $\rho$-compatible connection on~$L$.
We use the same notation~$\rho$ to denote the square of the norm function on~$L$
and the Hermitian form on~$L$ which is $\C$-antilinear in the second input.
Thus,
$$\rho(v)\equiv\rho(v,v), \quad 
\rho(\fI v,w)=\fI\rho(v,w)=-\rho(v,\fI w) 
\qquad\forall~(v,w)\!\in\!L\!\times_V\!L.$$
Let $\rho^{\R}$ denote the real part of~$\rho$.
Each triple $(L,\rho,\na)$ as above induces a \sf{connection 1-form}~$\al_{\rho,\na}$
on the principal $S^1$-bundle $SL$ of $\rho$-unit vectors; see Appendix~\ref{conn_app}.
Via the canonical retraction $L\!-\!V\!\lra\!SL$, $\al_{\rho,\na}$ extends 
to a 1-form on $L\!-\!V$. 
A smooth map $h\!:V'\!\lra\!V$ pulls back a Hermitian line bundle $(L,\rho,\na)$
over~$V$ to a Hermitian line bundle
$$h^*(L,\rho,\na)\equiv (h^*L,h^*\rho,h^*\na)\lra V'.$$

\vspace{.1in}

\noindent
A Riemannian metric on an oriented  real vector bundle \hbox{$L\!\lra\!V$} of rank~2
determines a complex structure on the fibers of~$L$.
A \sf{Hermitian structure} on an oriented  real vector bundle \hbox{$L\!\lra\!V$} of rank~2
is a pair $(\rho,\na)$ such that $(L,\rho,\na)$ is a Hermitian line bundle
with the complex structure~$\fI_{\rho}$ determined by the Riemannian metric~$\rho^{\R}$.
If $\Om$ is a fiberwise symplectic form on an oriented vector bundle \hbox{$L\!\lra\!V$} of rank~2,
an \sf{$\Om$-compatible Hermitian structure} on~$L$ is a Hermitian structure $(\rho,\na)$ on~$L$ 
such that
$\Om(\cdot,\fI_{\rho}\cdot)=\rho^{\R}(\cdot,\cdot)$.\\

\noindent
If $(L_i,\rho_i,\na^{(i)})_{i\in I}$ is a finite collection of Hermitian line
bundles over a symplectic manifold $(V,\om)$,
$$\pi\!: \cN\equiv\bigoplus_{i\in I}L_i\lra V\,,$$
and $\pr_{I;I-i}\!:\cN\!\lra\!L_i$ is the component projection map for each $i\!\in\!I$,
then 
\BE{ombund_e}\wh\om_{(\rho_i,\na^{(i)})_{i\in I}}^{\bu}\equiv 
\pi^*\om+\frac12
\sum_{i\in I} \pr_{I;I-i}^*\nd\big(\rho_i\al_{\rho_i,\na^{(i)}}\big)\EE
is a well-defined closed 2-form on the total space of $\cN$;
it is nondegenerate on a neighborhood of~$V$ in~$\cN$.
By~\eref{Omvsaldfn_e}, this definition agrees with \cite[(2.10)]{SympDivConf}
whenever $(\rho_i,\na^{(i)})$ is an $\Om_i$-compatible Hermitian structure on~$L_i$\\

\noindent
If $\Psi\!:V'\!\lra\!V$ is an embedding, $I'\!\subset\!I$,
$(L_i,\rho_i,\na^{(i)})_{i\in I}$ is a finite collection of Hermitian line
bundles over~$V$, and 
$(L_i',\rho_i',\na'^{(i)})_{i\in I'}$ is a finite collection of Hermitian line
bundles over~$V'$, a vector bundle homomorphism 
$$\wt\Psi\!: \bigoplus_{i\in I'}L'_i\lra \bigoplus_{i\in I}L_i$$
covering~$\Psi$ is a \sf{product Hermitian inclusion} if 
$$\wt\Psi\!: (L_i',\rho_i',\na'^{(i)}) \lra 
\Psi^*(L_i,\rho_i,\na^{(i)})$$
is an isomorphism of Hermitian line bundles over~$V'$ for every $i\!\in\!I'$.
We  call such a morphism a \sf{product Hermitian isomorphism covering~$\Psi$}
if $|I'|\!=\!|I|$.

\begin{dfn}\label{smreg_dfn}
Let $X$ be a manifold and $V\!\subset\!X$ be a  submanifold
with normal bundle $\cN_XV\!\lra\!V$. 
A \sf{regularization for~$V$ in~$X$} is a diffeomorphism $\Psi\!:\cN'\!\lra\!X$
from a neighborhood of~$V$ in~$\cN_XV$ onto a neighborhood of~$V$ in~$X$ such
that $\Psi(x)\!=\!x$ and the isomorphism
$$ \cN_XV|_x=T_x^{\ver}\cN_XV \lhra T_x\cN_XV
\stackrel{\nd_x\Psi}{\lra} T_xX\lra \frac{T_xX}{T_xV}\equiv\cN_XV|_x$$
is the identity for every $x\!\in\!V$.
\end{dfn}

\noindent
If $(X,\om)$ is a symplectic manifold and $V$ is a  symplectic submanifold in~$(X,\om)$,
then $\om$ induces a fiberwise symplectic form~$\om|_{\cN_XV}$ on
the normal bundle~$\cN_XV$ of~$V$ in~$X$ via the isomorphism~\eref{cNXVsymp_e}.
We denote the restriction of~$\om|_{\cN_XV}$ to a subbundle $L\!\subset\!\cN_XV$
by~$\om|_L$.

\begin{dfn}\label{sympreg1_dfn}
Let $X$ be a  manifold, $V\!\subset\!X$ be a  submanifold, and
$$\cN_XV=\bigoplus_{i\in I}L_i$$
be a fixed splitting into oriented rank~2 subbundles. 
\begin{enumerate}[label=(\arabic*),leftmargin=*]

\item\label{sympreg1_it} 
If $\om$ is a symplectic form on~$X$ such that $V$ is a symplectic submanifold
and $\om|_{L_i}$ is nondegenerate for every $i\!\in\!I$, then
an \sf{$\om$-regularization for~$V$ in~$X$} is a tuple $((\rho_i,\na^{(i)})_{i\in I},\Psi)$, 
where $(\rho_i,\na^{(i)})$ is an $\om|_{L_i}$-compatible Hermitian structure on~$L_i$
for each $i\!\in\!I$ and $\Psi$ is a regularization for~$V$ in~$X$, such that 
$$\Psi^*\om=\wh\om_{(\rho_i,\na^{(i)})_{i\in I}}^{\bu}\big|_{\Dom(\Psi)}.$$

\item\label{sympreg2_it} If $B$ is a  manifold, possibly with boundary, and
$(\om_t)_{t\in B}$ is a smooth family of symplectic forms on~$X$ 
which restrict to symplectic forms on~$V$, then
an \sf{$(\om_t)_{t\in B}$-family of regularizations} for~$V$ in~$X$ 
is a smooth family of tuples 
\BE{sympreg1dfn_e}(\cR_t)_{t\in B} \equiv 
\big((\rho_{t;i},\na^{(t;i)})_{i\in I},\Psi_t\big)_{t\in B}\EE
such that $\cR_t$ is an $\om_t$-regularization for~$V$ in~$X$ for each $t\!\in\!B$
and 
$$\big\{(t,v)\!\in\!B\!\times\!\cN_XV\!:\,v\!\in\!\Dom(\Psi_t)\big\}\lra X,\qquad
(t,v)\lra \Psi_t(v),$$
is a smooth map from a neighborhood of $B\!\times\!V$ in $B\!\times\!\cN_XV$.
\end{enumerate}
\end{dfn}

\noindent
Suppose $\{V_i\}_{i\in S}$ is a transverse collection of codimension~2 submanifolds of~$X$.
For each $I\!\subset\!S$, the last isomorphism in~\eref{cNorient_e2} with $I'\!=\!\eset$
provides 
a natural decomposition 
$$\pi_I\!:\cN_XV_I\!=\!\bigoplus_{i\in I}\cN_{V_{I-i}}V_I  \lra V_I$$
of the normal bundle of~$V_I$ in~$X$ into oriented rank~2 subbundles. 
We take this decomposition as given for the purposes of applying Definition~\ref{sympreg1_dfn}.
If in addition $I'\!\subset\!I$, let
$$\pi_{I;I'}\!:\cN_{I;I'}\equiv 
\bigoplus_{i\in I-I'}\!\!\!\cN_{V_{I-i}}V_I=\cN_{V_{I'}}V_I\lra V_I\,.$$
There are canonical identifications
\BE{cNtot_e}\cN_{I;I-I'}=\cN_XV_{I'}|_{V_I}, \quad
\cN_XV_I=\pi_{I;I'}^*\cN_{I;I-I'}=\pi_{I;I'}^*\cN_XV_{I'}
\qquad\forall~I'\!\subset\!I\!\subset\![N].\EE
The first equality in the second statement above
is used in particular in~\eref{overlap_e}.

\begin{dfn}\label{TransCollReg_dfn}
Let $X$ be a manifold and $\{V_i\}_{i\in S}$ be a transverse collection 
of submanifolds of~$X$.
A \sf{system of regularizations for}  $\{V_i\}_{i\in S}$ in~$X$ is a~tuple 
$(\Psi_I)_{I\subset S}$, where $\Psi_I$ is a regularization for~$V_I$ in~$X$
in the sense of Definition~\ref{smreg_dfn}, such~that
\BE{Psikk_e}
\Psi_I\big(\cN_{I;I'}\!\cap\!\Dom(\Psi_I)\big)=V_{I'}\!\cap\!\Im(\Psi_I)\EE
for all $I'\!\subset\!I\!\subset\!S$.
\end{dfn}

\noindent
Given a system of regularizations as in Definition~\ref{TransCollReg_dfn}
and $I'\!\subset\!I\!\subset\!S$, let
$$\cN_{I;I'}' = \cN_{I;I'}\!\cap\!\Dom(\Psi_I), \qquad
\Psi_{I;I'}\equiv \Psi_I\big|_{\cN_{I;I'}'}\!: \cN_{I;I'}'\lra V_{I'}\,.$$
The map $\Psi_{I;I'}$ is a regularization for $V_I$ in~$V_{I'}$.
As explained in \cite[Section~2.2]{SympDivConf}, $\Psi_I$ determines
an isomorphism
\BE{wtPsiIIdfn_e}
  \fD\Psi_{I;I'}\!:  \pi_{I;I'}^*\cN_{I;I-I'}\big|_{\cN_{I;I'}'}
\lra \cN_XV_{I'}\big|_{V_{I'}\cap\Im(\Psi_I)}\EE
of vector bundles covering~$\Psi_{I;I'}$ and
respecting the natural decompositions of 
$\cN_{I;I-I'}\!=\!\cN_XV_{I'}|_{V_I}$ and $\cN_XV_{I'}$.
By the last assumption in Definition~\ref{smreg_dfn}, 
\BE{wtPsiIIprop_e}\fD\Psi_{I;I'}\big|_{\pi_{I;I'}^*\cN_{I;I-I'}|_{V_I}}\!=\!\id\!:\,
\cN_{I;I-I'}\lra \cN_XV_{I'}|_{V_I}\EE
under the canonical identification of $\cN_{I;I-I'}$ with $\cN_XV_{I'}|_{V_I}$.

\begin{dfn}\label{TransCollregul_dfn}
Let $X$ be a manifold and  $\{V_i\}_{i\in S}$ be a transverse 
collection of submanifolds of~$X$. 
A \sf{regularization for $\{V_i\}_{i\in S}$ in~$X$} 
is a system of regularizations $(\Psi_I)_{I\subset S}$ 
for $\{V_i\}_{i\in S}$ in~$X$ such~that
\BE{overlap_e}
\Dom(\Psi_I)=\fD\Psi_{I;I'}^{\,-1}\big(\Dom(\Psi_{I'})\big), \quad
\Psi_I=\Psi_{I'}\circ\fD\Psi_{I;I'}|_{\Dom(\Psi_I)}\EE
for all $I'\!\subset\!I\!\subset\!S$.
\end{dfn}

\noindent
If $(\Psi_I)_{I\subset S}$ is a regularization for $\{V_i\}_{i\in S}$ in~$X$, 
then \eref{wtPsiIIprop_e} and~\eref{overlap_e} imply that 
\BE{SCDcons_e}\begin{split}
\cN_{I\cup J;I'\cup J}'=\cN_{I;I'}'\big|_{V_{I\cup J}}\,,
&\qquad \Psi_{I\cup J;I'\cup J}=\Psi_{I;I'}\big|_{\cN_{I\cup J;I'\cup J}'}\,,\\
\fD\Psi_{I\cup J;I'\cup J}
\big|_{\pi_{I\cup J;I'\cup J}^*\cN_{I\cup J;(I-I')\cup J}|_{\cN_{I\cup J;I'\cup J}'}}
&=\fD\Psi_{I;I'}\big|_{\pi_{I;I'}^*\cN_{I;I-I'}|_{\cN_{I\cup J;I'\cup J}'}}
\end{split}\EE
for all $I'\!\subset\!I\!\subset\!S$ and $J\!\subset\!S\!-\!I$.
Furthermore,
\BE{SCDcons_e2}
\Psi_{I;I''}=\Psi_{I';I''}\circ\fD\Psi_{I;I'}\big|_{\cN_{I;I''}'}\,, 
\qquad
\fD\Psi_{I;I''}=\fD\Psi_{I';I''}\circ
\fD\Psi_{I;I'}\big|_{\pi_{I;I''}^*\cN_{I;I-I''}|_{\cN_{I;I''}'}}\EE
for all $I''\!\subset\!I'\!\subset\!I\!\subset\!S$.

\begin{dfn}\label{SCDregul_dfn}
Let $X$ be a manifold and  $\{V_i\}_{i\in S}$ be a finite transverse collection 
of closed submanifolds of~$X$ of codimension~2. 

\begin{enumerate}[label=(\arabic*),leftmargin=*]

\item\label{sympregul_it} 
If $\om\in\Symp^+(X,\{V_i\}_{i\in S})$, then 
an \sf{$\om$-regularization for $\{V_i\}_{i\in S}$ in~$X$} 
is a~tuple
\BE{regdfn_e0}
(\cR_I)_{I\subset S} \equiv  
\big((\rho_{I;i},\na^{(I;i)})_{i\in I},\Psi_I\big)_{I\subset S}\EE
such that $\cR_I$ is an $\om$-regularization for~$V_I$ in~$X$ for each $I\!\subset\!S$,
$(\Psi_I)_{I\subset S}$ is a regularization for $\{V_i\}_{i\in S}$ in~$X$,
and the induced vector bundle isomorphisms~\eref{wtPsiIIdfn_e}
are product Hermitian isomorphisms for all $I'\!\subset\!I\!\subset\!S$.

\item\label{sympregul_it2}  If $B$ is a manifold, possibly with boundary, and
$(\om_t)_{t\in B}$ is a smooth family of symplectic forms in $\Symp^+(X,\{V_i\}_{i\in S})$,
then an \sf{$(\om_t)_{t\in B}$-family of regularizations for $\{V_i\}_{i\in S}$ in~$X$}
is a smooth family of~tuples 
\BE{regdfn_e3}(\cR_{t;I})_{t\in B,I\subset S}
\equiv 
\big((\rho_{t;I;i},\na^{(t;I;i)})_{i\in I},\Psi_{t;I}\big)_{t\in B,I\subset S}\EE
such that $(\cR_{t;I})_{I\subset S}$ is an $\om_t$-regularization for $\{V_i\}_{i\in S}$ in~$X$ 
for each $t\!\in\!B$
and $(\cR_{t;I})_{t\in B}$ is
an $(\om_t)_{t\in B}$-family of regularizations for $V_I$ in~$X$ 
for each $I\!\subset\!S$.
\end{enumerate}
\end{dfn}

\subsection{SC varieties}
\label{SCCregul_subs}

\noindent
This section is the analogue of Section~\ref{SCDregul_subs} for SC symplectic configurations,
especially those satisfying the topological condition~\eref{SympSumCond_e}.
Definition~\ref{SCCregul_dfn}\ref{SCCreg2_it} topologizes the set
$\wt\Aux_{\hb}(\X)$ of triples $((\om_i)_{i\in[N]},\fR,\Phi)$ 
consisting of a symplectic structure $(\om_i)_{i\in[N]}$ on
a transverse configuration~$\X$,
an $(\om_i)_{i\in[N]}$-regularization~$\fR$ for~$\X$, and 
a compatible trivialization~$\Phi$ of~\eref{PsiDfn_e} 
in a homotopy class~$\hb$.
By Proposition~\ref{SCC_prp} at the end of this section,
the  projection 
\BE{SCCproj_e}\wt\Aux_{\hb}(\X)\lra \Symp^+_0(\X)\subset \Symp^+(\X), \quad
\big((\om_i)_{i\in[N]},\fR,\Phi\big)\lra(\om_i)_{i\in[N]},\EE
to the space of SC symplectic configurations satisfying~\eref{SympSumCond_e} 
induces isomorphisms on~$\pi_k$ for all $k\!\in\!\Z^{\ge0}\!-\!\{1\}$.\\

\noindent
Suppose $\{X_I\}_{I\in\cP^*(N)}$ is a transverse configuration in the sense 
of Definition~\ref{TransConf_dfn1}.
For each $I\!\in\!\cP^*(N)$ with $|I|\!\ge\!2$, let
$$\pi_I\!:\cN X_I\equiv \bigoplus_{i\in I}\cN_{X_{I-i}}X_I\lra X_I\,.$$
If in addition $I'\!\subset\!I$, let
$$\pi_{I;I'}\!:\cN_{I;I'}\equiv 
  \bigoplus_{i\in I-I'}\!\!\!\cN_{X_{I-i}}X_I\lra X_I\,.$$
By the last isomorphism in~\eref{cNorient_e2} with $X\!=\!X_i$ for any $i\!\in\!I'$ and 
$\{V_j\}_{j\in S}\!=\!\{X_{ij}\}_{j\in[N]-i}$, 
$$	\cN_{I;I'}=\cN_{X_{I'}}X_I \qquad\forall~I'\!\subset\!I\!\subset\![N],~I'\!\neq\!\eset.$$
Similarly to~\eref{cNtot_e}, there are canonical identifications
\BE{cNXIide_e}\cN_{I;I-I'}=\cN X_{I'}|_{X_I}, \quad
\cN X_I=\pi_{I;I'}^*\cN_{I;I-I'}=\pi_{I;I'}^*\cN X_{I'}
\qquad\forall~I'\!\subset\!I\!\subset\![N];\EE
the first and last identities above hold if $|I'|\!\ge\!2$.

\begin{dfn}\label{TransConfregul_dfn}
Let $N\!\in\!\Z^+$ and $\X\!=\!\{X_I\}_{I\in\cP^*(N)}$ be a transverse configuration.
A \sf{regularization for $\X$} is a tuple 
$(\Psi_{I;i})_{i\in I\subset[N]}$, 
where for each $i\!\in\!I$ fixed the tuple $(\Psi_{I;i})_{i\in I\subset[N]}$
is a regularization for $\{X_{ij}\}_{j\in[N]-i}$ in~$X_i$ in the sense of
Definition~\ref{TransCollregul_dfn}, such that
\BE{SCCregCond_e0}
\Psi_{I;i_1}\big|_{\cN_{I;i_1i_2}\cap\Dom(\Psi_{I;i_1})}
=\Psi_{I;i_2}\big|_{\cN_{I;i_1i_2}\cap\Dom(\Psi_{I;i_2})}\EE
for all $i_1,i_2\!\in\!I\!\subset\![N]$.
\end{dfn}

\noindent
Given a regularization as in Definition~\ref{TransConfregul_dfn}
and $I'\!\subset\!I\!\subset\![N]$ with $|I|\!\ge\!2$ and $I'\!\neq\!\eset$, let
\BE{cNconfdfn_e}
\cN_{I;I'}'\!=\!\cN_{I;I'}\!\cap\!\Dom(\Psi_{I;i}),~~
\Psi_{I;I'}\!=\!\Psi_{I;i}|_{\cN_{I;I'}'}\!\!:\cN_{I;I'}'\lra X_{I'}
\qquad\hbox{if}~i\!\in\!I';\EE
by~\eref{SCCregCond_e0}, $\Psi_{I;I'}(v)$ does not depend on the choice of $i\!\in\!I'$.
Let
\BE{fDPsiIIconf_e0}\fD\Psi_{I;i;I'}\!: 
\pi_{I;I'}^*\cN_{I;i\cup(I-I')}\big|_{\cN_{I;I'}'}\lra \cN_{I';i}\big|_{\Im(\Psi_{I;I'})}\EE
be the associated vector bundle isomorphism as in~\eref{wtPsiIIdfn_e}.
If $|I'|\!\ge\!2$, we define an isomorphism of split vector bundles
\BE{fDPsiIIconf_e}\begin{split}
\fD\Psi_{I;I'}\!:\pi_{I;I'}^*\cN_{I;I-I'}\big|_{\cN_{I;I'}'}
&\lra \cN X_{I'}\big|_{\Im(\Psi_{I;I'})}\,, \\
\fD\Psi_{I;I'}\big|_{\pi_{I;I'}^*\cN_{I;i\cup(I-I')}\big|_{\cN_{I;I'}'}}
\!&=\fD\Psi_{I;i;I'} \quad\forall~i\!\in\!I';
\end{split}\EE
by~\eref{SCCregCond_e0}, the last maps agree on the overlaps.\\

\noindent
By \eref{cNconfdfn_e}-\eref{fDPsiIIconf_e} and~\eref{SCDcons_e}, 
\BE{SCCcons_e}\begin{split}
\cN_{I\cup J;I'\cup J}'=\cN_{I;I'}'\big|_{X_{I\cup J}}\,,
\qquad &\Psi_{I\cup J;I'\cup J}=\Psi_{I;I'}\big|_{\cN_{I\cup J;I'\cup J}'}\,,\\
\fD\Psi_{I\cup J;I'\cup J}
\big|_{\pi_{I\cup J;I'\cup J}^*\cN_{I\cup J;(I-I')\cup J}|_{\cN_{I\cup J;I'\cup J}'}}
&=\fD\Psi_{I;I'}\big|_{\pi_{I;I'}^*\cN_{I;I-I'}|_{\cN_{I\cup J;I'\cup J}'}}
\end{split}\EE
for all $I'\!\subset\!I\!\subset\![N]$ and $J\!\subset\![N]\!-\!I$
with $|I|\!\ge\!2$ in all three cases,
$|I'|\!\ge\!1$ in the first two cases, and $|I'|\!\ge\!2$ in the last case. 
By \eref{cNconfdfn_e}, \eref{fDPsiIIconf_e}, and~\eref{SCDcons_e2},
\BE{SCCcons_e2}\begin{split}
\Psi_{I;I''}=\Psi_{I';I''}\circ\fD\Psi_{I;I'}\big|_{\cN_{I;I''}'}\,, 
\quad
\fD\Psi_{I;I''}=\fD\Psi_{I';I''}\circ
\fD\Psi_{I;I'}\big|_{\pi_{I;I''}^*\cN_{I;I-I''}|_{\cN_{I;I''}'}}
\end{split}\EE
for all $I''\!\subset\!I'\!\subset\!I\!\subset\![N]$ with $|I'|\!\ge\!2$ in both cases,
$|I''|\!\ge\!2$ in the first case, and $|I''|\!\ge\!2$ in the second~case.

\begin{dfn}\label{SCCregul_dfn}
Let $N\!\in\!\Z^+$ and $\X\!\equiv\!\{X_I\}_{I\in\cP^*(N)}$ be a transverse configuration.

\begin{enumerate}[label=(\arabic*),leftmargin=*]
\item\label{SCCreg_it} 
If $(\om_i)_{i\in[N]}$ is a symplectic structure on $\X$ in
the sense of Definition~\ref{TransConf_dfn2}, 
an \sf{$(\om_i)_{i\in[N]}$-regularization for~$\X$} is a~tuple
\BE{SCCregdfn_e0}
\fR\equiv (\cR_I)_{I\in\cP^*(N)} \equiv
\big(\rho_{I;i},\na^{(I;i)},\Psi_{I;i}\big)_{i\in I\subset[N]}\EE
such that $(\Psi_{I;i})_{i\in I\subset[N]}$ is a  regularization 
for $\X$ in the sense of Definition~\ref{TransConfregul_dfn} and
for each $i\!\in\![N]$  the tuple
$$\big((\rho_{I;j},\na^{(I;j)})_{j\in I-i},\Psi_{I;i}\big)_{i\in I\subset[N]}$$
is an $\om_i$-regularization for $\{X_{ij}\}_{j\in[N]-i}$ in $X_i$
in the sense of Definition~\ref{SCDregul_dfn}\ref{sympregul_it}.

\item\label{SCCreg2_it}  If $B$ is a smooth manifold, possibly with boundary, and 
$(\om_{t;i})_{t\in B,i\in[N]}$ is a smooth family of symplectic structures
on $\X$, then
an \sf{$(\om_{t;i})_{t\in B,i\in[N]}$-family of regularizations for~$\X$} is a family of tuples
\BE{SCCregul_e2}(\fR_t)_{t\in B} \equiv
(\cR_{t;I})_{t\in B,I\in\cP^*(N)} \equiv
\big(\rho_{t;I;i},\na^{(t;I;i)},\Psi_{t;I;i}\big)_{t\in B,i\in I\subset [N]}\EE
such that $(\cR_{t;I})_{I\in\cP^*(N)}$ is an $(\om_{t;i})_{i\in[N]}$-regularization for
$\X$ for each $t\!\in\!B$ and
for each $i\!\in\![N]$  the tuple
$$\big((\rho_{t;I;j},\na^{(t;I;j)})_{j\in I-i},\Psi_{t;I;i}\big)_{t\in B,i\in I\subset[N]}$$
is an $(\om_{t;i})_{t\in B}$-family of regularizations for 
$\{X_{ij}\}_{j\in[N]-i}$ in $X_i$
in the sense of Definition~\ref{SCDregul_dfn}\ref{sympregul_it2}.
\end{enumerate}
\end{dfn}

\noindent
The assumptions in Definition~\ref{SCCregul_dfn}\ref{SCCreg_it} imply
that the corresponding isomorphisms~\eref{fDPsiIIconf_e} are 
product Hermitian isomorphisms covering the maps~\eref{cNconfdfn_e}.\\

\noindent
The precise definition of the total space of the complex line bundle
$\cO_{X_{\prt}}(X_{\eset})$ in~\eref{PsiDfn_e} depends
on the choices of identifications~$\Psi_{ij;j}$ of neighborhoods of~$X_{ij}$
in~$X_j$ and in~$X_j$ and of the $\om_j$-tame complex structures~$\fI_{ij;j}$
on (the fibers of) $\cN_{X_j}X_{ij}$ 
that satisfy~\eref{restrisom_e3a} and~\eref{restrisom_e3b}, respectively. 
For a smooth family $(X_{\eset},\om_{t;i})_{t\in B,i\in[N]}$ of 
SC symplectic varieties as in Definition~\ref{SCC_dfn},
such choices can be made continuously with respect to $t\!\in\!B$.
We then obtain a complex line~bundle 
\BE{BprtO_e}\pi_{B;\prt}\!:\cO_{B;X_{\prt}}\big(X_{\eset}\big)\equiv
\bigcup_{t\in B}\{t\}\!\times\!\cO_{t;X_{\prt}}\big(X_{\eset}\big)\lra B\!\times\!X_{\prt},\EE
where $\cO_{t;X_{\prt}}(X_{\eset})\!\lra\!X_{\prt}$ is the line bundle corresponding to
the symplectic structure $(\om_{t;i})_{i\in[N]}$ on~$X_{\eset}$.\\

\noindent
If  $\pi\!: L\!\lra\!M$ is a complex line bundle, we call 
a smooth map $\Phi\!: L\!\lra\!\C$  a \sf{trivialization}
of~$L$ if~$\Phi$ restricts to an isomorphism on each fiber of~$L$.
We call a family $(\hb_t)_{t\in B}$ of homotopy classes of trivializations of 
$\cO_{t;X_{\prt}}(X_{\eset})$ \sf{continuous} if for each $t_0\!\in\!B$ 
there exist a neighborhood~$U$ of~$t_0$ in~$B$ and a trivialization~$\Phi$ of 
$\cO_{B;X_{\prt}}(X_{\eset})|_{\pi_{B;\prt}^{\,-1}(U)}$ such that
the restriction of~$\Phi$ to $\{t\}\!\times\!\cO_{t;X_{\prt}}(X_{\eset})$
lies in~$\hb_t$ for every $t\!\in\!U$.\\

\noindent
A~regularization~$\fR$ for~$\X$ as in Definition~\ref{SCCregul_dfn}\ref{SCCreg_it} specifies 
the identifications~$\Psi_{ij;j}$ and complex structures~$\fI_{ij;j}$ 
needed for the construction of
the complex line bundles in~\eref{cOrestr_e} and~\eref{cOXXidfn_e}. 
Given a regularization~$\fR$, we thus view the line bundles~$\cO_{X_{jk}}(X_{ijk})$,
$\cO_{X_{\prt}}(X_i)$, and $\cO_{X_{\prt}}(X_{\eset})$
as explicitly specified and denote them by  $\cO_{\fR;X_{jk}}(X_{ijk})$, 
$\cO_{\fR;X_{\prt}}(X_i)$, and $\cO_{\fR;X_{\prt}}(X_{\eset})$, respectively.
By~\eref{SCCcons_e}, 
\BE{cORXiden_e}
\cO_{\fR;X_{jk}}(X_{ijk})\big|_{X_{jj'k}}=\cO_{\fR;X_{jj'k}}(X_{ijj'k})\big|_{X_{jj'k}}
=\cO_{\fR;X_{j'k}}(X_{ij'k})\big|_{X_{jj'k}}\EE
for all $i,j,j',k\!\in\![N]$ with $i,k\neq j,j'$ and $i\!\neq\!k$.
An $(\om_{t;i})_{t\in B,i\in[N]}$-family $(\fR_t)_{t\in B}$ of regularizations 
for~$\X$  completely specifies the complex line bundle~\eref{BprtO_e}.\\

\noindent
Let $s_{jk;i}$ denote the standard section 
of the line bundle $\cO_{\fR;X_{jk}}(X_{ijk})$:
\BE{cOXDsec_e}s_{jk;i}(x)=\begin{cases}
(x,v,v)\in\Psi_{ijk;jk}^{-1\,*}\pi_{ijk;jk}^{\,*}\cN_{ijk;jk}\,,
&\hbox{if}~x\!=\!\Psi_{ijk;jk}(v);  \\
(x,1)\!\in\!(X_{jk}\!-\!X_{ijk})\!\times\!\C,&\hbox{if}~x\!\in\!X_{jk}\!-\!X_{ijk}\,.
\end{cases}\EE
This section is well-defined on the overlap by definition of
$\cO_{\fR;X_{jk}}(X_{ijk})$; see~\eref{cOXVdfn_e}.
By~\eref{SCCcons_e}, 
\BE{scompp_e} s_{jk;i}\big|_{X_{jj'k}}=s_{j'k;i}\big|_{X_{jj'k}}
\qquad\forall~i,j,j',k\in[N],~i,k\neq j,j',~i\neq k.\EE
If $I\!\subset\![N]$, $j,k\!\in\!I$ are distinct, and $i\!\not\in\!I$, let
$$\cO_{\fR;X_I}(X_i)=\cO_{\fR;X_{jk}}(X_{ijk})\big|_{X_I}\,, \qquad
s_{I;i}=s_{jk;i}\big|_{X_I}\,.$$
By~\eref{cORXiden_e} and~\eref{scompp_e}, 
$\cO_{\fR;X_I}(X_i)$ and $s_{I;i}$ are independent of the choice of $j,k\!\in\!I$.
By~\eref{cOXDsec_e}, $s_{I;i}$ does not vanish outside of $X_{I\cup i}\!\subset\!X_I$.\\

\noindent
For every $I\!\subset\![N]$ with $|I|\!\ge\!2$,  define a smooth bundle~map
\begin{gather*}
\Pi_{\fR;I}\!:\cN X_I\equiv\bigoplus_{i\in I}\cN_{X_{I-i}}X_I 
\lra \cO_{\fR;X_{\prt}}(X_{\eset})\big|_{X_I}
\equiv \bigotimes_{i\in I}\cN_{X_{I-i}}X_I~\otimes\bigotimes_{i\not\in I}\cO_{\fR;X_I}(X_i),\\
\Pi_{\fR;I}\big((v_{I;i})_{i\in I}\big)= 
\bigotimes_{i\in I}v_{I;i}\otimes\bigotimes_{i\not\in I}s_{I;i}(x)
\qquad\forall~(v_{I;i})_{i\in I}\in \cN X_I\big|_x,~x\!\in\!X_I.
\end{gather*}
This map is surjective over the complement $X_I^{\star}$ 
of the submanifolds $X_{I'}\!\subset\!X_I$ with $I'\!\supsetneq\!I$.

\begin{dfn}\label{Phicomp_dfn}
Let $\X$ be an SC symplectic configuration as in~\eref{SCCdfn_e} 
and $\fR$ be a regularization for~$\X$ as in~\eref{SCCregdfn_e0}.
A trivialization~$\Phi$ of the complex line bundle~$\cO_{\fR;X_{\prt}}(X_{\eset})$
over~$X_{\prt}$ is \sf{$\fR$-compatible} if 
\BE{Phicompdfn_e}\Phi\big(\Pi_{\fR;I'}\big(\fD\Psi_{I;I'}(v_{I;I'},v_{I;I-I'})\big)\big)
=\Phi\big(\Pi_{\fR;I}(v_{I;I'},v_{I;I-I'})\big)\EE
for all $(v_{I;I'},v_{I;I-I'})\in\pi_{I;I'}^*\cN_{I;I-I'}\big|_{\cN_{I;I'}'}$
and $I'\!\subset\!I\!\subset\![N]$ with $|I'|\!\ge\!2$.
\end{dfn}

\noindent
Let $\X\!\equiv\!\{X_I\}_{I\in\cP^*(N)}$ be an $N$-fold transverse configuration
such that $X_{ij}$ is a closed submanifold of~$X_i$ of codimension~2
for all $i,j\!\in\![N]$ distinct and
$(\om_{t;i})_{t\in B,i\in[N]}$ be a family of symplectic structures on~$\X$.
Suppose the tuples  
\begin{equation*}\begin{split}
\big(\fR_t^{(1)}\big)_{t\in B}&\equiv
\big(\rho_{t;I;i},\na^{(t;I;i)},\Psi_{t;I;i}^{(1)}\big)_{t\in B,i\in I\subset[N]},\\
\big(\fR_t^{(2)}\big)_{t\in B}&\equiv
\big(\rho_{t;I;i},\na^{(t;I;i)},\Psi_{t;I;i}^{(2)}\big)_{t\in B,i\in I\subset[N]}
\end{split}\end{equation*}
are $(\om_{t;i})_{t\in B,i\in[N]}$-families of regularizations for~$\X$.
We define
\BE{fRequiv_e0}\big(\fR_t^{(1)}\big)_{t\in B} \cong \big(\fR_t^{(2)}\big)_{t\in B}\EE
if the two families of regularizations agree on the level of germs, i.e.~there exists 
an $(\om_{t;i})_{t\in B,i\in[N]}$-family of regularizations
as in~\eref{SCCregul_e2} such that 
\BE{fRequiv_e}
\Dom(\Psi_{t;I;i})\subset\Dom(\Psi_{t;I;i}^{(1)}),\Dom(\Psi_{t;I;i}^{(2)})
\quad
\Psi_{t;I;i}=\Psi_{t;I;i}^{(1)}\big|_{\Dom(\Psi_{t;I;i})},
\Psi_{t;I;i}^{(2)}\big|_{\Dom(\Psi_{t;I;i})}\EE
for all $t\!\in\!B$ and $i\!\in\!I\!\subset\![N]$.\\

\noindent
A family~\eref{SCCregul_e2} satisfying~\eref{fRequiv_e} provides a canonical identification
of the line bundles~\eref{BprtO_e} 
determined by $\fR_t^{(1)}$ and $\fR_t^{(2)}$.
This identification is independent of the choice of a family~\eref{SCCregul_e2}
satisfying~\eref{fRequiv_e}.
Thus, the line bundles~\eref{BprtO_e} determined by 
$(\om_{t;i})_{t\in B,i\in[N]}$-families of regularizations for~$\X$
satisfying~\eref{fRequiv_e0} are canonically identified.

\begin{prp}\label{SCC_prp}
Let $N\!\in\!\Z^+$, $\X\!\equiv\!\{X_I\}_{I\in\cP^*(N)}$ be a transverse configuration
such that $X_{ij}$ is a closed submanifold of~$X_i$ of codimension~2
for all $i,j\!\in\![N]$ distinct,
and $X_i^*\!\subset\!X_i$ for each $i\!\in\![N]$ be an open subset, possibly empty,
such that $\ov{X_i^*}\!\cap\!X_I\!=\!\eset$ for all $i\!\in\!I\!\!\subset\![N]$ with $|I|\!=\!3$.
Suppose
\begin{enumerate}[label=$\bullet$,leftmargin=*]

\item $B$ is a compact  manifold, possibly with non-empty boundary~$\prt B$, 
such that the restriction homomorphism \hbox{$H^1(B;\Z)\!\lra\!H^1(\prt B;\Z)$}
is surjective,

\item $N(\prt B),N'(\prt B)$ are tubular neighborhoods of $\prt B\!\subset\!B$
such that $\ov{N'(\prt B)}\!\subset\!N(\prt B)$,

\item $(\om_{t;i})_{t\in B,i\in[N]}$ is a smooth family of 
elements of $\Symp^+(\X)$ such that the associated line bundle
$\cO_{B;X_{\prt}}(X_{\eset})$ is trivializable, 

\item $(\hb_t)_{t\in B}$ is a continuous family of homotopy classes
of trivializations of the line bundles
$\cO_{t;X_{\prt}}(X_{\eset})$ determined by $(\om_{t;i})_{i\in[N]}$
for all $t\!\in\!B$,

\item $(\fR_t)_{t\in N(\prt B)}$ is an 
$(\om_{t;i})_{t\in N(\prt B),i\in[N]}$-family of regularizations for~$\X$, and

\item $(\Phi_t)_{t\in N(\prt B)}$ is a smooth family of $\fR_t$-compatible trivializations 
of the complex line bundles $\cO_{\fR_t;X_{\prt}}(X_{\eset})$
in the homotopy class~$\hb_t$.

\end{enumerate}
Then there exist a smooth family $(\mu_{t,\tau;i})_{t\in B,\tau\in\bI,i\in[N]}$ of
1-forms on~$X_{\eset}$, an $(\om_{t,1;i})_{t\in B,i\in[N]}$-family $(\wt\fR_t)_{t\in B}$
of regularizations for~$\X$, and  
a smooth family $(\wt\Phi_t)_{t\in B}$ of $\wt\fR_t$-compatible trivializations
of $\cO_{\fR_t;X_{\prt}}(X_{\eset})$ in the homotopy class~$\hb_t$
such~that 
\begin{gather*}
\big(\om_{t,\tau;i}\!\equiv\!\om_{t;i}\!+\!\nd\mu_{t,\tau;i}\big)_{i\in[N]}
\in \Symp^+(\X),\quad
\mu_{t,0;i}=0, \quad 
\supp\big(\mu_{\cdot,\tau;i}\big)\subset 
\big(B\!-\!N'(\prt B)\big)\!\times\!(X_i\!-\!X_i^*)
\end{gather*}
for all $t\!\in\!B$, $\tau\!\in\!\bI$, and $i\!\in\![N]$, and
\BE{SCCprp_e1}\big(\wt\fR_t\big)_{t\in N'(\prt B)} \cong \big(\fR_t\big)_{t\in N'(\prt B)}
\,,\quad
\big(\wt\Phi_t\big)_{t\in N'(\prt B)} = \big(\Phi_t\big)_{t\in N'(\prt B)}.\EE
\end{prp}

\subsection{Existence of compatible isomorphisms}
\label{SCCprpPf_subs}

\noindent
We prove Proposition~\ref{SCC_prp} below by making trivializations of 
$\cO_{\fR_t;X_{\prt}}(X_{\eset})$ $\fR_t$-compatible over 
neighborhoods of the strata of~$X_{\eset}$.
This argument in a sense adapts the setup of 
the proof of \cite[Theorem~2.17]{SympDivConf} to deal with bundle trivializations.
The key inductive step in this case is carried out by Lemma~\ref{SCC_lmm}.\\

\noindent
Let $\X$ be an SC symplectic configuration as in~\eref{SCCdfn_e},
$\fR$ be a regularization for~$\X$ as in~\eref{SCCregdfn_e0}, and
$W\!\subset\!X_{\prt}$.
We call a \sf{trivialization}~$\Phi$ of $\cO_{\fR;X_{\prt}}(X_{\eset})|_W$
\sf{$\fR$-compatible} if~\eref{Phicompdfn_e} is satisfied whenever 
$v_{I;I'}\!\in\!\Psi_{I;I'}^{\,-1}(W)|_{X_I\cap W}$.

\begin{proof}[{\bf{\emph{Proof of Proposition~\ref{SCC_prp}}}}]
With all references to the line bundle $\cO_{X_{\prt}}(X_{\eset})$ dropped,
Proposition~\ref{SCC_prp} is a special case of \cite[Theorem~2.17]{SympDivConf}.
Thus, we can assume that 
$(\fR_t)_{t\in B}$ is an $(\om_{t;i})_{t\in B,i\in[N]}$-family of regularizations
for~$\X$ as in~\eref{SCCregul_e2} and that $\cO_{B;X_{\prt}}(X_{\eset})$
is the line bundle as in~\eref{BprtO_e} constructed using this family of regularizations.
By Lemma~\ref{SCC_lmm0}, we can assume that this line bundle admits a trivialization
$$\Phi_B\!\equiv\!(\Phi_t)_{t\in B}\!: \cO_{B;X_{\prt}}(X_{\eset})\lra\C$$
so that  $\Phi_t$ lies in~$\hb_t$ for every $t\!\in\!B$.
In particular, this trivialization is $\fR_t$-compatible for every $t\!\in\!N(\prt B)$.
The above applications of \cite[Theorem~2.17]{SympDivConf} and Lemma~\ref{SCC_lmm0}
 require 
shrinking $N(\prt B)$ slightly;
so~$N(\prt B)$ in the remainder of this proof corresponds to
$N'(\prt B)$ in the statement of the proposition.\\

\noindent
Below we  deform~$\Phi_B$ to make its restriction to $\cO_{t;X_{\prt}}(X_{\eset})$
compatible with a shrinking of $\fR_t$ for all $t\!\in\!B$.
Fix a total order~$>$ on subsets $I\!\subset\![N]$ with $|I|\!\ge\!2$ 
so that  $I\!>\!I^*$ whenever $I\!\supsetneq\!I^*$.
We will proceed inductively on the strata $X_{I^*}$ of~$X_{\prt}$
using the total order~$>$.\\

\noindent
Suppose $I^*\!\subset\![N]$ with $|I^*|\!\ge\!2$,
$W^>$ is a neighborhood of
$$X_{I^*}^>\equiv \bigcup_{I>I^*}\!\! X_I \subset X_{\prt}$$ 
in $X_{\prt}$, and  $(\Phi^>_t)_{t\in B}$
is a smooth family of $\fR_t$-compatible trivializations of
$\cO_{\fR_t;X_{\prt}}(X_{\eset})|_{W^>}$  such~that 
\BE{SCCprp_e3}
\big(\Phi_t^>\big)_{t\in N(\prt B)}=\big(\Phi_t|_{W^>}\big)_{t\in N(\prt B)}, \quad
\big|\Phi_t(x) - \Phi_t^>(x)\big|_t <  \big|\Phi_t(x)\big|_t
~~\forall\,x\!\in\!W^>,\, t\!\in\!B.\EE
Let $W'$ be a neighborhood of~$X_{I^*}^>\!\subset\!X_{\prt}$ such that $\ov{W'}\!\subset\!W$. 
We apply Lemma~\ref{SCC_lmm} below with $W\!=\!W^>$ and $\Phi_t'\!=\!\Phi_t^>$.
There thus exist a neighborhood~$W_{I^*}$ of $X_{I^*}\!\subset\!X_{\prt}$,
an $(\om_{t;i})_{t\in B,i\in[N]}$-family $(\wt\fR_t)_{t\in B}$
of regularizations for~$\X$, and  
a smooth family $(\wt\Phi_t')_{t\in B}$ of $\wt\fR_t$-compatible
trivializations of $\cO_{\fR_t;X_{\prt}}(X_{\eset})|_{W'\cup W_{I^*}}$
satisfying the first condition in~\eref{SCCprp_e1}
and the two conditions in~\eref{SCCprp_e3} with $\Phi_t^>$ replaced~by
$\Phi_t^{\ge}\!\equiv\!\wt\Phi_t'$ and $W^>$ by $W^{\ge}\!\equiv\!W'\!\cup\!W_{I^*}$.\\

\noindent
By the downward induction on~$\cP^*(N)$ with respect to~$>$, 
we thus obtain an $(\om_{t;i})_{t\in B,i\in[N]}$-family $(\wt\fR_t)_{t\in B}$ 
of regularizations for~$\X$ and 
a smooth family $(\wt\Phi_t)_{t\in B}$ of $\wt\fR_t$-compatible trivializations
of $\cO_{\fR_t;X_{\prt}}(X_{\eset})$ satisfying~\eref{SCCprp_e1} such~that 
$$\big|\Phi_t(x)-\wt\Phi_t(x)\big|_t< \big|\Phi_t(x)\big|_t 
~~\forall\,x\!\in\!X_{\prt},\,t\!\in\!B.$$
This implies that 
$$\Phi_{t;\tau}\!\equiv\!(1\!-\!\tau)\Phi_t\!+\!\tau\wt\Phi_t\!:\,
  \cO_{\wt\fR_t;X_{\prt}}(X_{\eset})\lra X_{\eset}\!\times\!\C,
\qquad \tau\!\in\!\bI,$$
is a homotopy from $\Phi_t$ to $\wt\Phi_t$ through trivializations of 
$\cO_{\wt\fR_t;X_{\prt}}(X_{\eset})$ for every $t\!\in\!B$.
Thus, the trivialization~$\wt\Phi_t$ lies in the homotopy class~$\hb_t$
for every $t\!\in\!B$.
\end{proof}

\begin{lmm}\label{SCC_lmm0}
Let $N'(\prt B)\!\subset\!N(\prt B)\!\subset\!B$ be as in Proposition~\ref{SCC_prp},
$X$ be a CW complex, and $\cL\!\lra\!B\!\times\!X$ be a trivializable complex line bundle.
Suppose 
\begin{enumerate}[label=$\bullet$,leftmargin=*]

\item $(\hb_t)_{t\in B}$ is a continuous family of homotopy classes
of trivializations of the complex line bundles $\cL_t\!\equiv\!\cL|_{\{t\}\times X}$,

\item $\Phi_{N(\prt B)}$ is a trivialization of $\cL|_{N(\prt B)\times X}$
such that $\Phi_{N(\prt B)}|_{\cL_t}$ 
lies in~$\hb_t$ for every $t\!\in\!N(\prt B)$.

\end{enumerate}
Then there exists a trivialization~$\Phi_B$ of~$\cL$ so that $\Phi_B|_{\cL_t}$ 
lies in~$\hb_t$ for every $t\!\in\!B$ and
\BE{SCClmm0_e0}\Phi_B|_{N'(\prt B)}=\Phi_{N(\prt B)}|_{N'(\prt B)}\,.\EE
\end{lmm}

\begin{proof} Let $\eta_{S^1}\!\in\!H^1(S^1;\Z)$ be a generator.
For any topological space~$Y$, denote by $[Y,S^1]$  the set of homotopy classes
of continuous maps $Y\!\lra\!S^1$.
If $Y$ is a CW complex, then the~map
\BE{YS1H1_e}\big[Y,S^1\big]\lra H^1(Y;\Z), \qquad
\big[f\!:Y\!\lra\!S^1\big]\lra f^*\eta_{S^1},\EE
is a bijection.\\

\noindent
We can assume that $X$ is connected.
Let $\pi_B\!:B\!\times\!X\!\lra\!B$ be the projection and
$$\Phi_B'\!:\cL\lra \C$$
be a trivialization of~$\cL$.
For each $t\!\in\!B$, denote by $\hb_t'$ the homotopy class of maps 
\hbox{$f\!:X\!\lra\!S^1$} such~that the trivialization
$$\Phi_f'\!:\cL_t\lra\C, \qquad
\Phi_f'(v)\!=\!f(\pi_{\cL}(v))\Phi_B'(v),$$
lies in~$\hb_t$.\\

\noindent
Since $(\hb_t)_{t\in B}$ is a continuous family of homotopy classes,
for each $t_0\!\in\!B$ there exist a contractible neighborhood~$U$ of~$t_0$ in~$B$ and 
a continuous function $F_U\!:U\!\times\!X\!\lra\!S^1$ such that
$F_U|_{\{t\}\times X}$ lies in~$\hb_t'$ for every $t\!\in\!U$.
The class
$$\eta_U\equiv F_U^*\eta_{S^1}\in H^1\big(U\!\times\!X;\Z\big)$$
is then independent of the choice of~$U$.
These classes agree on the overlaps and thus determine an element 
$\eta_B\!\in\!H^1(B\!\times\!X;\Z)$.
Since the map~\eref{YS1H1_e} is surjective, there exists a continuous map
\hbox{$F\!:B\!\times\!X\!\lra\!S^1$}  such that $\eta_B\!=\!F^*\eta_{S^1}$.
Define
$$\Phi_F\!:\cL\lra \C, \qquad \Phi_F'(v)=F(\pi_{\cL}(v))\Phi_B'(v)\,.$$
For each $t\!\in\!B$, the trivialization $\Phi_F|_{\cL_t}$  lies in~$\hb_t$.\\

\noindent 
Let $F_{N(\prt B)}\!:N(\prt B)\!\times\!X\!\lra\!\C^*$ be the continuous function so that
\BE{SCClmm0_e7}
\Phi_{N(\prt B)}(v)=F_{N(\prt B)}(\pi_{\cL}(v))\Phi_{F;2}'(v)
\qquad\forall\,v\!\in\!\cL|_{N(\prt B)\times X}\,.\EE
For each $t\!\in\!\prt B$,
the restrictions of $\Phi_{N(\prt B)}$ and $\Phi_F'$ to $\cL_t$ are homotopic.
Thus, the restriction of $F_{N(\prt B)}$ to $\{t\}\!\times\!X$ with $t\!\in\!\prt B$
is null-homotopic.
This implies that 
$$F_{N(\prt B)}^*\eta_{S^1}\big|_{\prt B\times X}\in
H^1\big(\prt B;\Z\big)\!\otimes\!H^0(X;\Z\big)
\subset  H^1\big(\prt B\!\times\!X;\Z\big).$$
Since the restriction homomorphism \hbox{$H^1(B;\Z)\!\lra\!H^1(\prt B;\Z)$}
is surjective and the map~\eref{YS1H1_e} is bijective, there thus exists a continuous map 
$f\!:B\!\lra\!S^1$ such~that 
$$F_{N(\prt B)}^*\eta_{S^1}\big|_{\prt B\times X}=
\big(f^*\eta_{S^1}\big|_{\prt B}\big)\!\otimes\!1,\quad
\big[f\!\circ\!\pi_B|_{\prt B\times X}\big]\!=\!\big[F_{N(\prt B)}|_{\prt B\times X}\big]
\in\big[\prt B\!\times\!X,S^1\big].$$
 
\vspace{.2in}

\noindent
Since $N(\prt B)$  is a tubular neighborhood of $\prt B$,
the last equality above implies that there exists a continuous~function
\BE{SCClmm0_e11}G\!:B\!\times\!X\lra S^1 \qquad\hbox{s.t.}\quad
G(t,x)=\begin{cases}F_{N(\prt B)}(t,x),&\hbox{if}~t\!\in\!N'(\prt B);\\
f(t),&\hbox{if}~t\!\not\in\!N(\prt B).
\end{cases}\EE
Define 
$$\Phi_B\!:\cL\lra \C, \qquad
\Phi_B(v)=G(\pi_{\cL}(v))\Phi_{F;2}'(v)\,.$$
By the first case in~\eref{SCClmm0_e11} and~\eref{SCClmm0_e7}, this trivialization
of~$\cL$ satisfies~\eref{SCClmm0_e0}.
Since $N(\prt B)$  is a tubular neighborhood of~$\prt B$,
this implies that the restrictions of~$\Phi_B$ and $\Phi_{N(\prt B)}$
to~$\cL_t$ are homotopic for every $t\!\in\!N(\prt B)$.
By the second case in~\eref{SCClmm0_e11}, 
the restrictions of~$\Phi_B$ and $\Phi_F'$ 
to~$\cL_t$ are homotopic for every $t\!\not\in\!N(\prt B)$ as trivializations of~$\cL_t$.
Thus,  $\Phi_B|_{\cL_t}$ lies in~$\hb_t$ for every $t\!\in\!B$.
\end{proof}

\begin{lmm}\label{SCC_lmm}
Let $\X$,  $B$,  and $(\om_{t;i})_{t\in B,i\in[N]}$  be as in Proposition~\ref{SCC_prp}
and $N(\prt B)$ be a neighborhood of $\prt B\!\subset\!B$.
Suppose 
\begin{enumerate}[label=$\bullet$,leftmargin=*]

\item $I^*\!\in\!\cP(N)$ and $W,W'\!\subset\!X_{\prt}$ are open subsets such~that 
\BE{SCCweak_e0}
|I^*|\ge2, \quad \ov{W'}\subset W,\quad
X_I\subset W'~~\forall\,I\!\in\!\cP(N),\,I\!\supsetneq\!I^*\,,\EE

\item $(\fR_t)_{t\in B}$ is an $(\om_{t;i})_{t\in B,i\in[N]}$-family of 
regularizations for~$\X$,

\item $(\Phi_t)_{t\in B}$ is a smooth family of trivializations 
of $\cO_{\fR_t;X_{\prt}}(X_{\eset})$ over~$X_{\prt}$
which are $\fR_t$-compatible for $t\!\in\!N(\prt B)$,

\item $(\Phi_t')_{t\in B}$ is a smooth family of $\fR_t$-compatible trivializations 
of $\cO_{\fR_t;X_{\prt}}(X_{\eset})|_W$ such~that
\BE{SCCweak_e1}
\big(\Phi_t|_W\big)_{t\in N(\prt B)}=\big(\Phi_t'\big)_{t\in N(\prt B)}, \quad
\big|\Phi_t(x) - \Phi_t'(x)\big|_t <  \big|\Phi_t(x)\big|_t
~~\forall\,x\!\in\!W,\, t\!\in\!B.\EE

\end{enumerate}
Then there exist a neighborhood~$W_{I^*}$ of $X_{I^*}\!\subset\!X_{\prt}$,
an $(\om_{t;i})_{t\in B,i\in[N]}$-family $(\wt\fR_t)_{t\in B}$
of regularizations for~$\X$, and  
a smooth family $(\wt\Phi_t')_{t\in B}$ of $\wt\fR_t$-compatible
trivializations of $\cO_{\fR_t;X_{\prt}}(X_{\eset})|_{W'\cup W_{I^*}}$
such~that 
\begin{gather}\label{SCClmm_e2a}
\big(\wt\fR_t\big)_{t\in B} \cong \big(\fR_t\big)_{t\in B},\quad
\big(\wt\Phi_t'|_{W'}\big)_{t\in B}=\big(\Phi_t'|_{W'}\big)_{t\in B}\,, \\
\label{SCClmm_e2b}
\big(\Phi_t|_{W'\cup W_{I^*}}\big)_{t\in N(\prt B)}=\big(\wt\Phi_t'\big)_{t\in N(\prt B)}, \quad
\big|\Phi_t(x) - \wt\Phi_t'(x)\big| <  \big|\Phi_t(x)\big|
~~\forall\,x\!\in\!W'\!\cup\!W_{I^*},\, t\!\in\!B,
\end{gather}
\end{lmm}

\begin{proof} 
Let $(\fR_t)_{t\in B}$ be as in~\eref{SCCregul_e2}.
For each $x\!\in\!X_{\prt}$ and $x\!\in\!W$, let
\begin{equation*}\begin{split}
\Phi_t(x)&=\!\Phi_t|_{\cO_{\fR_t;X_{\prt}}(X_{\eset})|_x}\!:
 \cO_{\fR_t;X_{\prt}}(X_{\eset})|_x \lra \C \qquad\hbox{and}\\
\Phi_t'(x)&=\!\Phi_t'|_{\cO_{\fR_t;X_{\prt}}(X_{\eset})|_x}\!:
 \cO_{\fR_t;X_{\prt}}(X_{\eset})|_x \lra \C\,,
\end{split}\end{equation*}
respectively.\\

\noindent
Choose open subsets $W''\!\subset\!W'''\!\subset\!X_{\prt}$ such that 
\BE{SCClmm_e4a1}
\ov{W'}\subset W'', \quad  \ov{W''}\subset W''', \quad  \ov{W'''}\subset W.\EE
By~\eref{SCCweak_e0} and \cite[Lemma~5.8]{SympDivConf}, 
we can shrink the domains of the maps~$\Psi_{t;I;i}$ so~that
\BE{SCClmm_e3a}
\Psi_{t;I^*;\eset}^{-1}\big(W')\subset\cN X_{I^*}|_{X_{I^*}\!\cap\!W''},\quad
\Psi_{t;I;\eset}\big(\Dom\big(\Psi_{t;I;\eset}|_{X_I\cap W'})\big)\subset W''
~~\forall\,I\!\in\!\cP^*(N),\,|I|\!\ge\!2.\EE
Let $\rho\!: X_{I^*}\!\lra\![0,1]$ a smooth function such~that
\BE{SCClmm_e4a}
\rho(x)=\begin{cases}1,&\hbox{if}~x\!\in\!X_{I^*}\!\cap\!W''';\\
0,&\hbox{if}~x\!\not\in\!X_{I^*}\!\cap\!W.
\end{cases}\EE

\vspace{.1in}

\noindent
Define 
\begin{gather}
\notag
\wt\Phi_{t;I^*}\!:\, \cO_{\fR_t;X_{\prt}}(X_{\eset})|_{X_{I^*}} \lra \C,\\
\label{SCClmm_e4c}
\wt\Phi_{t;I^*}(x)  \equiv
\wt\Phi_{t;I^*}|_{\cO_{\fR_t;X_{\prt}}(X_{\eset})|_x} 
\equiv \rho(x)\Phi'_t(x) 
+ \big(1\!-\!\rho(x)\big) \Phi_t(x)
\quad \forall~x\!\in\!X_{I^*},\,t\!\in\!B.
\end{gather}
By~\eref{SCClmm_e4c} and~\eref{SCClmm_e4a}, 
\BE{SCClmm_e5a} 
\big(\wt\Phi_{t;I^*}|_{X_{I^*}\cap W'''}\big)_{t\in B}=
\big(\Phi_t'|_{X_{I^*}\cap W'''}\big)_{t\in B}\,.\EE
By~\eref{SCClmm_e4c} and~\eref{SCCweak_e1}, 
\BE{SCClmm_e7} 
\big(\Phi_t|_{X_{I^*}}\big)_{t\in N(\prt B)}=\big(\wt\Phi_{t;I^*}\big)_{t\in N(\prt B)}, \quad
\big|\Phi_t(x) - \wt\Phi_{t;I^*}(x)\big| <  \big|\Phi_t(x)\big|
~~\forall\,x\!\in\!X_{I^*},\, t\!\in\!B.\EE
In particular, $\wt\Phi_{t;I^*}(x)$ is a complex linear isomorphism for
all $x\!\in\!X_{I^*}$ and $t\!\in\!B$.\\

\noindent
Since $B$ is compact, there exists a neighborhood $W_{I^*}'$
of $X_{I^*}\!\subset\!X_{\eset}$ such~that
\BE{SCClmm_e9a}B\!\times\!W_{I^*}'\subset \bigcup_{t\in B}
\{t\}\!\times\!\Im\big(\Psi_{t;I^*;\eset}\big) \,.\EE
Since $\ov{W''}\!\subset\!W'''$ and $X_I\!\subset\!W'$ for all $I\!\in\!\cP(N)$
with $I\!\supsetneq\!I^*$, we can shrink~$W_{I^*}'$ so~that
\BE{SCClmm_e9b}
W''\!\cap\!W_{I^*}'\subset 
\Psi_{t;I^*;\eset}\big(\Dom(\Psi_{t;I^*;\eset})|_{X_{I^*}\cap W'''}\big)~~\forall\,t\!\in\!B,
\quad X_I\!\cap\!W_{I^*}'\subset W''~~I\!\in\!\cP(N),\,I\!\supsetneq\!I^*.\EE
Let $W_{I^*}^{\star}\!\subset\!W_{I^*}'$ be the complement of 
the subspaces $X_I\!\subset\!X_{\eset}$ with $I\!\subset\![N]$ such that $I\!\not\subset\!I^*$.\\

\noindent
Fix $t\!\in\!B$ and $x\!\in\!W_{I^*}^{\star}$. 
Let $I_x\!\subset\!I^*$ be the largest subset such that $x\!\in\!X_{I_x}$.
By~\eref{SCClmm_e9a}, 
$$x=\Psi_{t;I^*;I_x}(v_{I^*;I_x})$$
for a unique $v_{I^*;I_x}\!\in\!\cN_{I^*;I_x}\!\subset\cN X_{I^*}$.
Since $W_{I^*}'\!\subset\!X_{\prt}$, $|I_x|\!\ge\!2$.
Let $x_0\!=\!\pi_{I^*}(v_{I^*;I_x})$.
By the maximality of~$I_x$, 
the bundle map~$\Pi_{\fR_t;I_x}$ is surjective over~$x$ for every $t\!\in\!B$.
Given \hbox{$v\!\in\!\cO_{\fR_t;X_{\prt}}(X_{\eset})|_x$},
choose
$$v_{I^*;I^*-I_x}\in \cN_{I^*;I^*-I_x}\big|_{x_0} \quad\hbox{s.t.}~~
v= \Pi_{\fR_t;I_x}\big(\fD\Psi_{t;I^*;I_x}(v_{I^*;I_x},v_{I^*;I^*-I_x})\big).$$
Since $\fD\Psi_{t;I^*;I_x}$ is an isomorphism of split vector bundles,
$$\Pi_{\fR_t;I^*}(v_{I^*;I_x},v_{I^*;I^*-I_x})\in \cO_{\fR_t;X_{\prt}}(X_{\eset})\big|_{x_0}$$
is determined by $v$.
Thus, the~map 
\BE{SCClmm_e11}
\wt\Phi_t(x)\!: \cO_{\fR_t;X_{\prt}}(X_{\eset})\big|_x \lra \C, \quad
\big\{\wt\Phi_t(x)\big\}(v) \equiv
\wt\Phi_{t;I^*}\big(\Pi_{\fR_t;I^*}(v_{I^*;I_x},v_{I^*;I^*-I_x})\big),\EE
is well-defined.
Since $x_0\!\in\!X_{I^*}^{\star}$ 
and all components of~$v_{I^*;I_x}$ are nonzero, this map is an isomorphism.
Since $\wt\Phi_t|_{X_{I^*}^{\star}}\!=\!\wt\Phi_{t;I^*}|_{X_{I^*}^{\star}}$, 
\eref{SCClmm_e5a} and~\eref{SCClmm_e7} imply~that 
\BE{SCClmm_e12a} 
\big(\wt\Phi_t|_{X_{I^*}^{\star}\cap W'''}\big)_{t\in B}=
\big(\Phi_t'|_{X_{I^*}^{\star}\cap W'''}\big)_{t\in B}, \quad
\big(\wt\Phi_t|_{X_{I^*}^{\star}}\big)_{t\in N(\prt B)}=
\big(\Phi_t|_{X_{I^*}^{\star}}\big)_{t\in N(\prt B)}.\EE

\vspace{.1in}

\noindent
With $x$ as above, suppose $I_x\!\subset\!I'\!\subset\!I^*$ and
$x\!=\!\Psi_{t;I';I_x}(v_{I';I_x})$ for some 
$$v_{I';I_x}\in \cN_{I';I_x}\big|_{x'},\quad
x'=\Psi_{t;I^*;I'}(v_{I^*;I'})\in W_{I^*}^{\star},\quad
v_{I^*;I'}\in\cN_{I^*;I'}\subset\cN X_{I^*}\,.$$
By~\eref{overlap_e} with $I\!=\!I^*$ and the injectivity of~$\Psi_{t;I^*;I'}$,
$v_{I';I_x}=\fD\Psi_{t;I^*;I'}(v_{I^*;I_x})$.
By~\eref{SCClmm_e11} and the second statement in~\eref{SCCcons_e2}
with $I''\!\subset\!I'\!\subset\!I$ replaced by $I_x\!\subset\!I'\!\subset\!I^*$,
\begin{equation*}\begin{split}
&\big\{\wt\Phi_t(x)\big\}
\big(\Pi_{\fR_t;I_x}\big(\fD\Psi_{t;I';I_x}(\fD\Psi_{t;I^*;I'}(v_{I^*;I_x},v_{I^*;I^*-I_x}))\big)\big)
=\big\{\wt\Phi_{t;I^*}(x_0)\big\}\big(\Pi_{\fR_t;I^*}(v_{I^*;I_x},v_{I^*;I^*-I_x})\big)\\
&\hspace{2in}=\big\{\wt\Phi_t(x')\big\}
\big(\Pi_{\fR_t;I'}(\fD\Psi_{t;I^*;I'}(v_{I^*;I_x},v_{I^*;I^*-I_x}))\big).
\end{split}\end{equation*} 
Thus,
\begin{gather}\label{SCClmm_e15}
\begin{split}
&\big\{\wt\Phi_t(\Psi_{I;I'}(v_{I;I'})\big)\big\}\big(\Pi_{\fR_t;I'}
\big(\fD\Psi_{t;I;I'}(v_{I;I'},v_{I;I-I'})\big)\big)\\
&\hspace{2in}=\big\{\wt\Phi_t\big(\pi_I(v_{I;I'})\big)\big\}
\big(\Pi_{\fR_t;I}\big(v_{I;I'},v_{I;I-I'})\big)
\end{split}\\
\notag
\forall\,I'\!\subset\!I\!\subset\!I^*,\,|I'|\!\ge\!2, \quad
(v_{I;I'},v_{I;I-I'})\in \pi_{I;I'}^*\cN_{I;I-I'}
|_{\Psi_{t;I;I'}^{-1}(W_{I^*}^{\star})|_{X_I\cap W_{I^*}^{\star}}}.
\end{gather}
We conclude that $\wt\Phi_{t;2}$ satisfies~\eref{Phicompdfn_e} 
over~$W_{I^*}^{\star}$  whenever 
\hbox{$I'\!\subset\!I\!\subset\!I^*$}.\\

\noindent
By the $\fR_t$-compatibility of~$\Phi_t$ for all $t\!\in\!N(\prt B)$ and
the $\fR_t$-compatibility of~$\Phi_t'$ for all $t\!\in\!B$,
\begin{alignat*}{2}
\label{SCClmm_e16a}
\Phi_t\big(\Pi_{\fR;I}\big(\fD\Psi_{I^*;I}(v)\big)\big)
&=\Phi_t\big(\Pi_{\fR;I^*}(v)\big) &\quad& 
\forall\,
v\!\in\!\pi_{I^*;I}^*\cN_{I^*;I^*-I}\big|_{\Dom(\Psi_{t;I^*;I})},\,t\!\in\!N(\prt B),\\
\Phi_t'\big(\Pi_{\fR;I}\big(\fD\Psi_{I^*;I}(v)\big)\big)
&=\Phi_t'\big(\Pi_{\fR;I^*}(v)\big)
&\quad& \forall\,
v\!\in\!\pi_{I^*;I}^*\cN_{I^*;I^*-I}\big|_{\Psi_{t;I^*;I}^{\,-1}(W)|_{X_{I^*}\cap W}},
\,t\!\in\!B,
\end{alignat*}
whenever $I\!\subset\!I^*$ and $|I|\!\ge\!2$. 
Along with~\eref{SCClmm_e11} and~\eref{SCClmm_e12a}, these two statements imply~that 
\begin{alignat}{1}
\big(\wt\Phi_t
|_{\Psi_{t;I^*;\eset}(\Dom(\Psi_{t;I^*;\eset})|_{X_{I^*}\cap W'''})\cap W_{I^*}^{\star}}\big)_{t\in B}
&=\big(\Phi_t'|_{
\Psi_{t;I^*;\eset}(\Dom(\Psi_{t;I^*;\eset})|_{X_{I^*}\cap W'''})\cap W_{I^*}^{\star}}\big)_{t\in B}, \\
\notag
\big(\wt\Phi_t|_{\Im(\Psi_{t;I^*;\eset})\cap W_{I^*}^{\star}}\big)_{t\in N(\prt B)}
&=\big(\Phi_t|_{\Im(\Psi_{t;I^*;\eset})\cap W_{I^*}^{\star}}\big)_{t\in N(\prt B)}.
\end{alignat}
Combining these identities with the first assumption in~\eref{SCClmm_e9b} and
with~\eref{SCClmm_e9a}, 
we obtain
\BE{SCClmm_e19}
\big(\wt\Phi_t|_{W''\cap W_{I^*}^{\star}}\big)_{t\in B}
=\big(\Phi_t'|_{W''\cap W_{I^*}^{\star}}\big)_{t\in B}, \quad
\big(\wt\Phi_t|_{W_{I^*}^{\star}}\big)_{t\in N(\prt B)}
=\big(\Phi_t|_{W_{I^*}^{\star}}\big)_{t\in N(\prt B)}.\EE

\vspace{.1in}

\noindent
By the first identity in~\eref{SCClmm_e19}, the isomorphism
\BE{SCClmm_e25}\wt\Phi'_t(x)\!:\, \cO_{\fR_t;X_{\prt}}(X_{\eset})|_x\lra\C,
\qquad
\wt\Phi'_t(x)=\begin{cases}
\Phi'_t(x),&\hbox{if}~x\!\in\!W'';\\
\wt\Phi_t(x),&\hbox{if}~x\!\in\!W_{I^*}^{\star};
\end{cases}\EE
is well-defined for every $t\!\in\!B$.
By the second assumption in~\eref{SCClmm_e9b},  $W''\!\cup\!W_{I^*}^{\star}\!=\!W''\!\cup\!W_{I^*}'$.
Let
\BE{SCClmm_e27}\wt\Phi_t'\!:  
\cO_{\fR_t;X_{\prt}}(X_{\eset})|_{W''\cup W_{I^*}'}\lra \C, \quad
\wt\Phi_t'(v)=\big\{\wt\Phi_t'(\pi(v))\big\}(v)\,.\EE
By the first case in~\eref{SCClmm_e25}, this trivialization satisfies
the second condition in~\eref{SCClmm_e2a}.
By~\eref{SCClmm_e25}, the first assumption in~\eref{SCCweak_e1}, and
the second identity in~\eref{SCClmm_e19}, it also satisfies 
the first condition in~\eref{SCClmm_e2b}.\\

\noindent
We next verify that the restriction of~\eref{SCClmm_e27}
to $W'\!\cup\!W_{I^*}'$ is $\fR_t$-compatible.
Suppose
\begin{gather*}
t\!\in\!B,\quad I'\subset I\subset[N],\quad |I'|\ge2, \quad 
x\in \big(W'\!\cup\!W_{I^*}'\big)\!\cap\!X_I, \\ 
\big(v_{I;I'},v_{I;I-I'}\big)\in\pi_{I;I'}^*\cN_{I;I-I'}\big|_x,\quad
x'\!\equiv\!\Psi_{t;I;I'}(v_{I;I'})\in W'\!\cup\!W_{I^*}'\,.
\end{gather*}
If $x,x'\!\in\!W''$ or $x,x'\!\in\!W_{I^*}^{\star}$, then 
\BE{SCClmm_e28}\wt\Phi_{t;2}'\big(\Pi_{\fR_t;I'}
\big(\fD\Psi_{t;I;I'}(v_{I;I'},v_{I;I-I'})\big)\big)
=\wt\Phi_{t;2}'\big(\Pi_{\fR_t;I}(v_{I;I'},v_{I;I-I'})\big)\EE
by the $\fR_t$-compatibility of~$\Phi'_t$ in the first case and 
by~\eref{SCClmm_e15} in the second case.
If $x\!\in\!W_{I^*}^{\star}$ and $x'\!\in\!W'$, then \eref{overlap_e} 
and the first assumption in~\eref{SCClmm_e3a} imply~that 
$$x'\in W'\!\cap\!\Im\big(\Psi_{t;I^*;\eset}\big)\subset 
\Psi_{t;I^*;\eset}\big(\Dom\big(\Psi_{t;I;\eset}\big)\big|_{X_{I^*}\cap W''}\big),
\quad x\in \Psi_{t;I^*;\eset}
\big(\Dom\big(\Psi_{t;I;\eset}\big)\big|_{X_{I^*}^{\star}\cap W''}\big)\!\cap\!W_{I^*}^{\star}.$$
The identity~\eref{SCClmm_e28} in this case follows from~\eref{SCClmm_e16a}
and the $\fR_t$-compatibility of~$\Phi'_t$. 
If $x\!\in\!W'$ and $x'\!\in\!W_{I^*}^{\star}$, then 
the second assumption in~\eref{SCClmm_e3a} implies~that $x'\!\in\!W''$. 
The identity~\eref{SCClmm_e28} in this case follows from the $\fR_t$-compatibility of~$\Phi'_t$.\\

\noindent
Along with $\wt\Phi_{t;2}|_{X_{I^*}^{\star}}\!=\!\wt\Phi_{t;I^*}|_{X_{I^*}^{\star}}$, 
\eref{SCClmm_e25} and \eref{SCClmm_e5a} imply that 
$\wt\Phi_{t;2}'|_{X_{I^*}}\!=\!\wt\Phi_{t;I^*}$.
By the second statement in~\eref{SCClmm_e7} and the compactness of~$B$,
there thus exists a neighborhood~$W_{I^*}$ of $X_{I^*}\!\subset\!W_{I^*}'$ such~that
$$\big|\Phi_{t;2}(x)-\wt\Phi_{t;2}'(x)\big| <  \big|\Phi_{t;2}(x)\big|
~~\forall\,x\!\in\!W_{I^*},\, t\!\in\!B.$$
Combining this with the first case in~\eref{SCClmm_e25} and 
the second assumption in~\eref{SCCweak_e1}, 
we conclude that the isomorphism~$\wt\Phi_t'$ satisfies 
the second condition in~\eref{SCClmm_e2b}.
\end{proof}

\begin{rmk}\label{SCC_rmk}
Let $\Symp^+_{\hb}(\X)$ denote the space of pairs consisting
of an element $(\om_i)_{i\in[N]}$ of $\Symp^+(\X)$
and a trivialization~$\Phi$ of the associated line bundle $\cO_{X_{\prt}}(X_{\eset})$
in a homotopy class~$\hb$.
By our proof of Proposition~\ref{SCC_prp}, the projection
$$\wt\Aux_{\hb}(\X)\lra \Symp^+_{\hb}(\X), \qquad
\big((\om_i)_{i\in[N]},\fR,\Phi\big)\lra\big((\om_i)_{i\in[N]},\Phi),$$
is a weak homotopy equivalence.
\end{rmk}

\section{Main construction}
\label{SympSumPf_sec}

\noindent
Let $\X$ be an SC symplectic configuration as in~\eref{SCCdfn_e} 
which satisfies~\eref{SympSumCond_e} and 
$\hb$ be a homotopy class of trivializations of the associated line bundle~\eref{PsiDfn_e}.
By the $B\!=\!\{\tn{pt}\}$ case of Proposition~\ref{SCC_prp}, we can assume that 
this SC symplectic configuration admits a regularization~$\fR$ as in~\eref{SCCregdfn_e0} 
and an $\fR$-compatible trivialization~$\Phi$ of the complex line bundle
$\cO_{\fR;X_{\prt}}(X_{\eset})$ as in Definition~\ref{Phicomp_dfn}.\\

\noindent
In Section~\ref{SetupConstr_subs}, we rescale the diffeomorphisms $\Psi_{I;i}$ 
to increase their domains so that they contain balls of size at least~$2^N$
in each fiber.
In Section~\ref{FibrConstr_subs}, we patch together open subsets of these domains to form
a smooth manifold~$\cZ'$ with a smooth map $\pi_{\cC,\ve}\!:\cZ'\!\lra\!\C$. 
The latter is obtained by scaling the trivialization~$\Phi$ so that
the restriction of~$\pi_{\cC,\ve}$ to the preimage of the ball of radius~1 in~$\C$
forms a nearly regular fibration with uniform smooth fibers.
In Section~\ref{SympFibr_subs}, we construct
a closed two-form~$\wt\om_{\cC}^{(\ve)}$ on~$\cZ'$
and show that its restriction to a neighborhood~$\cZ$ of $X_{\eset}\!\subset\!\cZ'$ is symplectic. 
If $X_{\eset}$ is compact, this implies that there exists a neighborhood~$\De$ of~0
in~$\C$ such that $\cZ|_{\pi_{\cC,\ve}^{-1}(\De)}$ is a nearly regular symplectic fibration
and its fibers are compact.
As the various pieces of~$\cZ'$ are patched together only along~$X_{\prt}$,
the compactness of~$X_{\prt}$ suffices for the first conclusion;
we comment in Remark~\ref{noncpmt_rmk} on obtaining 
this conclusion even if $X_{\prt}$ is not compact.\\

\noindent
The construction in Sections~\ref{FibrConstr_subs} and~\ref{SympFibr_subs}
works on compact families of the relevant data on $(X_I)_{I\in\cP^*(N)}$.
By the $B\!=\!\bI$ case of Proposition~\ref{SCC_prp}, the deformation equivalence class
of the output of this construction is thus determined by the deformation equivalence class of 
the original SC symplectic configuration~$\X$ and the homotopy class~$\hb$ of 
trivializations of~\eref{PsiDfn_e}.

\subsection{Setup and notation}
\label{SetupConstr_subs}

\noindent
We begin by setting up the relevant notation.
We will need several smooth $\R^+$-valued functions on the strata $X_I$ 
and their open subspaces.
These will be denoted by~$\ve$ or~$\cC$ with some decorations, 
depending on whether the function should be sufficiently small or sufficiently large.
The former means that it is pointwise smaller than another pre-specified continuous function
on the same space or on a neighborhood of its closure;
the meaning of sufficiently large is similar.
If $X_{\prt}$ is compact, such functions can be chosen to be constant.\\

\noindent
For each $I\!\in\!\cP^*(N)$, let $\om_I\!=\!\om_i|_{X_I}$ for any $i\!\in\!I$;
this symplectic form on~$X_I$ is independent of the choice of $i\!\in\!I$.
For $i\!\in\!I\!\subset\![N]$ with $|I|\!\ge\!2$, let 
\BE{rhoIidfn_e}\rho_{I;i}\!:\cN_{X_{I-i}}X_I\lra\R\EE
be as in~\eref{SCCregdfn_e0} and 
\BE{alIidfn_e}\al_{I;i}\equiv \al_{\rho_{I;i},\na^{(I;i)}}\in  
\Ga(\cN_{X_{I-i}}X_I\!-\!X_I;T^*\cN_{X_{I-i}}X_I\big)\EE
be the connection 1-form on $\cN_{I;I-i}\!=\!\cN_{X_{I-i}}X_I$
determined by the Hermitian structure $(\rho_{I;i},\na^{(I;i)})$.\\

\noindent
We also denote~by
$$\rho_{I;i}\!:\cN X_I\lra\R \qquad\hbox{and}\qquad
\al_{I;i}\in\Ga(\cN X_I\!-\!\cN_{I;i};T^*\cN X_I\big)$$
the function and the 1-form obtained by pulling back~\eref{rhoIidfn_e} and~\eref{alIidfn_e}
by the projection map \hbox{$\pr_{I;I-i}\!:\cN X_I\!\lra\!\cN_{X_{I-i}}X_I$}.
The 1-form $\rho_{I;i}\al_{I;i}$ is then smooth on~$\cN X_I$.
Define
$$\rho_I\!: \cN X_I\lra \R, \qquad
\rho_I(v)=\max\big\{\rho_{I;i}(v)\!:\,i\!\in\!I\big\},$$
to be the square norm on~$\cN X_I$.
Let
$$\wh\om_I\equiv\wh\om_{(\rho_{I;i},\na^{(I;i)})_{i\in I}}^{\bu}
=\pi_I^*\om_I+\frac12\sum_{i\in I} \nd(\rho_{I;i}\al_{I;i})
=\pi_I^*\om_I+\frac12\sum_{i\in I} \pr_{I;I-i}^*\nd(\rho_{I;i}\al_{I;i})$$
be the closed 2-form on the total space of $\cN X_I$ as in~\eref{ombund_e}.\\

\noindent
For each $I\!\subset\![N]$ with $|I|\!\ge\!2$, the homomorphism
\begin{gather}
\notag
\La_{\C}^{\top}\cN X_I\equiv\bigotimes_{i\in I}\cN_{X_{I-i}}X_I
\lra \cO_{\fR;X_{\prt}}(X_{\eset})\big|_{X_I}
\equiv \La_{\C}^{\top}\cN X_I\otimes\bigotimes_{i\not\in I}\cO_{\fR;X_I}(X_i),\\
\label{PhiIisom_e}
w\lra w\otimes\bigotimes_{i\not\in I}\!s_{I;i}(x)
\qquad\forall~w\in\La_{\C}^{\top}\cN X_I\big|_x,~x\!\in\!X_I,
\end{gather}
is an isomorphism over~$X_I^{\star}$.
Thus, there exists a smooth function \hbox{$\cC_{\Phi;I}\!:X_I^{\star}\!\lra\!\R^+$} such~that 
\BE{PhiIsomBnd_e1}
\prod_{i\in I}\rho_{I;i}(v)= 
\cC_{\Phi;I}\big(\pi_I(v)\big) \big|\Phi\big(\Pi_{\fR;I}(v)\big)\big|^2
\qquad\forall\,v\!\in\!\cN X_I\big|_{X_I^{\star}}.\EE

\vspace{.1in}

\noindent
Given $\ve\!:X_{\eset}\!\lra\!\R^+$ and $I\!\subset\![N]$
with $|I|\!\ge\!2$, we also denote by $\ve$ the composition
$$\ve\!:\cN X_I\stackrel{\pi_I}{\lra} X_I \stackrel{\ve}{\lra}\R^+\,.$$
Define
\begin{gather*}
\cN X_I(\ve)=\big\{v\!\in\!\cN X_I\!:\,\rho_I(v)\!<\!\ve(v)\big\}, \qquad
\cN_{I;I'}(\ve)=\cN X_I(\ve)\!\cap\!\cN_{I;I'}~~\forall\,I'\!\subset\!I\,,\\
m_{\ve;I}\!:\cN X_I\lra \cN X_I, \qquad 
m_{\ve;I}(v)=\ve(v)v~~\forall~v\!\in\!\cN X_I\,.
\end{gather*}
For each $i\!\in\!I$, let
$$\rho_{I;i}^{(\ve)}\!:\cN X_I\lra\R^{\ge0},\qquad 
\rho_{I;i}^{(\ve)}=m_{\ve;I}^*\rho_{I;i}=\ve^2\rho_{I;i}\,.$$
If $\ve|_{X_I}$ is smooth, set
\BE{SympSum_e7}
\wh\om_I^{(\ve)}=m_{\ve;I}^*\wh\om_I
=\pi^*\om_I+\frac12\sum_{i\in I} \nd(\rho_{I;i}^{(\ve)}\al_{I;i})\,.\EE

\vspace{.1in}

\noindent
For $I'\!\subset\!I\!\subset\![N]$ with $|I|\!\ge\!2$ and $I'\!\neq\!\eset$, let
$$\cN_{I;I'}'\subset\cN_{I;I'}\subset\cN X_I, \qquad
\Psi_{I;I'}\!:\cN_{I;I'}'\lra X_{I'}$$
be as in~\eref{cNconfdfn_e}.
If in addition $|I'|\!\ge\!2$, let 
$$\fD\Psi_{I;I'}\!:\pi_{I;I'}^*\cN_{I;I-I'}\big|_{\cN_{I;I'}'}
\lra \cN X_{I'}\big|_{\Im(\Psi_{I;I'})}$$
be as in~\eref{fDPsiIIconf_e}.
Since $\fD\Psi_{I;I'}$ is a product Hermitian isomorphism, \eref{Phicompdfn_e} implies~that
\BE{PhiIsomBnd_e2}
\cC_{\Phi;I}\big(\pi_I(v)\big) =\cC_{\Phi;I'}\big(\Psi_{I;I'}(v)\big)
\!\!\!\!\prod_{i\in I-I'}\!\!\!\!\rho_{I;i}(v)
\quad\forall~v\!\equiv\!(v_i)_{i\in I-I'}\!\in\!\cN_{I;I'}'\big|_{X_I^{\star}}
~\hbox{s.t.}~v_i\!\neq\!0~\forall\,i\!\in\!I\!-\!I',\EE
whenever $I'\!\subset\!I\!\subset\![N]$ and $|I'|\!\ge\!2$.\\

\vspace{.1in}

\noindent
As shown in the proof of \cite[Lemma~5.8]{SympDivConf},
there exists a continuous function $\ve\!:X_{\eset}\!\lra\!\R^+$ such that 
$\ve|_{X_i}$ is smooth for all $i\!\in\![N]$,  
\BE{SympSum_e15}
\ov{\cN_{I;i}(4^N\ve^2)}\subset\cN_{I;i}',~\qquad
\ve\big(\Psi_{I;i}(v)\big)=\ve(v)
~~\forall\,v\!\in\!\cN_{I;i}(4^N\ve^2)\EE
for all $i\!\in\!I\!\subset\![N]$ with $|I|\!\ge\!2$, and 
\BE{Psiover0_e}
\Psi_{I_1;i}\big(\cN_{I_1;i}(4^N\ve^2)\big)\cap
\Psi_{I_2;i}\big(\cN_{I_2;i}(4^N\ve^2)\big)
=\Psi_{I_1\cup I_2;i}\big(\cN_{I_1\cup I_2;i}(4^N\ve^2)\big)\EE
for all $i\!\in\!I_1,I_2\!\subset\![N]$ with $|I_1|,|I_2|\!\ge\!2$.
Furthermore, $\ve$ can be chosen so that its restriction to $X_{\prt}$ is
smaller than any given continuous function $\ve_{\prt}\!:X_{\prt}\!\lra\!\R^+$.\\

\noindent
For $\ve\!:X_{\eset}\!\lra\!\R^+$ as in~\eref{SympSum_e15} and 
$i\!\in\!I\!\subset\![N]$ with $|I|\!\ge\!2$,
\BE{SympSum_e17}
\Psi_{I;i}^{(\ve)}\!: \big(\cN_{I;i}(4^N),\wh\om_I^{(\ve)}\big)\lra (X_i,\om_i), \quad
\Psi_{I;i}^{(\ve)}(v)=\Psi_{I;i}\big(m_{\ve;I}(v)\big),\EE
is a symplectomorphism onto a neighborhood of~$X_I$ in~$X_i$.
Since these symplectomorphisms satisfy the matching condition~\eref{SCCregCond_e0} 
with~$\Psi_{I;i}$ replaced by~$\Psi_{I;i}^{(\ve)}$, we can define smooth maps
$$\Psi_{I;I'}^{(\ve)}\!: \cN_{I;I'}(4^N)\lra X_{I'}, \qquad \eset\!\neq\!I'\!\subset\!I,$$
as in~\eref{cNconfdfn_e}.
By~\eref{PhiIsomBnd_e2} and the second assumption in~\eref{SympSum_e15}, 
\begin{gather}\label{PhiIsom_e2}
\cC_{\Phi;I}\big(\pi_I(v)\big) =
\cC_{\Phi;I'}\big(\Psi_{I;I'}^{(\ve)}(v)\big)
\,\ve\big(\Psi_{I;I'}^{(\ve)}(v)\big)^{2|I-I'|}
\!\!\!\!\prod_{i\in I-I'}\!\!\!\!\rho_{I;i}(v)\\
\notag
\quad\forall~v\!\equiv\!(v_i)_{i\in I-I'}\!\in\!\cN_{I;I'}(4^N\ve)\big|_{X_I^{\star}},~
I'\!\subset\!I\!\subset\![N]~
\hbox{s.t.}~v_i\!\neq\!0\,\forall\,i\!\in\!I\!-\!I',\,|I'|\!\ge\!2.
\end{gather}
By~\eref{Psiover0_e}, 
\BE{Psiover0_e2}
\Psi_{I_1;I'}^{(\ve)}\big(\cN_{I_1;I'}(4^N)\big)\cap
\Psi_{I_2;I'}^{(\ve)}\big(\cN_{I_2;I'}(4^N)\big)
=\Psi_{I_1\cup I_2;I'}^{(\ve)}\big(\cN_{I_1\cup I_2;I'}(4^N)\big)\EE
for all $I'\!\subset\!I_1,I_2\!\subset\![N]$ with $I'\!\neq\!\eset$ and
$|I_1|,|I_2|\!\ge\!2$.\\

\noindent
For $I'\!\subset\!I$ with $|I'|\!\ge\!2$, define
\begin{gather}
\label{fDeIIpr0_e}
\fD\Psi_{I;I'}^{(\ve)}\!: \pi_{I;I'}^*\cN_{I;I-I'}\big|_{\cN_{I;I'}(4^N)}
\lra \cN X_{I'}\big|_{\Im(\Psi_{I;I'}^{(\ve)})} \qquad\hbox{by}\\
\notag
\fD\Psi_{I;I'}^{(\ve)}\big(v_{I;I'},v_{I;I-I'}\big)
=\fD\Psi_{I;I'}\big(m_{\ve;I}(v_{I;I'}),v_{I;I-I'}\big)
\quad \forall\,\big(v_{I;I'},v_{I;I-I'}\big)
\in \pi_{I;I'}^*\cN_{I;I-I'}\big|_{\cN_{I;I'}(4^N)}. 
\end{gather}
By the second assumption in~\eref{SympSum_e15}, 
\BE{epscl_e}\begin{split}
\ve\!\circ\!\fD\Psi_{I;I'}^{(\ve)}
&=\ve\!\circ\pi_{I;I'}^*\pi_{I;I-I'}
\big|_{\pi_{I;I'}^*\cN_{I;I-I'}|_{\cN_{I;I'}(4^N)}}, \\
m_{\ve;I'}\!\circ\!\fD\Psi_{I;I'}^{(\ve)}
&=\fD\Psi_{I;I'}^{(\ve)}\!\circ\!\pi_{I;I'}^*m_{\ve;I}
\big|_{\pi_{I;I'}^*\cN_{I;I-I'}|_{\cN_{I;I'}(4^N)}}\,.
\end{split}\EE
Since $\fD\Psi_{I;I'}$ lifts $\Psi_{I;I'}$ to a product Hermitian isomorphism with respect to 
the product Hermitian structures $(\rho_{I;i},\na^{(I;i)})_{i\in I'}$ 
and $(\rho_{I';i},\na^{(I';i)})_{i\in I'}$ ,
$\fD\Psi_{I;I'}^{(\ve)}$ lifts $\Psi_{I;I'}^{(\ve)}$ to a product Hermitian isomorphism  
with respect to these structures.
In particular,
$$\fD\Psi_{I;I'}^{(\ve)}\big(\cN X_I(4^N)\big)=
\cN X_{I'}(4^N)\big|_{\Im(\Psi_{I;I'}^{(\ve)})}\,.$$
By~\eref{epscl_e} and~\eref{SCCcons_e2},
\BE{SCCepcons_e2}\begin{split}
\Psi_{I;I''}^{(\ve)}=\Psi_{I';I''}^{(\ve)}\!\circ\!\fD\Psi_{I;I'}^{(\ve)}\big|_{\cN_{I;I''}(4^N)}\,, 
\quad
\fD\Psi_{I;I''}^{(\ve)}=\fD\Psi_{I';I''}^{(\ve)}\circ
\fD\Psi_{I;I'}^{(\ve)}\big|_{\pi_{I;I''}^*\cN_{I;I-I''}|_{\cN_{I;I''}(4^N)}}
\end{split}\EE
for all $I''\!\subset\!I'\!\subset\!I\!\subset\![N]$ with $|I'|\!\ge\!2$ in both cases,
$|I''|\!\ge\!1$ in the first case, and $|I''|\!\ge\!2$ in the second~case.\\

\noindent
Let $\cC\!:X_{\eset}\!\lra\!\R^+$ be a continuous function such that  
\BE{SympSum_e25}
\cC\big(\Psi_{I;i}(v)\big)=\cC\big(\pi_{I;i}(v)\big)
~~\forall\,v\!\in\!\cN_{I;i}(4^N\ve),~i\!\in\!I\!\subset\![N],~|I|\!\ge\!2.\EE
For $I\!\subset\![N]$ with $|I|\!\ge\!2$, define
\BE{SympSum_e27}
\Phi_{\cC,\ve;I}\!:\cN X_I\lra\C, \quad
\Phi_{\cC,\ve;I}(v)=\cC\big(\pi_I(v)\big)\ve\big(\pi_I(v)\big)^{|I|-1}
\Phi\big(\Pi_{\fR;I}(v)\big)
\quad\forall~v\!\in\!\cN X_I\,.\EE
By~\eref{PhiIsomBnd_e1},
\BE{PhiIsom_e1}
\cC\big(\pi_I(v)\big)^2\ve(v)^{2(|I|-1)}\prod_{i\in I}\rho_{I;i}(v)= 
\cC_{\Phi;I}\big(\pi_I(v)\big) \big|\Phi_{\cC,\ve;I}(v)\big|^2
\quad\forall~v\!\in\!\cN X_I\big|_{X_I^{\star}},\EE
whenever $I\!\subset\![N]$ and $|I|\!\ge\!2$.
By~\eref{SympSum_e25}, \eref{epscl_e}, and~\eref{Phicompdfn_e},  
\BE{PhiPsi_e}
\Phi_{\cC,\ve;I'}\!\circ\!\fD\Psi_{I;I'}^{(\ve)}
=\Phi_{\cC,\ve;I}\big|_{\pi_{I;I'}^*\cN_{I;I-I'}|_{\cN_{I;I'}(4^N)}}\EE
for all $I'\!\subset\!I\!\subset\![N]$ and $|I'|\!\ge\!2$.\\

\noindent
If $I\!=\!\{i\}$  with $i\!\in\![N]$, let 
\BE{PhiSimpDfn_e}
\Phi_{\cC,\ve;I}\!=\!\pi_2:\cN X_I\equiv X_i\!\times\!\C\lra\C \EE
be the projection to the second component.
For $I'\!\equiv\!\{i\}\!\subsetneq\!I\!\subset\![N]$, define
\begin{gather}
\label{DPsiext_e}
\fD\Psi_{I;I'}^{(\ve)}\!:\pi_{I;I'}^*\cN_{I;I-I'}\big|_{\cN_{I;I'}(4^N)}
\!\equiv\!\pi_{I;I'}^*\cN_{X_{I-i}}X_I\big|_{\cN_{I;I'}(4^N)}\lra 
\cN X_{I'}\big|_{\Im(\Psi_{I;I'}^{(\ve)})}\,,\\
\notag
\fD\Psi_{I;I'}^{(\ve)}\big((v_{I;j})_{j\in I-I'},v_{I;i}\big)=
\big(\Psi_{I;I'}^{(\ve)}\big((v_{I;j})_{j\in I-I'}\big),
\Phi_{\cC,\ve;I}\big((v_{I;j})_{j\in I}\big)\big).
\end{gather} 
The restriction of this map to the subspace
\BE{diffsub_e}
\big\{((v_{I;j})_{j\in I-I'},v_{I;i})\!\in\!\pi_{I;I'}^*\cN_{I;I-I'}
|_{\cN_{I;I'}(4^N)|_{X_I^{\star}}}\!:
\,v_{I;j}\!\neq\!0~\forall\,j\!\in\!I\!-\!I'\big\}\EE
is a diffeomorphism onto an open subspace of~$\cN X_{I'}\!=\!\cN X_i$.\\

\noindent
For $I\!=\!\{i\}$, let
$$\fD\Psi_{I;I}^{(\ve)}=\id\!: \pi_{I;I}^*\cN_{I;\eset}\big|_{\cN_{I;I}(4^N)}
=\cN X_I\lra\cN X_I\,;$$
for $I\!\subset\![N]$ with $|I|\!\ge\!2$, this is already the case
by~\eref{fDPsiIIconf_e} and the $I'\!=\!I$ case of~\eref{wtPsiIIprop_e}.
By~\eref{PhiPsi_e}, \eref{PhiSimpDfn_e}, and~\eref{DPsiext_e},
\BE{PhiPsi_e2}
\Phi_{\cC,\ve;I'}\!\circ\!\fD\Psi_{I;I'}^{(\ve)}
=\Phi_{\cC,\ve;I}\big|_{\pi_{I;I'}^*\cN_{I;I-I'}|_{\cN_{I;I'}(4^N)}}\EE
for all $I'\!\subsetneq\!I\!\subset\![N]$ with $I'\!\neq\!\eset$.
By~\eref{SCCepcons_e2} and~\eref{PhiPsi_e}, 
\BE{SCCcons_e4}
\fD\Psi_{I;I''}^{(\ve)}=\fD\Psi_{I';I''}^{(\ve)}\circ\fD\Psi_{I;I'}^{(\ve)}
\big|_{\pi_{I;I''}^*\cN_{I;I-I''}|_{\cN_{I;I''}(4^N)}}\EE
for all $I''\!\subset\!I'\!\subset\!I\!\subset\![N]$ with $I''\!\neq\!\eset$.\\

\noindent
For $\ve\!:X_{\eset}\!\lra\!\R^+$ as in~\eref{SympSum_e15} and $I\!\in\!\cP^*(N)$, let
\BE{riXdfn_e}\ri{X}_I=X_{I;I}=
X_I-\bigcup_{I\subsetneq I'\subset[N]}\!\!\!\!\!\!
\ov{\Psi_{I';I}^{(\ve)}\big(\cN_{I';I}(2^{|I'|-1})\big)};\EE
see the left diagram in Figure~\ref{3conf_fig2}.
Since $\fD\Psi_{I_1\cup I_2;I_1}^{(\ve)}$ is a product Hermitian isomorphism,
\eref{Psiover0_e2} and the first equality in~\eref{SCCepcons_e2} 
imply~that 
\BE{Psiover_e}\begin{split}
\Psi_{I_1;I'}^{(\ve)}\big(\cN_{I_1;I'}(2^{|I_1|})|_{\ri{X}_{I_1}}\big)\cap 
\Psi_{I_2;I'}^{(\ve)}\big(\cN_{I_2;I'}(2^{|I_2|})|_{\ri{X}_{I_2}}\big)
\neq\eset \quad&\Lra\quad
I_1\!\supset\!I_2~~\hbox{or}~~I_1\!\subset\!I_2\,,\\
\ov{\Psi_{I_1;I'}^{(\ve)}\big(\cN_{I_1;I'}(2^{|I_1|})|_{\ri{X}_{I_1}}\big)}\cap 
\ov{\Psi_{I_2;I'}^{(\ve)}\big(\cN_{I_2;I'}(2^{|I_2|-1})\big)}
\neq\eset \quad&\Lra\quad
I_1\!\supset\!I_2\,,
\end{split}\EE
whenever $I'\!\subset\!I_1,I_2\!\subset\![N]$
with $I'\!\neq\!\eset$ and $|I_1|,|I_2|\!\ge\!2$.\\

\noindent
Let $i\!\in\![N]$.
By \eref{Psiover0_e2}, the first identity in~\eref{SCCepcons_e2},
and~\eref{PhiIsom_e2}, the function 
\begin{gather*}
\wh\ve_i\!:\ri{X}_i\cap
\bigcup_{\{i\}\subsetneq I\subset[N]}\!\!\!\!\!\!
\Psi_{I;i}^{(\ve)}\big(\cN_{I;i}(4^{|I|})\big)\lra\R^+,\\
\wh\ve_i\big(\Psi_{I;i}^{(\ve)}(v)\big)^2=
\frac{\ve(v)^{2(|I|-1)}}{\cC_{\Phi;I}(\pi_I(v))}
\!\!\!\prod_{j\in I-i}\!\!\!\!\rho_{I;j}(v) 
~~\forall\,v\!\in\!\{\Psi_{I;i}^{(\ve)}\}^{-1}(\ri{X}_i)\!\cap\!\cN_{I;i}(4^N),
~\{i\}\!\subsetneq\!I\!\subset\![N],
\end{gather*}
is well-defined and smooth.
Thus, there is a smooth function \hbox{$\wt\ve_i\!:\ri{X}_i\!\lra\!\R^+$}
such~that
$$\wt\ve_i\big(\Psi_{I;i}^{(\ve)}(v)\big)^2=
\frac{\ve(v)^{2(|I|-1)}}{\cC_{\Phi;I}(\pi_I(v))}
\!\!\!\prod_{j\in I-i}\!\!\!\!\rho_{I;j}(v)
~~\forall\,v\!\in\!\{\Psi_{I;i}^{(\ve)}\}^{-1}(\ri{X}_i)\!\cap\!\cN_{I;i}(2^N),
~\{i\}\!\subsetneq\!I\!\subset\![N].$$
By~\eref{PhiIsom_e1} and~\eref{SympSum_e25}, 
\begin{gather}\label{PhiIsom_e4}
\big|\Phi_{\cC,\ve;I}(v)\big|^2=
\cC\big(\Psi_{I;i}^{(\ve)}((v_{I;j})_{j\in I-i})\big)^2
\wt\ve_i\big(\Psi_{I;i}^{(\ve)}((v_{I;j})_{j\in I-i})\big)^2
\rho_{I;i}(v)\\
\notag\forall~
v\!\equiv\!\big((v_{I;j})_{j\in I-i},v_{I;i}\big)
\!\in\!\pi_{I;i}^*\cN_{I;I-i}
\big|_{\{\Psi_{I;i}^{(\ve)}\}^{-1}(\ri{X}_i)\cap\cN_{I;i}(2^N)},
~\{i\}\!\subsetneq\!I\!\subset\![N].
\end{gather}

\subsection{Construction of fibration}
\label{FibrConstr_subs}

\noindent
With $\cC$ as in~\eref{SympSum_e25}, define
$$\cZ_I=\begin{cases}
\{v_I\!\in\!\cN X_I|_{\ri{X}_I}\!:\,\rho_I(v_I)\!<\!2^{|I|}\},&
\hbox{if}~|I|\ge2;\\
\{(x,\la)\!\in\!\cN X_i|_{\ri{X}_i}\!:\,
|\la|^2\!<\!2\,\cC(x)^2\wt\ve_i(x)^2\},&\hbox{if}~I\!=\!\{i\}.
\end{cases}$$
For $I\!\subsetneq\!I'\!\subset\![N]$ with $I\!\neq\!\eset$, let
\begin{alignat}{1}
\label{XIIprdfn_e}
X_{I';I}&=\cZ_{I'}\!\cap\!\cN_{I';I}, \qquad 
X_{I;I'}=\Psi_{I';I}^{(\ve)}\big(X_{I';I}\big)\subset X_I,\\
\notag
\cZ_{I';I}&=\big\{v\!\in\!\cZ_{I'}\!:\,\rho_{I';i}(v)\!<\!2^{|I|}~
\forall\,i\!\in\!I\big\}
-\bigcup_{I\subsetneq J\subset I'}\!\!\!\!\!
\big\{v\!\in\!\cZ_{I'}\!:\,\rho_{I';i}(v)\!\le\!2^{|J|-1}~
\forall\,i\!\in\!J\!-\!I\big\}\,, \\
\label{cZIOverlap_e}
\cZ_{I;I'}&= \fD\Psi_{I';I}^{(\ve)}\big(\cZ_{I';I}\big)
=\cZ_I\big|_{\ri{X}_I\cap X_{I;I'}}\,.
\end{alignat}
The first two subspaces above are illustrated in Figure~\ref{3conf_fig2}.
For $|I|\!\ge\!2$, the last equality in~\eref{cZIOverlap_e}
holds because $\fD\Psi_{I';I}^{(\ve)}$
is a product Hermitian isomorphism;
for $|I|\!=\!1$, it holds by~\eref{DPsiext_e} and~\eref{PhiIsom_e4}.
This crucial equality is used in the proof of Hausdorffness in Proposition~\ref{cZtopol_prp} 
below.\\

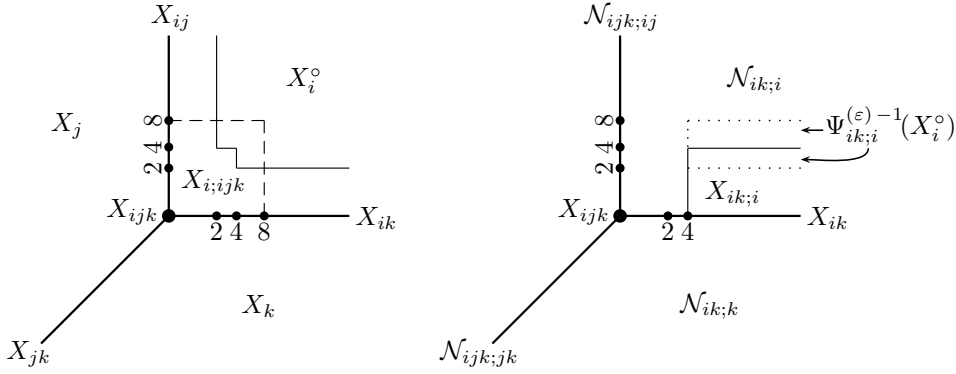
\begin{figure}
\begin{pspicture}(-4,-2.5)(11,2.4)
\psset{unit=.3cm}
\psline[linewidth=.1](-1,-2)(7,-2)\psline[linewidth=.1](-1,-2)(-1,6)
\psline[linewidth=.1](-1,-2)(-6.66,-7.66)\pscircle*(-1,-2){.3}
\rput(5,4){\sm{$\ri{X}_i$}}\rput(1,-.5){\sm{$X_{i;ijk}$}}
\rput(-5.5,2){\sm{$X_j$}}\rput(3,-6){\sm{$X_k$}}
\rput(8.2,-2.1){\sm{$X_{ik}$}}\rput(-.9,6.8){\sm{$X_{ij}$}}
\rput(-7.2,-8.1){\sm{$X_{jk}$}}
\rput(-2.6,-1.8){\sm{$X_{ijk}$}}
\psline[linewidth=.05,linestyle=dashed](-1,2.23)(3.23,2.23)
\psline[linewidth=.05](1.12,1)(1.12,6)
\psline[linewidth=.05](1.12,1)(2,1)\psline[linewidth=.05](2,.12)(2,1)
\psline[linewidth=.05](2,.12)(7,.12)
\psline[linewidth=.05,linestyle=dashed](3.23,-2)(3.23,2.23)
\rput(1.12,-2.7){\sm{$2$}}\rput(2,-2.7){\sm{$4$}}\rput(3.23,-2.7){\sm{$8$}}
\pscircle*(1.12,-2){.2}\pscircle*(2,-2){.2}\pscircle*(3.23,-2){.2}
\rput{90}(-1.7,.12){\sm{$2$}}\rput{90}(-1.7,1.05){\sm{$4$}}\rput{90}(-1.7,2.23){\sm{$8$}}
\pscircle*(-1,.12){.2}\pscircle*(-1,1.05){.2}\pscircle*(-1,2.23){.2}
\psline[linewidth=.1](19,-2)(27,-2)\psline[linewidth=.1](19,-2)(19,6)
\psline[linewidth=.1](19,-2)(13.34,-7.66)\pscircle*(19,-2){.3}
\rput(19.1,6.8){\sm{$\cN_{ijk;ij}$}}\rput(12.8,-8.1){\sm{$\cN_{ijk;jk}$}}
\rput(25,4){\sm{$\cN_{ik;i}$}}\rput(23,-6){\sm{$\cN_{ik;k}$}}
\rput(28.2,-2.1){\sm{$X_{ik}$}}\rput(24,-1){\sm{$X_{ik;i}$}}
\rput(31,2){\sm{$\Psi_{ik;i}^{(\ve)\,-1}\!(\ri{X}_i)$}}
\rput(17.4,-1.8){\sm{$X_{ijk}$}}
\psline[linewidth=.05](22,1)(27,1)\psline[linewidth=.05](22,1)(22,-2)
\psline[linewidth=.08,linestyle=dotted](22,1.05)(22,2.23)
\psline[linewidth=.08,linestyle=dotted](22,.12)(27,.12)
\psline[linewidth=.08,linestyle=dotted](22,2.23)(27,2.23)
\rput(21.12,-2.7){\sm{$2$}}\rput(22,-2.7){\sm{$4$}}
\pscircle*(21.12,-2){.2}\pscircle*(22,-2){.2}
\rput{90}(18.3,.12){\sm{$2$}}\rput{90}(18.3,1.05){\sm{$4$}}\rput{90}(18.3,2.23){\sm{$8$}}
\pscircle*(19,.12){.2}\pscircle*(19,1.05){.2}\pscircle*(19,2.23){.2}
\pnode(30,1){A}\pnode(27,0.5){B}
\nccurve[linewidth=.07,angleA=-90,angleB=0,ncurv=.5]{->}{A}{B}
\pnode(28,1.8){A}\pnode(27,1.8){B}
\nccurve[linewidth=.07,angleA=180,angleB=0,ncurv=.5]{->}{A}{B}
\end{pspicture}
\caption{The subspaces $\ri{X}_i,X_{i;ijk}\!\subset\!X_i$, and
$X_{ik;i}\!\subset\!\cN_{ik;i}$ appearing in the smoothing construction 
for a triple configuration as in Figure~\ref{3conf_fig};
the numbers indicate the value of the square distance along each axis
in $\cN X_{ijk}$.}
\label{3conf_fig2}
\end{figure}

\noindent
Define
\begin{gather}\notag
\cZ'=\bigg(\bigsqcup_{I\in\cP^*(N)}\!\!\!\!\!\!\cZ_I\bigg)\!\!\bigg/\!\!\!\sim, \\
\label{cZequiv_e}
\cZ_{I'}\supset\cZ_{I';I}\ni v_{I'}\sim \fD\Psi_{I';I}^{(\ve)}(v_{I'})\in\cZ_{I;I'}\subset\cZ_I
\qquad\forall\,I\!\subsetneq\!I'\!\subset\![N],~I\!\neq\!\eset.
\end{gather}
By  the first statement in~\eref{Psiover_e} and~\eref{SCCcons_e4}, 
$\sim$ is an equivalence relation.
The restriction of the quotient~map
\BE{qmapdfn_e}q\!: \bigsqcup_{I\in\cP^*(N)}\!\!\!\!\!\!\cZ_I \lra \cZ'\EE
to each $\cZ_I$ is injective; we thus identify~$\cZ_I$ with $q(\cZ_I)$ as~sets.
By Proposition~\ref{cZtopol_prp} below, this identification respects the smooth structures.\\

\noindent
By~\eref{PhiPsi_e2} with~$I$ and~$I'$ interchanged, the~map 
\BE{piRdfn_e}\pi_{\cC,\ve}\!:\cZ'\lra \C,\qquad
\pi_{\cC,\ve}\big(q(v)\big)=\Phi_{\cC,\ve;I}(v)\quad
\forall~v\!\in\!\cZ_I,\,I\!\in\!\cP^*(N),\EE
is well-defined.
It is continuous, since the maps $\Phi_{\cC,\ve;I}$ are continuous.
For each $i\!\in\![N]$, the~map
\BE{ioZdfn_e}\io_{\cC,\ve;i}\!: X_i\lra\cZ', \quad
\io_{\cC,\ve;i}(x)=\begin{cases}
q(x,0),&\hbox{if}~x\!\in\!\ri{X}_i;\\
q(v),&\hbox{if}~x\!=\!\Psi_{I;i}^{(\ve)}(v),\,v\!\in\!X_{I;i},\,
i\!\in\!I\!\subset\![N];
\end{cases}\EE
is well-defined, injective, and continuous.
Since $\io_{\cC,\ve;i}|_{X_{ij}}\!=\!\io_{\cC,\ve;j}|_{X_{ij}}$ for all $i,j\!\in\![N]$,
we obtain an injective continuous~map 
\BE{iocCvedfn_e}\io_{\cC,\ve}\!: X_{\eset}\lra\cZ'\,.\EE
The image of this map is $\cZ_0'\!\equiv\!\pi_{\cC,\ve}^{-1}(0)$.
The substance of the last statement in Proposition~\ref{cZtopol_prp} is that 
the fibers $\cZ_{\la}'\!\equiv\!\pi_{\cC,\ve}^{-1}(\la)$ are compact if
$X_{\eset}$ is compact, $\cC$ is sufficiently large, and~$|\la|\!<\!1$.

\begin{prp}\label{cZtopol_prp}
If $\ve|_{X_i}$ and $\cC|_{X_i}$ are smooth for all $i\!\in\![N]$, then
$\cZ'$ is a Hausdorff topological space with a smooth structure 
so that the restriction of~\eref{qmapdfn_e} to each~$\cZ_I$ is a diffeomorphism
onto~$q(\cZ_I)$.
The map~\eref{piRdfn_e} is smooth on~$\cZ'$ and a submersion outside of $\io_{\cC,\ve}(X_{\prt})$
 with respect to this smooth structure.
The maps~\eref{ioZdfn_e} are smooth embeddings;
their images form a transverse collection of closed submanifolds of~$\cZ'$ of codimension~2.
If $\cC$ is sufficiently large (depending on~$\ve$), every sequence $v_I^{(k)}$ in 
$\cZ_I\!\cap\!\pi_{\cC,\ve}^{-1}(\la)$ 
with $|\la|\!<\!1$ and 
with the sequence $\pi_I(v_I^{(k)})$ converging in~$X_I$ has a limit point in~$\cZ'$.
\end{prp}

\begin{proof}
For all $I\!\subset\!I'\!\subset\![N]$ with $I\!\neq\!\eset$,
the subspace $\cZ_{I';I}\!\subset\!\cN X_{I'}$ is open by definition.
Since each map~$\fD\Psi_{I';I}^{(\ve)}$ is a diffeomorphism onto an open subspace of~$\cN X_I$,
everywhere on its domain if $|I|\!\ge\!2$ and 
on the subspace~\eref{diffsub_e} with $I$ and $I'$ interchanged if $|I|\!=\!1$,
the subspace $\cZ_{I;I'}\!\subset\!\cN X_I$ is also open.
Since the identification maps are diffeomorphisms between open subspaces of manifolds,
the quotient map~\eref{qmapdfn_e}
is open.
Assuming $\cZ'$ is Hausdorff (as shown below), $q$ thus induces a smooth structure on~$\cZ'$.
Since $\Phi_{\cC,\ve;I}\!:\cZ_I\!\lra\!\C$ is a submersion outside of 
the subspaces $\cZ_I\!\cap\!\cN_{I;I'}$ with $I'\!\subset\!I$ such that $|I'|\!\ge\!2$, 
the map~\eref{piRdfn_e} is a submersion outside of $\io_{\cC,\ve}(X_{\prt})$.
The maps~\eref{ioZdfn_e} are smooth embeddings because their restrictions
to the preimages of~$\cZ_I$ correspond to the inclusions of the hyperplane
subbundles $\cN_{I;i}$ of~$\cN X_I$.
For the same reason, their images form a transverse collection 
of submanifolds of~$\cZ'$ of codimension~2.
These submanifolds~$X_i$ are closed in~$\cZ'$ because $X_i$ is closed in~$X_{\eset}$
and $\io_{\cC,\ve}(X_{\eset})\!=\!\pi_{\cC,\ve}^{-1}(0)$
is closed in~$\cZ'$.\\

\noindent
Let $[v],[w]\!\in\!\cZ'$ be distinct points and $I,J\!\subset\![N]$ be 
the maximal subsets so that~$[v]$ lies in the image of some $v\!\in\!\cZ_I$
under~$q$ and $[w]$ lies in the image of some $w\!\in\!\cZ_J$;
$I$ and~$J$ are well-defined by the first statement in~\eref{Psiover_e}.
If $I\!=\!J$, let $V,W\!\subset\!\cZ_I$ be disjoint open subsets around~$v$ and~$w$.
Since $q$ is an open map which is injective on~$\cZ_I$,
$q(V),q(W)\!\subset\!\cZ'$ are disjoint open subsets containing~$[v]$ and~$[w]$,
respectively. 
If $I\!\not\subset\!J$ and $J\!\not\subset\!I$, 
then the open neighborhoods $q(\cZ_I),q(\cZ_J)\!\subset\!\cZ'$ of~$[v]$ and~$[w]$,
respectively, are disjoint by the first statement in~\eref{Psiover_e}.\\

\noindent
Suppose $I\!\subsetneq\!J$.
Let $\de\!>\!0$ be such~that 
$$w\in W\equiv
\big\{v_J\!\in\!\cZ_J\!:\,
\rho_{J;j}(v_J)\!<\!2^{|J|}\!-\!\de~\forall\,j\!\in\!J\!-\!I\big\}.$$
Since $v\!\not\in\!\cZ_{I;J}$ (by the maximality assumption on~$I$),
$$v\in V\equiv \cZ_I-
\cN X_I\big|_{\Psi_{J;I}^{(\ve)}(\{v_{J;I}\in\cN_{J;I}:\,
\rho_J(v_{J;I})\le2^{|J|}-\de\})}$$
by~\eref{cZIOverlap_e} with $I'\!=\!J$.
Since~\eref{cZequiv_e} is an equivalence relation,
$q(V),q(W)\!\subset\!\cZ'$ are disjoint open subsets containing~$[v]$ and~$[w]$,
respectively.
Thus, $\cZ'$ is Hausdorff.\\

\noindent
We now verify the last claim.
By~\eref{PhiIsom_e1}, there exists a continuous function 
\hbox{$\cC_{\Phi,\ve}\!:X_{\eset}\!\lra\!\R^+$} such~that 
\BE{PhiBnd_e}
\cC\big(\pi_I(v)\big)^2\prod_{i\in I}\rho_{I;i}(v)\le 
\cC_{\Phi,\ve}\big(\pi_I(v)\big) \big|\Phi_{\cC,\ve;I}(v)\big|^2
\quad\forall~v\!\in\!\cN X_I\big|_{\ri{X}_I},
\,I\!\subset\![N],\,|I|\!\ge\!2;\EE
the function~$\cC_{\Phi,\ve}$ depends on~$\ve$, but not on~$\cC$.
The inequality~\eref{PhiBnd_e} provides a bound on the product of the norms~$\rho_{I;i}(v)$
with $i\!\in\!I$ in terms of $|\la|\!=\!|\Phi_{\cC,\ve;I}(v)|$;
this bound becomes stronger as~$\cC$ increases.\\

\noindent
Suppose $\la\!\in\!\C^*$, $I\!\in\!\cP^*(N)$, and 
$$v_I^{(k)}\equiv  \big(v^{(k)}_{I;i}\big)_{i\in I}
\in \cZ_I\!\cap\!\Phi_{\cC,\ve;I}^{-1}(\la)\subset \cN X_I \stackrel{\pi_I}{\lra}\ri{X}_I\subset X_I$$ 
is a sequence such that $\pi_I(v_I^{(k)})$ converges to a point $x_I\!\in\!X_I$.
A subsequence of~$v_I^{(k)}$ then converges to some 
$$v_I\equiv(v_{I;i})_{i\in I}\in \cN X_I\big|_{x_I}
\qquad\hbox{s.t.}\qquad \Phi_{\cC,\ve;I}(v_I)=\la,\quad
\rho_{I;i}(v_I)\le2^{|I|}~\forall\,i\!\in\!I~~\hbox{if}~~|I|\!\ge\!2.$$
If $x_I\!\in\!\ri{X}_I$ and $|I|\!=\!\{i\}$ for some $i\!\in\![N]$, 
then $v_I\!=\!(x_I,\la)\!\in\!\cZ_I$
if $|\la|\!<\!1$ and \hbox{$2\cC^2\wt\ve_i^2\!>\!1$} (i.e.~$\cC$ is sufficiently large).
The sequence~$v_I^{(k)}$ then has a limit point in~$\cZ'$.\\

\noindent
Suppose $x_I\!\in\!\ri{X}_I$ and $|I|\!\ge\!2$.
Let $J\!\subset\!I$ be a maximal subset 
such that $\rho_{I;i}(v_I)\!<\!2^{|J|}$ for all $i\!\in\!J$.
If $J\!\neq\!\eset$, define
$$x_J=\Psi_{I;J}^{(\ve)}\big((v_{I;i})_{i\in I-J}\big)\in X_J.$$
If $x_J\!=\!\Psi_{I';J}^{(\ve)}(v_{I'})$ for some $v_{I'}\!\in\!\cN_{I';J}$
with $J\!\subsetneq\!I'\!\subset\![N]$
and $\rho_{I'}(v_{I'})\!\le\!2^{|I'|-1}$, then  
$I\!\supset\!I'$ by the second statement in~\eref{Psiover_e}
with $(I_1,I_2,I')$ replaced by $(I,I',J)$.
By the first identity in~\eref{SCCepcons_e2} and 
$\fD\Psi_{I;I'}^{(\ve)}$ being a product Hermitian isomorphism,
$$v_{I'} =\fD\Psi_{I;I'}^{(\ve)}\big((v_{I;i})_{i\in I-I'},(v_{I;i})_{i\in I'-J}\big),
\qquad \rho_{I;i}(v_I)\le 2^{|I'|-1}~~\forall\,i\!\in\!I'\!-\!J.$$
This would contradict maximality of~$J$.
Thus, $x_J\!\in\!\ri{X}_J$.
If $J\!=\!\eset$, 
replace it with any single-element subset of~$I$ and define~$x_J$ in the same way.
As in the first case,  $x_J\!\in\!\ri{X}_J$.
If $|J|\!=\!1$,
$$\fD\Psi_{I;J}^{(\ve)}(v_I)=(x_J,\la)\in\cZ_J\,,$$
as in the previous paragraph.  
If $|J|\!\ge\!2$, $\fD\Psi_{I;J}^{(\ve)}(v_I)$ lies in $\cZ_J$ as well,
since $\fD\Psi_{I;J}^{(\ve)}$ is a product Hermitian isomorphism and
$\rho_{I;i}(v_I)\!<\!2^{|J|}$ for all $i\!\in\!J$.
It follows that  $\fD\Psi_{I;J}^{(\ve)}(v_I^{(k)})\!\in\!\cZ_J$
for all $k$ large and 
$\fD\Psi_{I;J}^{(\ve)}(v_I)$ is a limit point of this set.
Thus, $\fD\Psi_{I;J}^{(\ve)}(v_I)$ is a limit point of the sequence~$v_I^{(k)}$
in~$\cZ'$.\\

\noindent
Suppose $x_I\!\not\in\!\ri{X}_I$.
Let $I'\!\subset\![N]$ be the maximal subset so that $I'\!\supset\!I$ and
$x_I\!=\!\Psi_{I';I}^{(\ve)}(v_{I'})$ for~some 
$v_{I'}\!\in\!\cN_{I';I}$ with $\rho_{I'}(v_{I'})\!=\!2^{|I'|-1}$.
By~\eref{riXdfn_e}, \eref{Psiover0_e2}, and the first identity in~\eref{SCCepcons_e2},
$I'$ is well-defined.
Furthermore,
\BE{vcond_e}
\rho_{I';j}(v_{I'})\ge 2^{|I|}\ge2~~\forall\,j\!\in\!I'\!-\!I.\EE
If $\pi_{I'}(v_{I'})\!=\!\Psi_{J;I'}^{(\ve)}(v_J)$ for some $v_J\!\in\!\cN_{J;I'}$
with $I'\!\subsetneq\!J\!\subset\![N]$ and $\rho_J(v_J)\!\le\!2^{|J|-1}$, then either
$$x_I\in\Psi_{J;I}^{(\ve)}\big(\cN_{J;I}(2^{|J|-1})\big) \qquad\hbox{or}\qquad
\rho_J(v_J)=2^{|J|-1}\,.$$
The first possibility would contradict to $x_I$ being in the closure of~$\ri{X}_I$;
the second possibility would contradict the maximality of~$I'$.
Thus, $\pi_{I'}(v_{I'})\!\in\!\ri{X}_{I'}$.\\

\noindent
For all $k\!\in\!\Z^+$ sufficiently large (so that $\pi_I(v_I^{(k)})$ lies in~$X_{I;I'}$), 
let  
$$v_{I'}^{(k)}\equiv 
\big(v_{I';I}^{(k)},v_{I';I'-I}^{(k)}\big)= \big\{\fD\Psi_{I';I}^{(\ve)}\big\}^{-1}
\big(v_I^{(k)}\big)\in \pi_{I';I}^*\cN_{I;I-I'}\big|_{\cN_{I';I}(2^{|I'|})}\,.$$
Since $\pi_I(v_I^{(k)})$ converges to $x_I$, the sequence $v_{I';I}^{(k)}$
converges to~$v_{I'}$.
If $|I|\!\ge\!2$, 
$$\rho_{I';i}(v_{I'}^{(k)})=\rho_{I;i}(v_I^{(k)})\le 2^{|I|}<2^{|I'|}
\qquad\forall~i\!\in\!I;$$
thus, $v_{I'}^{(k)}$ has a limit point in~$\cZ_{I'}$.
If $I\!=\!\{i\}$, 
$$2^{|I'|-1}\cC\big(\pi_{I'}(v_{I'}^{(k)})\big)^2\rho_{I';i}(v_{I'}^{(k)})
\le \cC\big(\pi_{I'}(v_{I'}^{(k)})\big)^2\prod_{j\in I'}\rho_{I';j}(v_{I'}^{(k)}) 
\le \cC_{\Phi,\ve}\big(\pi_I(v_{I'}^{(k)})\big)|\la|^2  $$
by~\eref{vcond_e} and~\eref{PhiBnd_e}.
Thus, the sequence $v_{I'}^{(k)}$ converges in $\cZ_{I'}$, provided 
$$\cC_{\Phi,\ve}(x)|\la|^2< 2^{|I'|-1}\cC(x)^2\cdot 2^{|I'|}
\qquad\forall~x\!\in\!X_{I'}.$$
Therefore, the last claim of the proposition holds if $\cC^2\!>\!\cC_{\Phi,\ve}$.
\end{proof}

\subsection{Symplectic structure}
\label{SympFibr_subs}

\noindent
We next define a closed 2-form $\wt\om_{\cC;I}^{(\ve)}$ on $\cZ_I$ 
for each $I\!\in\!\cP^*(N)$.
Let $\eta\!:\R\!\lra\![0,1]$ be a smooth function such~that 
$$\eta(r)=\begin{cases}0,&\hbox{if}~r\le \frac34;\\
1,&\hbox{if}~r\ge1. \end{cases}$$
For $i\!\in\!I\!\subset\![N]$ with $|I|\!\ge\!2$, let
$\eta_{I;i}=\eta\!\circ\!\rho_{I;i}\!:\cN X_I\!\lra\!\R$ and
\BE{wtomdfn_e}
\wt\om_{\cC;I}^{(\ve)}=\Bigg(\!\!\pi_I^*\om_I+
\frac12
\nd\bigg(\!\sum_{i\in I}\!\bigg(1-\!\!\prod_{j\in I-\{i\}}\!\!\!\!\!\!\eta_{I;j}\!\!\bigg)
\ve^2\rho_{I;i}\al_{I;i}
+\bigg(\sum_{i\in I}
\prod_{j\in I-\{i\}}\!\!\!\!\!\!\eta_{I;j}\!\!\bigg)\ve^2\Phi_{\cC,\ve;I}^*\big(r^2\nd\th\big)
\!\!\bigg)\!\!\Bigg)\bigg|_{\cZ_I}.\EE
For $I\!=\!\{i\}$, define 
\BE{wtomdfn_e0}
\wt\om_{\cC;I}^{(\ve)}
=\Big(\pi_1^*\om_i+\frac12\nd\big(\ve^2\Phi_{\cC,\ve;I}^*(r^2\nd\th)\big)\Big)\Big|_{\cZ_I}
=\Big(\pi_1^*\om_i+\frac12\nd\big(\ve^2\pi_2^*(r^2\nd\th)\big)\Big)\Big|_{\cZ_I}\,,\EE
where $\pi_1,\pi_2\!:X_i\!\times\!\C\!\lra\!X_i,\C$ are the two projections and
$(r,\th)$ are the polar coordinates on~$\C$.
Let $\De_r\!\subset\!\C$ denote the disk of radius~$r$ centered at the origin.

\begin{lmm}\label{sympform_lmm}
If $\ve|_{X_i}$ and $\cC|_{X_i}$ are smooth for all $i\!\in\![N]$
and $\cC$ is sufficiently large (depending on~$\ve$), then
\BE{sympform_e}
\wt\om_{\cC;I'}^{(\ve)}\big|_{\Phi_{\cC,\ve;I'}^{-1}(\De_1)\cap\cZ_{I';I}}=
\big\{\fD\Psi_{I';I}^{(\ve)}\big\}^*
\wt\om_{\cC;I}^{(\ve)}\big|_{\Phi_{\cC,\ve;I'}^{-1}(\De_1)\cap\cZ_{I';I}}\EE
for all $I\!\subsetneq\!I'\!\subset\![N]$ with $I\!\neq\!\eset$. 
\end{lmm}

\begin{proof}
We show that the claim holds with $\cC^2\!>\!2^N\cC_{\Phi,\ve}$, 
with $\cC_{\Phi,\ve}$ as in~\eref{PhiBnd_e}.
This assumption implies~that
\BE{sympform_e1}
\big\{i\!\in\!I\!:\rho_{I;i}(v_I)\!\le\!\frac34\big\}\neq\eset
\qquad\forall~v_I\!\in\!\cZ_I\!\cap\!\Phi_{\cC,\ve;I}^{-1}(\De_1),
~I\!\subset\!N,~|I|\!\ge\!2\,.\EE
In particular, there exists at most one~$i$ with a nonzero product 
in~\eref{wtomdfn_e}.
By the definition of~$\cZ_{I';I}$, $\rho_{I';j}\!\ge\!2$ on~$\cZ_{I';I}$ for all $j\!\in\!I'\!-\!I$
and~so
\BE{sympform_e2b}
\eta_{I';j}\big|_{\cZ_{I';I}}\equiv1 \qquad\forall~j\!\in\!I'\!-\!I\,.\EE
By~\eref{sympform_e1} and~\eref{sympform_e2b},
\BE{sympform_e5}\begin{split}
\sum_{i\in I'-I}\!\!\!\rho_{I';i}\al_{I';i}
+\sum_{i\in I}\bigg(1-\prod_{j\in I-\{i\}}\!\!\!\!\!\!\eta_{I';j}\bigg)
\rho_{I';i}\al_{I';i}
&=\sum_{i\in I'}\bigg(1-\prod_{j\in I'-\{i\}}\!\!\!\!\!\!\eta_{I';j}\bigg)
\rho_{I';i}\al_{I';i}\\
\hbox{and}\qquad
\sum_{i\in I}\prod_{j\in I-\{i\}}\!\!\!\!\!\!\eta_{I';j}
&=\sum_{i\in I'}\prod_{j\in I'-\{i\}}\!\!\!\!\!\!\eta_{I';j}
\end{split}\EE
on $\Phi_{\cC,\ve;I'}^{-1}(\De_1)\!\cap\!\cZ_{I';I}$.\\

\noindent
By \eref{SympSum_e17} with $I$ replaced by~$I'$ and~\eref{SympSum_e7},
\BE{sympform_e2}
\Psi_{I';I}^{(\ve)\,*}\om_I
=\bigg(\pi_{I';I}^*\om_{I'}+
\frac12\nd\!\! \sum_{i\in I'-I}\!\!\!\ve^2\rho_{I';i}\al_{I';i}\bigg)
\bigg|_{\cN_{I';I}(4^N)}\,.\EE
Suppose $|I|\!\ge\!2$.
The product Hermitian isomorphism $\fD\Psi_{I';I}^{(\ve)}$ satisfies
$$\pi_I\!\circ\!\fD\Psi_{I';I}^{(\ve)}=
\Psi_{I';I}^{(\ve)}\!\circ\!
\pi_{I';I}^*\pi_{I';I'-I}
\big|_{\pi_{I';I}^*\cN_{I';I'-I}|_{\cN_{I';I}(4^N)}}\,.$$
Along with~\eref{sympform_e2}, this implies~that  
\BE{sympform_e3a}
\big\{\fD\Psi_{I';I}^{(\ve)}\big\}^*\pi_I^*\om_I
=\bigg(\pi_{I'}^*\om_{I'}+
\frac12\nd\!\! \sum_{i\in I'-I}\!\!\!\ve^2\rho_{I';i}\al_{I';i}\bigg)
\bigg|_{\pi_{I';I}^*\cN_{I';I'-I}|\cN_{I';I}(4^N)}\,.\EE
Combining this with 
the second assumption in~\eref{SympSum_e15} with $I$ replaced by~$I'$ and
\eref{PhiPsi_e2}, we~obtain
\BE{sympform_e3}\begin{split}
\big\{\fD\Psi_{I';I}^{(\ve)}\big\}^*\wt\om_{\cC;I}^{(\ve)}
&=\pi_{I'}^*\om_{I'}+\frac12\nd\Bigg( 
\sum_{i\in I'-I}\!\!\!\ve^2\rho_{I';i}\al_{I';i}
+\sum_{i\in I}\bigg(1-\prod_{j\in I-\{i\}}\!\!\!\!\!\!\eta_{I';j}\bigg)
\ve^2\rho_{I';i}\al_{I';i}\\
&\hspace{2.1in}
+\sum_{i\in I}\bigg(\prod_{j\in I-\{i\}}\!\!\!\!\!\!\eta_{I';j}\bigg)
\ve^2\Phi_{\cC,\ve;I'}^*\big(r^2\nd\th\big)\Bigg).
\end{split}\EE
By \eref{sympform_e5}, the right-hand side of~\eref{sympform_e3} equals 
the right-hand side of~\eref{wtomdfn_e} with $I$ replaced by~$I'$
on $\Phi_{\cC,\ve;I'}^{-1}(\De_1)\!\cap\!\cZ_{I';I}$.
By~\eref{DPsiext_e} and~\eref{sympform_e2}, \eref{sympform_e3} also holds if $|I|\!=\!1$.
Thus, \eref{sympform_e} holds in this case as well.
\end{proof}

\noindent
By Lemma~\ref{sympform_lmm}, the 2-forms~$\wt\om_{\cC;I}^{(\ve)}$ induce
a closed 2-form~$\wt\om_{\cC}^{(\ve)}$ on~$\cZ'$. 
It remains to show that its restrictions to a neighborhood~$\cZ$ of~$\cZ_0'$ in~$\cZ'$
and to the fibers $\cZ\!\cap\!\pi_{\cC,\ve}^{-1}(\la)$  with $\la\!\in\!\C^*$ are nondegenerate.
The next lemma, which follows immediately from Corollary~\ref{conn_crl}
and~\eref{SympSum_e27}, is used to verify
that this is the case in 
the ``middle" region of each domain.

\begin{lmm}\label{Phipull_lmm}
For each $I\!\!\subset\![N]$ with $|I|\!\ge\!2$,
there exist $f_I\!\in\!C^{\i}(X_I^{\star};\R^+)$
and an $\R$-valued 1-form~$\mu_I$ on~$X_I^{\star}$ such~that 
\BE{Phipull_e}
\Phi_{\cC,\ve;I}^*r^2= \big(\cC\ve^{|I|-1}\big)^2 f_I\prod_{i\in I}\rho_{I;i},
~~
\Phi_{\cC,\ve;I}^*\nd\th = \sum_{i\in I}\!\al_{I;i}+\pi_I^*\mu_I
\quad\hbox{on}~~\cN X_I|_{X_I^{\star}}\EE
for all functions $\ve,\cC\!:X_{\eset}\!\lra\!\R^+$ that restrict to smooth functions
on~$X_I$.
\end{lmm}

\begin{prp}\label{sympform_prp}
Let  $\ve,\cC\!:X_{\eset}\!\lra\!\R^+$ be as in~\eref{SympSum_e15} and~\eref{SympSum_e25}.
Suppose $\ve|_{X_i},\cC|_{X_i}$ are smooth for all $i\!\in\![N]$,
$\ve$ is sufficiently small, and $\cC$ is sufficiently large (depending on~$\ve$).
Then there exists a neighborhood~$\cZ$ of~$\cZ_0'$ in~$\cZ'$ 
such that the restrictions of $\wt\om_{\cC}^{(\ve)}$ to~$\cZ$
and to $\cZ\!\cap\!\pi_{\cC,\ve}^{-1}(\la)$  with $\la\!\in\!\C^*$
are nondegenerate. 
\end{prp}

\begin{proof}
We need to show that every $x\!\in\!X_I$ with $I\!\in\!\cP^*(N)$
has a neighborhood~$\cZ_x$ in~$\cZ'$ such that 
the restrictions of $\wt\om_{\cC}^{(\ve)}$ to~$\cZ_x$
and to $\cZ_x\!\cap\!\pi_{\cC,\ve}^{-1}(\la)$  with $\la\!\in\!\C^*$ are nondegenerate. 
Since $X_I$ is contained in~$\cZ_0'$, 
it is sufficient to show that for each $x\!\in\!\ri{X}_I$ 
there exist a neighborhood~$U_x$ of~$x$ in~$\ri{X}_I$ and $r_x\!\in\!\R^+$
such that the restrictions of $\wt\om_{\cC;I}^{(\ve)}$ to 
$\cZ_I|_{U_x}\!\cap\!\Phi_{\cC,\ve;I}^{-1}(\De_{r_x})$ and to
$\cZ_I|_{U_x}\!\cap\!\Phi_{\cC,\ve;I}^{-1}(\la)$ with $\la\!\in\!\De_{r_x}$
are nondegenerate.
Since the 2-form~\eref{wtomdfn_e0} and its restriction to a fiber of~$\Phi_{\cC,\ve;i}$ 
are symplectic over~$X_i\!\times\!\C$ (even if $\ve$ is not constant), 
it remains to consider the case $|I|\!\ge\!2$.\\

\noindent
For each $i\!\in\!I$ with $|I|\!\ge\!2$, let $\ka_{I;i}\!\in\!\Om^2(TX_I)$
be the curvature form of~$\al_{I;i}$.
Thus,
$$\wh\om_I=\pi_I^*\om_I+\frac12\sum_{i\in I}\!
\big(\rho_{I;i}\pi_I^*\ka_{I;i}\!+\!
\pr_{I;I-i}^*(\nd\rho_{I;i}\!\w\!\al_{I;i})\!\big).$$
Let $\wt{J}_I$ denote the almost complex  structure on the total space 
of $\cN X_I\!\lra\!X_I$ induced by an $\om$-tame almost complex structure~$J_I$
on~$X_I$ via the product Hermitian structure $(\rho_{I;i},\na^{(I;i)})_{i\in I}$;
see the paragraph above Lemma~\ref{almostJ_lmm}.
If $\ve$ as in~\eref{SympSum_e15} is sufficiently small, 
then $\wh\om_I$ tames~$\wt{J}_I$ over $\cN X_I(4^N\ve^2)$. 
Thus, $\wh\om_I^{(\ve)}$ tames $\wt{J}_{\ve;I}\!\equiv\!m_{\ve;I}^*\wt{J}_I$ 
over $\cN X_I(4^N)$.\\

\noindent
Let $\la\!\in\!\C^*$ and $v_I\!\in\!\cZ_I\!\cap\!\Phi_{\cC,\ve;I}^{-1}(\la)$.
Define
$$I_0=\big\{i\!\in\!I\!:\,\rho_{I;i}(v_I)\!\le\!\frac34\big\},\quad
I_{01}=\big\{i\!\in\!I\!:\,\frac34<\rho_{I;i}(v_I)\!<\!1\big\},\quad
I_1=\big\{i\!\in\!I\!:\,\rho_{I;i}(v_I)\!\ge\!1\big\}.$$
If $|\la|\!<\!1$ and $\cC^2\!>\!2^N\cC_{\Phi,\ve}$, with $\cC_{\Phi,\ve}$
as in~\eref{PhiBnd_e}, then $I_0\!\neq\!\eset$; see~\eref{sympform_e1}.\\

\noindent
Suppose $|I_0|\!\ge\!2$.
By~\eref{wtomdfn_e}, $\wt\om_{\cC;I}^{(\ve)}$ at $v_I$ is then given~by 
$$\wt\om_{\cC;I}^{(\ve)}=\pi_I^*\om_I+
\frac12 \sum_{i\in I}\nd\big(\ve^2\rho_{I;i}\al_{I;i}\big)
\equiv \wh\om_I^{(\ve)}\,.$$
This form tames $\wt{J}_{\ve;I}$ on $T_{v_I}\cZ_I$ and thus is
nondegenerate on this space.
By Corollary~\ref{almostJ_crl}, the restriction of this form to 
$T_{v_I}\Phi_{\cC,\ve;I}^{-1}(\la)$ is also nondegenerate if $|\la|$ is
sufficiently small (depending only on~$\pi_I(v_I)$).\\

\noindent
Suppose $|I_0|\!=\!1$ and $I_{01}\!=\!\eset$. 
By~\eref{wtomdfn_e}, $\wt\om_{\cC;I}^{(\ve)}$ at $v_I$ is then given~by
$$\wt\om_{\cC;I}^{(\ve)}
=\pi_I^*\om_I+ \frac12\nd\sum_{i\in I_1}\ve^2\rho_{I;i}\al_{I;i}
+\ve^2\Phi_{\cC,\ve;I}^*\big(r^2\nd\th\big)\,.$$
Along with \eref{sympform_e3a} with $I\!\subset\!I'$ replaced by $I_0\!\subset\!I$, 
the second assumption in~\eref{SympSum_e15}, \eref{PhiPsi_e2} with $I'\!=\!I_0$, 
and~\eref{PhiSimpDfn_e}, this implies~that
$$\wt\om_{\cC;I}^{(\ve)}\big|_{v_I}
=\bigg(\big\{\fD\Psi_{I;I_0}^{(\ve)}\big\}^*
\Big(\pi_1^*\om_{I_0}+\frac12\nd\big(\ve^2\pi_2^*(r^2\nd\th)\big)\!\!\Big)
\!\bigg)_{\!v_I}\,.$$
This form is nondegenerate, since $\fD\Psi_{I;I_0}^{(\ve)}$ is a diffeomorphism on
a neighborhood of~$v_I$ contained in the set~\eref{diffsub_e} with $I'\!=\!I_0$.
By~\eref{PhiPsi_e2} with $I'\!=\!I_0$, the restriction of this form to 
$T_{v_I}\Phi_{\cC,\ve;I}^{-1}(\la)$ is also nondegenerate.\\

\noindent
Suppose $I_0\!=\!\{i_0\}$ is a single-element set and $I_{01}\!\neq\!\eset$.
By~\eref{wtomdfn_e}, $\wt\om_{\cC;I}^{(\ve)}$ at $v_I$ is then given~by
$$\wt\om_{\cC;I}^{(\ve)}=\pi_I^*\om_I+
\frac12
\nd\bigg( \sum_{i\in I_{01}\cup I_1}\!\!\!\!\!\ve^2\rho_{I;i}\al_{I;i}
+\Big(1-\prod_{i\in I_{01}}\!\!\!\eta_{I;i}\Big)
\ve^2\rho_{I;i_0}\al_{I;i_0}
+\Big(\prod_{i\in I_{01}}\!\!\eta_{I;i}\Big)
\ve^2\Phi_{\cC,\ve;I}^*\big(r^2\nd\th\big)\!\!\bigg).$$
Since $\fD\Psi_{I;I_0\cup I_{01}}^{(\ve)}$ is a product Hermitian isomorphism,
\eref{sympform_e3a}
 with $I\!\subset\!I'$ replaced by $I_0\!\cup\!I_{01}\!\subset\!I$,
the second assumption in~\eref{SympSum_e15},  and~\eref{PhiPsi_e2} 
with $I'\!=\!I_0\!\cup\!I_{01}$ thus imply~that
\begin{equation*}\begin{split}
\wt\om_{\cC;I}^{(\ve)}\big|_{v_I}
&=\big\{\fD\Psi_{I;I_0\cup I_{01}}^{(\ve)}\!\big\}^*\Bigg(\pi_{I_0\cup I_{01}}^*\om_{I_0\cup I_{01}}+
\frac12\nd\bigg(
\sum_{i\in I_{01}}\!\!\rho_{I_0\cup I_{01};i}^{(\ve)}\al_{I_0\cup I_{01};i}\\
&\hspace{.3in}+\Big(1-\prod_{i\in I_{01}}\!\!\!\eta_{I_0\cup I_{01};i}\Big)
\rho_{I_0\cup I_{01};i_0}^{(\ve)}\al_{I_0\cup I_{01};i_0}
+\Big(\prod_{i\in I_{01}}\!\!\eta_{I_0\cup I_{01};i}\Big)
\ve^2\Phi_{\cC,\ve;I_0\cup I_{01}}^*\big(r^2\nd\th\big)
\bigg)\!\!\!\Bigg)_{\!\!v_I}.
\end{split}\end{equation*}
Since $\fD\Psi_{I;I_0\cup I_{01}}^{(\ve)}$ is a diffeomorphism satisfying 
\eref{PhiPsi_e2} with $I'\!=\!I_0\!\cup\!I_{01}$, 
we can therefore assume that $I_1\!=\!\eset$.\\

\noindent
With $f_I$ and $\mu_I$ as in Lemma~\ref{Phipull_lmm}, let
\begin{gather*}
f_{\cC;I}=\cC^2f_I, \quad 
\wt{f}_{I;i_0}=\Big(1-\!\!\prod_{i\in I_{01}}\!\!\!\eta_{I;i}\!\Big)
+f_{\cC;I} \Big(\prod_{i\in I_{01}}\!\!\eta_{I;i}\rho_{I;i}^{(\ve)}\Big),\quad
\wt{f}_{I;i}=f_{\cC;I}\eta_{I;i}\!\!\!\!\!
\prod_{j\in I_{01}-i}\!\!\!\!\!\!\eta_{I;j}\rho_{I;j}^{(\ve)}
~\forall\,i\!\in\!I_{01},\\
\be_{I;i_0}=\nd 
\bigg(f_{\cC;I}\!\!\!\prod_{i\in I_{01}}\!\!\eta_{I;i}\rho_{I;i}^{(\ve)}
-\!\!\prod_{i\in I_{01}}\!\!\eta_{I;i}\bigg),\qquad
\wt\mu_I=f_{\cC;I}
\bigg(\prod_{i\in I_{01}}\!\!\!\eta_{I;i}\rho_{I;i}^{(\ve)}\bigg)\pi_I^*\mu_I\,.
\end{gather*}
Fix a Riemannian metric on~$X_I$.
Via the Hermitian data $(\rho_{I;i},\na^{(I;i)})_{i\in I}$, it lifts to 
a Riemannian metric on~$\cN X_I$. 
By the first statement in~\eref{Phipull_e} 
and the definition of~$I_{01}$,  
there exists a smooth function \hbox{$\cC_I\!:X_I\!\lra\!\R^+$}, dependent only on~$\ve$ and~$\cC$,
such~that 
\begin{gather}
\label{topbnd_e}
\frac{1}{\cC_I(\pi_I(v_I))}\le \big|\wt{f}_{I;i_0}\big|_{v_I}
\le \cC_I\big(\pi_I(v_I)\big),\\
\label{botbnd_e}
\rho_{I;i_0}^{(\ve)}(v_I)\le \cC_I\big(\pi_I(v_I)\big)|\la|^2, \quad
\big|\nd\rho_{I;i_0}^{(\ve)}\big|_{v_I},
\big|\rho_{I;i_0}^{(\ve)}\al_{I;i_0}\big|_{v_I}\le \cC_I\big(\pi_I(v_I)\big)|\la|,\\
\label{botbnd_e2}
\big|\nd\rho_{I;i}^{(\ve)}\big|_{v_I},\big|\rho_{I;i}^{(\ve)}\al_{I;i}\big|_{v_I},
\big|\be_{I;i_0}(v_I)\big|,
\big|\wt{f}_{I;i}(v_I)\big|,\big|\nd\wt{f}_{I;i}\big|_{v_I},
\big|\wt\mu_I\big|_{v_I},\big|\nd\wt\mu_I\big|_{v_I}
\le \cC_I\big(\pi_I(v_I)\big);
\end{gather}
these bounds apply to any $v_I\!\in\!\cZ_I$ with $I_0\!=\!\{i_0\}$ and $I_1\!=\!\eset$.\\

\noindent
By~\eref{Phipull_e}, 
$$\Big(\prod_{i\in I_{01}}\!\!\eta_{I;i}\Big)
\ve^2\Phi_{\cC,\ve;I}^*\big(r^2\nd\th\big)
=f_{\cC;I}\Big(\prod_{i\in I_{01}}\!\!\eta_{I;i}\rho_{I;i}^{(\ve)}\Big)
\big(\rho_{I;i_0}^{(\ve)}\al_{I;i_0}\big)
+\sum_{i\in I_{01}}\!\rho_{I;i_0}^{(\ve)}\wt{f}_{I;i}\big(\rho_{I;i}^{(\ve)}\al_{I;i}\big) 
+\rho_{I;i_0}^{(\ve)}\wt\mu_I\,.$$
Thus, $\wt\om_{\cC;I}^{(\ve)}$ at $v_I$ is given~by
\begin{alignat}{1}
\notag
\wt\om_{\cC;I}^{(\ve)}
=&\pi_I^*\om_I+\frac12\sum_{i\in I_{01}}\!\!\nd(\rho_{I;i}^{(\ve)}\al_{I;i})
+\frac12\wt{f}_{I;i_0}
\nd(\rho_{I;i_0}^{(\ve)}\al_{I;i_0})\\
\label{omIexp_e}
&+\frac12\sum_{i\in I_{01}}\!\! 
\Big(\rho_{I;i_0}^{(\ve)}\wt{f}_{I;i}\nd(\rho_{I;i}^{(\ve)}\al_{I;i})+
\big(\wt{f}_{I;i}\nd\rho_{I;i_0}^{(\ve)}+
\rho_{I;i_0}^{(\ve)}\nd\wt{f}_{I;i}\big)
\!\w\!(\rho_{I;i}^{(\ve)}\al_{I;i})\Big)\\
&+\frac12\be_{I;i_0}\!\w\!(\rho_{I;i_0}^{(\ve)}\al_{I;i_0})
+\frac12\big(\nd\rho_{I;i_0}^{(\ve)}\!\w\!\wt\mu_I+
\rho_{I;i_0}^{(\ve)}\nd\wt\mu_I\big)\,.
\notag\end{alignat}

\vspace{.2in}

\noindent
The 2-form~$\wt\om_{\tn{top}}^{(\ve)}$ on the first line on 
the right-hand side of~\eref{omIexp_e} is the pullback by~$m_{\ve;I}$ of the 2-form
$$\pi_I^*\om_I+\frac12\sum_{i\in I_{01}}\!\!\nd(\rho_{I;i}\al_{I;i})
+\frac12\wt{f}_{I;i_0}\!\circ\!m_{\ve;I}^{-1}\,\nd\big(\rho_{I;i_0}\al_{I;i_0}\big)$$
at $m_{\ve;I}(v_I)$.
By the $\wh\om_I$-tameness of $\wt{J}_I$ on $\cN X_I(4^N)$ and~\eref{topbnd_e}, 
there exists a smooth function
\hbox{$\ve_{I;i_0}\!:X_I\!\lra\!\R^+$}, dependent only on~$\ve$ and~$\cC$, 
so that this form tames~$\wt{J}_I$ on
$$\big\{v\!\in\!\cN X_I(3^N\ve^2)\!:
\rho_{I;i_0}(v)\!\le\!\ve_{I;i_0}\big(\pi_I(v)\big)\big\}\,.$$
The 2-form~$\wt\om_{\tn{top}}^{(\ve)}$ then tames $\wt{J}_{\ve;I}$ on
$$W_{I;i_0}\equiv
\big\{v\!\in\!\cN X_I(3^N)\!:
\rho_{I;i_0}^{(\ve)}(v)\!\le\!\ve_{I;i_0}\big(\pi_I(v)\big)\big\}\,.$$
By Corollary~\ref{almostJ_crl}, 
for every $x\!\in\!\ri{X}_I$ there exist a precompact neighborhood $U_x$ and $r_x\!\in\!\R^+$
such that the restriction of~$\wt\om_{\tn{top}}^{(\ve)}$ to $T_v\Phi_{\cC,\ve;I}^{-1}(\la)$  
is nondegenerate for all $v\!\in\!W_{I;i_0}|_{\ov{U_x}}\!\cap\!\Phi_{\cC,\ve;I}^{-1}(\la)$
and $\la\!\in\!\De_{r_x}$.
There thus exists $r_x'\!\in\!(0,r_x)$ such~that 
the right-hand side of~\eref{omIexp_e} is nondegenerate on $T_v\cZ_I$ and
on $T_v\Phi_{\cC,\ve;I}^{-1}(\la)$
for all 
$v\!\in\!W_{I;i_0}|_{\ov{U_x}}\!\cap\!\Phi_{\cC,\ve;I}^{-1}(\De_{r_x'})$
and $\la\!\in\!\De_{r_x'}$ satisfying the bounds in~\eref{botbnd_e} and~\eref{botbnd_e2}. 
We conclude that $\wt\om_{\cC;I}^{(\ve)}|_{v_I}$ and the restriction $\wt\om_{\cC;I}^{(\ve)}|_{v_I}$
to $T_{v_I}\Phi_{\cC,\ve;I}^{-1}(\la)$ are nondegenerate if $\la\!\in\!\De_{r_x'}$.
\end{proof}

\begin{crl}\label{sympform_crl}
Let  $\ve,\cC\!:X\!\lra\!\R^+$ be
as in~\eref{SympSum_e15} and~\eref{SympSum_e25}.
Suppose $\ve|_{X_i},\cC|_{X_i}$ are smooth for all $i\!\in\![N]$,
$\ve$ is sufficiently small, and $\cC$ is sufficiently large (depending on~$\ve$).
Then $\io_{\cC,\ve;i}^*\wt\om_{\cC}^{(\ve)}\!=\!\om_i$ for all $i\!\in\![N]$
and  $\{\io_{\cC,\ve;i}(X_i)\}_{i\in[N]}$ is an SC symplectic divisor
in $(\cZ,\wt\om_{\cC}^{(\ve)})$ for some neighborhood~$\cZ$ of~$\cZ_0'$ in~$\cZ'$.
\end{crl}

\begin{proof}
By the first case in~\eref{ioZdfn_e} and~\eref{wtomdfn_e0},
$$\io_{\cC,\ve;i}^*\wt\om_{\cC}^{(\ve)}\big|_{\ri{X}_i}=\om_i|_{\ri{X}_i}
~~\forall~i\!\in\![N]\,.$$
By the second in~\eref{ioZdfn_e},  \eref{wtomdfn_e}, and 
the vanishing of $\eta_{I;i}$, $\rho_{I;i}\al_{I;i}$, and $\Phi_{\cC,\ve;I}$
along $X_{I;i}\!\subset\!\cN_{I;i}$,
$$\big\{\Psi_{I;i}^{(\ve)}\big\}^*\io_{\cC,\ve;i}^*\wt\om_{\cC}^{(\ve)}|_{X_{I;i}}=
\Bigg(\!\!\pi_I^*\om_I+
\frac12
\nd\bigg(\!\sum_{j\in I-i}\!\!\ve^2\rho_{I;j}\al_{I;j}\bigg)\!\!\Bigg)\bigg|_{X_{I;i}}
~~\forall~\{i\}\!\subsetneq\!I\!\subset\![N].$$
The first claim of the corollary now follows from~\eref{SympSum_e17}.\\

\noindent
By Proposition~\ref{cZtopol_prp}, $\{\io_{\cC,\ve;i}(X_i)\}_{i\in[N]}$ is
a transverse collection of closed submanifolds of~$\cZ'$ of codimension~2.
By Proposition~\ref{sympform_prp} and the first claim of the corollary, 
the intersections of these submanifolds
are symplectic submanifolds of $(\cZ,\wt\om_{\cC}^{(\ve)})$ 
for some neighborhood~$\cZ$ of~$\cZ_0'$ in~$\cZ'$.
The preimages of these intersections in~$\cZ_I$ correspond
to the intersections of the hyperplane subbundles $\cN_{I;i}\!\subset\!\cN X_I$
on a neighborhood of~$\ri{X}_I$.
These hyperplanes are $\wt\om_{\cC;I}^{(\ve)}$-orthogonal; 
see~\eref{TrivExt1_e5}.
Thus, the symplectic orientations on neighborhoods of~$\ri{X}_I$
in the corresponding subbundles $\cN_{I;I'}\!\subset\!\cN X_I$
are the same as the intersection orientations induced from 
the hyperplane  subbundles $\cN_{I;i}\!\subset\!\cN X_I$.
This establishes the last claim of the corollary.
\end{proof}

\begin{rmk}\label{noncpmt_rmk}
We believe that the dependence of the maximal norm~$r_x$ for the acceptable values of 
\hbox{$\la\!\in\!\C$} on~$x$
in the proof of Proposition~\ref{sympform_prp} can be dropped
by constructing the trivializations $(\wt\Phi_t)_{t\in B}$
in Proposition~\ref{SCC_prp} along with constructing 
the regularizations~$(\fR_t)_{t\in B}$ in the proof of \cite[Theorem~2.17]{SympDivConf}.
This would then lead to a fibration~$\cZ$ with uniform smooth fibers over
a neighborhood~$\De$ of~0 in~$\C$ and remove the compactness assumption from
the last statement of Theorem~\ref{SympSum_thm}.
\end{rmk}

\section{The smoothability criterion}
\label{SympSumTriv_sec}

\noindent
We show in Section~\ref{SympSumConv_subs} that a nearly regular symplectic fibration 
$(\cZ,\om_{\cZ},\pi)$ as in Definition~\ref{SimpFibr_dfn}
determines a homotopy class of trivializations of 
the associated complex line bundle~$\cO_{\cZ}(\cZ_0)$;
see Proposition~\ref{InducedIsom_prp}.
If $(\cZ,\om_{\cZ},\pi)$ is a one-parameter family of smoothings of an SC symplectic variety 
$(X_{\eset},(\om_i)_{i\in[N]})$ as in Definition~\ref{SCC_dfn}, 
this line bundle restricts to~\eref{PsiDfn_e} over the singular locus~$X_{\prt}$.
Proposition~\ref{InducedIsom_prp} then determines a homotopy class of trivializations 
of~\eref{PsiDfn_e}.
In particular, \eref{SympSumCond_e} is a necessary condition 
for an SC symplectic variety to be smoothable.
By Proposition~\ref{PhiExt_prp} proved in Section~\ref{TrivExt_subs}, the homotopy class of trivializations determined by a one-parameter family $(\cZ,\wt\om_{\cC}^{(\ve)},\pi_{\cC,\ve})$
of smoothings
constructed in Section~\ref{SympSumPf_sec} is the homotopy class used to construct this~family.\\

\noindent
Proposition~\ref{InducedIsom_prp} and its proof readily extend to families 
of nearly regular symplectic fibrations parametrized by a manifold~$B$.
They endow a natural complex line bundle $\cO_{B;\cZ}(\cZ_0)$ over the total space
of such a family with a canonical homotopy class of trivializations. 
Proposition~\ref{PhiExt_prp} and its proof extend directly to compact families 
of the relevant data on $(X_I)_{I\in\cP^*(N)}$.
They ensure that the family of one-parameter families of smoothings then arising
from the construction of Section~\ref{SympSumPf_sec} encodes the homotopy class
of trivializations used to obtain this family.

\subsection{The necessity of~\eref{SympSumCond_e}}
\label{SympSumConv_subs}

\noindent
If $V\!\subset\!X$ is a smooth symplectic divisor,
the line bundle $\cO_X(V)$ has a canonical section~$s_V$ with zero set~$V$;
it is described similarly to~\eref{cOXDsec_e}.
Thus, any tensor product of such line bundles also has a canonical section;
its  zero set is the union of the associated symplectic divisors.

\begin{prp}\label{InducedIsom_prp}
Let $(\cZ,\om_{\cZ},\pi)$ be a nearly regular symplectic fibration over~$\De$
as in Definition~\ref{SimpFibr_dfn}
and $s_{\eset}$ be the canonical section of the complex line bundle
$$\cO_{\cZ}(\cZ_0)\equiv\bigotimes_{i\in[N]}\!\!\cO_{\cZ}(X_i)\lra \cZ.$$
Then there exists a trivialization~$\Phi_{\cZ}$ of $\cO_{\cZ}(\cZ_0)$ such~that 
the smooth maps
$$\Phi_{\cZ}\!\circ\!s_{\eset}|_{\cZ-\cZ_0},\pi|_{\cZ-\cZ_0}\!: \cZ\!-\!\cZ_0\lra\C^*$$
are homotopic.
Furthermore, any two such trivializations~$\Phi_{\cZ}$ are homotopic.
\end{prp}

\begin{lmm}\label{OrientBase_lmm}
Let $(\cZ,\om_{\cZ},\pi)$ be as in Proposition~\ref{InducedIsom_prp}
and $\cZ_{0;\prt}\!\subset\!\cZ_0$ be the singular locus of~$\cZ_0$.
If $\cZ_0$ is connected, there exists a smooth embedding $(\De,0)\!\lhra\!(\C,0)$
such that the linear~map
\BE{OrientBase_e}D\pi\!: 
\cN_{\cZ}\big(\cZ_0\!-\!\cZ_{0;\prt}\big)\lra T_0\De \cong T_0\C \cong \C\EE
induced by $\nd\pi$ is an orientation-preserving trivialization of 
$\cN_{\cZ}(\cZ_0\!-\!\cZ_{0;\prt})$.
\end{lmm}

\begin{proof} Choose a point $p\!\in\!\cZ_0\!-\!\cZ_{0;\prt}$ and 
a smooth embedding $(\De,0)\!\lhra\!(\C,0)$ such that the isomorphism  
$$D_p\pi\!:  \cN_{\cZ}(\cZ_0\!-\!\cZ_{0;\prt})\big|_p \lra T_0\De \cong T_0\C \cong \C$$ 
is orientation-preserving.
Let $\cZ^+,\cZ^-\!\subset\!\cZ\!-\!\cZ_{0;\prt}$ be the subspaces of points~$x$ 
such that the isomorphism
\BE{OrientBase_e3}
D_x\pi\!:  \cN_{\cZ}\big(\cZ_{\pi(x)}\!-\!\cZ_{0;\prt}\big)\big|_x 
\lra T_{\pi(x)}\De \cong T_{\pi(x)}\C\EE
is orientation-preserving and orientation-reversing, respectively.
By assumption, $p\!\in\!\cZ^+$.
Since $\cZ_0$ is connected and $\cZ_{0;\prt}$ consists of codimension~4 submanifolds
of~$\cZ$, the complement~$\cZ^*$ of~$\cZ_{0;\prt}$ in the topological component of~$\cZ$
containing~$\cZ_0$ is connected as well.
Since the disjoint open subsets $\cZ^+$ and~$\cZ^-$ cover~$\cZ^*$,
$\cZ^*\!\subset\!\cZ^+$ and thus~\eref{OrientBase_e3} is orientation-preserving for 
all $x\!\in\!\cZ_0\!-\!\cZ_{0;\prt}$.
\end{proof}

\begin{proof}[{\bf{\emph{Proof of Proposition~\ref{InducedIsom_prp}}}}]
If $s_1$ and $s_2$ are non-vanishing sections of a complex line bundle~$L$
over some space~$X$, we write $s_1\!\sim\!s_2$ if
$s_1$ and $s_2$ are homotopic through non-vanishing sections of~$L$.
A trivialization~$\Phi$ of~$L$ 
corresponds to a non-vanishing section of~$L$.
The equivalence classes of non-vanishing sections of~$L$ 
correspond to the homotopy classes of trivializations of~$L$.
The first claim of the proposition is equivalent to the existence of 
a non-vanishing section~$s$ of $\cO_{\cZ}(\cZ_0)$ so that the smooth~maps
\BE{InducedIsom_e1}(s_\eset/s)|_{\cZ-\cZ_0},\pi|_{\cZ-\cZ_0}\!: \cZ\!-\!\cZ_0\lra\C^*\EE
are homotopic.\\

\noindent
We can assume that $\cZ_0$ is connected.
Let $(\De,0)\!\lhra\!(\C,0)$ be an embedding as in Lemma~\ref{OrientBase_lmm},
$$\cN=\cN_{\cZ}(\cZ_0\!-\!\cZ_{0;\prt}), \qquad \cN_i=\cN_{\cZ}X_i\,,$$ 
and $U\!\subset\!\cZ$ be an open subset such that the inclusions
$$\cZ\!-\!U\lra \cZ\!-\!\cZ_{0;\prt} \qquad\hbox{and}\qquad
\cZ\!-\!(\cZ_0\!\cup\!U) \lra \cZ\!-\!\cZ_0 $$
are homotopy equivalences.
Denote by $\Psi_{\De}\!: \wt\De\!\lra\!\De$ the canonical identification
of a neighborhood~$\wt\De$ of~$0$ in $T_0\De\!=\!T_0\C$ with $\De\!\subset\!\C$
and by~$\fI$ the standard complex structure on~$T_0\De$.
Since the restriction of~\eref{OrientBase_e} to each fiber of~$\cN$ is orientation-preserving,
the complex structure $\wt\fI\!\equiv\!\{D\pi\}^*\fI$ is tamed by~$\om_{\cZ}|_{\cN}$.\\

\noindent
Let $\Psi_i\!: \cN_i'\!\lra\!\cZ$, with $i\!\in\![N]$, be regularizations of $X_i$ in~$\cZ$ in
the sense of Definition~\ref{smreg_dfn} such~that 
\begin{gather}
\notag
\Im\big(\Psi_i|_{\cN_i'|_{X_i-U}}\big)\cap\Im\big(\Psi_j|_{\cN_j'|_{X_j-U}}\big)=\eset 
\quad\forall~i,j\!\in\![N],\, i\!\neq\!j, \\
\label{projectiontrivializaion_eqn}
\pi\!\circ\!\Psi_i\!=\!\Psi_{\De}\!\circ\!D\pi\!\big|_{\cN'_i|_q}\!:\,
\cN'_i|_q \lra \De\!\subset\!\C \quad\forall~q\!\in\!X_i\!-\!U.
\end{gather}
We can use these regularizations and the complex structure~$\wt\fI$
to construct the line bundles $\cO_{\cZ}(X_i)$ as in (\ref{cOXVdfn_e}).
The canonical regularization~$\Psi_{\De}$ of~$0$ in~$\C$ and 
the standard complex structure~$\fI$ on~$T_0\De$
similarly determine a line bundle $\cO_{\De}(0)\!\lra\!\De$.
By \eref{projectiontrivializaion_eqn}, $D\pi$ induces an isomorphism
\begin{equation*}
\Phi\!:  \cO_{\cZ}(\cZ_0)\big|_{\cZ-U} \lra \pi^*\cO_{\De}(0)\big|_{\cZ-U},
\quad
\begin{aligned} 
\big[x,c\big]&\lra\big[\pi(x),c\big];\\
\big[\Psi_i(v),v,w\big]&\lra\big[\pi(\Psi_i(v)),D\pi(v),D\pi(w)\big];
\end{aligned}\end{equation*}
see \eref{cOXVdfn_e} for the notation.\\

\noindent
Let $s_0$ be the canonical section of $\cO_{\De}(0)$
and $s_1$ be the constant section~1 of $\cO_{\De}(0)$, i.e.
$$s_1\big(\Psi_{\De}(u)\big)=\big[\Psi_{\De}(u),u,1\big]
\qquad\forall\,u\!\in\!\wt\De.$$
In particular, $s_0/s_1\!: \De\!\lra\!\C$ is the inclusion map.
They lift to sections  $\pi^*s_0$ and $\pi^*s_1$ of $\pi^*\cO_{\De}(0)$ so~that 
\BE{pistarrat_e}\pi^*s_0/\pi^*s_1\!=\!\pi\!: \cZ\lra\C.\EE
By~\eref{projectiontrivializaion_eqn}, 
\BE{seset_e}s_{\eset}|_{\cZ-U}=\Phi^{-1}\!\circ\!\pi^*s_0\big|_{\cZ-U}\,.\EE
Since the inclusion $\cZ\!-\!U\!\lra\!\cZ\!-\!\cZ_{0;\prt}$ is a homotopy equivalence,
there exists a non-vanishing section $s'$ of $\cO_{\cZ}(\cZ_0)|_{\cZ-\cZ_{0;\prt}}$
such~that  
\BE{restriction1_eqn} 
s'|_{\cZ-U}\sim\Phi^{-1}\!\circ\!\pi^*s_1\big|_{\cZ-U}\,.\EE
By Corollary~\ref{LBisomExt_crl}\ref{LBisomExt2_it} below, 
there exists a non-vanishing section~$s$ of $\cO_{\cZ}(\cZ_0)$ so that
\BE{restriction2_eqn} s|_{\cZ-\cZ_{0;\prt}} \sim s'.\EE
By~\eref{seset_e}, \eref{restriction2_eqn}, \eref{restriction1_eqn},
and~\eref{pistarrat_e}, 
$$\big(s_{\eset}/s\big)\big|_{\cZ-(\cZ_0\cup U)}
\sim \Phi^{-1}\!\circ\!\pi^*s_0\big/\Phi^{-1}\!\circ\!\pi^*s_1
\big|_{\cZ-(\cZ_0\cup U)}=\pi\big|_{\cZ-(\cZ_0\cup U)}\,.$$
Since the inclusion $\cZ\!-\!(\cZ_0\!\cup\!U)\!\lra\!\cZ\!-\!\cZ_0$
is a homotopy equivalence, the maps~\eref{InducedIsom_e1} are homotopic.
This establishes the first claim of the proposition.\\

\noindent
The section~$s_{\eset}$ of  $\cO_{\cZ}(\cZ_0)$  does not vanish on $\cZ\!-\!\cZ_0$.
If $\Phi_{\cZ}$ and $\Phi_{\cZ}'$ are trivializations of $\cO_{\cZ}(\cZ_0)$ 
satisfying the homotopy condition in the proposition, then
$\Phi_{\cZ}|_{\cZ-\cZ_0}$ and $\Phi_{\cZ}'|_{\cZ-\cZ_0}$ are homotopic trivializations
of $\cO_{\cZ}(\cZ_0)|_{\cZ-\cZ_0}$.
By Corollary~\ref{LBisomExt_crl}\ref{LBisomExt1_it}, 
$\Phi_{\cZ}$ and $\Phi_{\cZ}'$ are thus homotopic trivializations of~$\cO_{\cZ}(\cZ_0)$.
\end{proof}

\noindent
It remains to establish the two statements of Corollary~\ref{LBisomExt_crl} used above.

\begin{lmm}\label{FibLmm_lmm}
Let $B$ be a paracompact topological space and $(E,\dot{E})\!\lra\!B$
be a relative bundle pair with fiber pair $(F,\dot{F})$ in the sense of 
\cite[Section~5.7]{Spanier}.
If $\fc\!\in\!\Z^+$ and $H_i(F,\dot{F};\Z)\!=\!0$ for all $i\!<\!\fc$,
then $H^i(E,\dot{E};\Z)\!=\!0$ for all $i\!<\!\fc$.
\end{lmm}

\begin{proof}
By our assumptions and the Kunneth formula \cite[Theorem~5.3.10]{Spanier},
$$H_i\big(B\!\times\!F,B\!\times\!\dot{F};\Z\big)=0 \qquad\forall\,i\!<\!\fc.$$
By Mayer-Vietoris \cite[Corollary~5.1.14]{Spanier} and induction, this implies that 
$H_i(E,\dot{E};\Z)\!=\!0$ for all $i\!<\!\fc$ if $(E,\dot{E})\!\lra\!B$
admits a finite trivializing open cover.
Taking the direct limit over compact subsets of~$B$, we find that
$H_i(E,\dot{E};\Z)\!=\!0$ for all $i\!<\!\fc$ for any paracompact space~$B$.
The claim now follows from the Universal Coefficients Theorem \cite[Theorem~53.1]{Mu2}.
\end{proof}

\begin{crl}\label{LBisomExt_crl}
Suppose $M$ is a manifold, $\pi\!: L\!\lra\!M$ is a Hermitian line bundle,
and $M'\!\subset\!M$ is the complement of closed submanifolds $V_1,\cdots,V_{\ell}\!\subset\!M$ 
of real codimension $\fc\!\in\!\Z^+$ or higher.
\begin{enumerate}[label=(\arabic*),leftmargin=*]

\item\label{LBisomExt1_it}
If $\fc\!\ge\!2$, any two trivializations of~$L$ over~$M$ that restrict to 
homotopic trivializations of~$L|_{M'}$ are homotopic as trivializations of~$L$ over~$M$.

\item\label{LBisomExt2_it} If $\fc\!\ge\!3$, every smooth trivialization
of $L|_{M'}$ is homotopic through trivializations of $L|_{M'}$
to the restriction of a smooth trivialization of~$L$ over~$M$.

\end{enumerate}
\end{crl}

\begin{proof}
By induction, we can assume that $\ell\!=\!1$ and $V\!\equiv\!V_1$ is a closed submanifold
of~$M$ of codimension~$\fc$.
By the Tubular Neighborhood Theorem \cite[(12.11)]{BJ} and 
excision \cite[Corollary~4.6.5]{Spanier},
$$H^i(M,M';\Z)\approx H^i(\cN_XV,\cN_XV\!-\!V;\Z) \qquad\forall~i\!\in\!\Z.$$
Since $H_i(\R^{\fc},\R^{\fc-1})\!=\!0$ for $i\!<\!\fc$, 
$$H^i(M,M';\Z)\approx H^i(\cN_XV,\cN_XV\!-\!V;\Z)=0 \qquad\forall~i\!\le\!\fc\!-\!1$$
by Lemma~\ref{FibLmm_lmm}.
By the cohomology long exact sequence for the pair $(M,M')$,
the sequences
\BE{LBisomExt_e5}
0\lra H^i(M;\Z) \lra H^i(M';\Z)\lra0,  ~i\!\le\!\fc\!-\!2,
\quad 0\lra H^{\fc-1}(M;\Z) \lra H^{\fc-1}(M';\Z),\EE
where the second arrows are the restriction homomorphisms, are thus exact.\\

\noindent
(1) Suppose $\fc\!\ge\!2$.
If $\Phi,\Phi'\!:L\!\lra\!\C$ are  trivializations of~$L$,
there exists a smooth map \hbox{$f\!:M\!\lra\!\C^*$} such~that 
$$\Phi'(v) = f\big(\pi(x)\big)\Phi(v) \qquad\forall~v\!\in\!L.$$
If $\Phi|_M$ and $\Phi'|_{M'}$ are homotopic, then $f|_{M'}$ is homotopic
to  the constant map. 
By  the injectivity of~\eref{YS1H1_e} with $Y\!=\!M'$,
$f|_{M'}$ then corresponds to the trivial element of $H^1(M';\Z)$.
By the $H^1$ case of~\eref{LBisomExt_e5}, $f$ corresponds to 
the trivial element of $H^1(M;\Z)$ and so is homotopic to the identity.
Thus, $\Phi$ and $\Phi'$ are homotopic as trivializations of~$L$ over~$M$.\\

\noindent
(2) Suppose $\fc\!\ge\!3$ and $\Phi\!:L|_{M'}\!\lra\!\C$ is a trivialization of $L|_{M'}$.
By the $H^2$ case of \eref{LBisomExt_e5} and $c_1(L)|_{M'}\!=\!0$, $c_1(L)\!=\!0$.
Therefore, there exists a trivialization $\Phi'\!: L\!\lra\!\C$ of~$L$ over~$M$.
Since $\Phi$ and $\Phi'|_{M'}$ are trivializations of~$L$ over~$M'$,
there exists a smooth map $f\!:M'\!\lra\!\C^*$ such that 
\BE{LBisomExt_e7} \Phi'(v) \equiv f\big(\pi(v)\big)\Phi(v) \qquad\forall~v\!\in\!L|_{M'}.\EE
By the bijectivity of~\eref{YS1H1_e} and
 the $i\!=\!1$ case of~\eref{LBisomExt_e5}, there exists  
a smooth map $g\!: M\!\lra\!\C^*$ such that $g|_{M'}$ is homotopic to~$f$.
Define a trivialization of~$L$ by
$$\Phi''\!: L\!\lra\!\C, \qquad \Phi''(v) = g\big(\pi(v)\big)^{-1}\Phi'(v).$$
By~\eref{LBisomExt_e7}, $\Phi''|_{M'}$ is homotopic to~$\Phi$.
\end{proof}

\subsection{Equivalence of input and output trivializations}
\label{TrivExt_subs}

\noindent
We now show that the complex line bundle~\eref{PsiDfn_e} and its trivialization~$\Phi$
used in the construction of Section~\ref{SympSumPf_sec}
extend over a neighborhood of $\cZ_0'\!\equiv\!\io_{\cC,\ve}(X_{\eset})$ in 
the space~$\cZ'$ constructed in Section~\ref{FibrConstr_subs}.
Furthermore, these extensions can be chosen to lie in the homotopy class of 
 trivializations determined as in Proposition~\ref{InducedIsom_prp}
by the nearly regular symplectic fibration $(\cZ,\wt\om_{\cC}^{(\ve)},\pi_{\cC,\ve})$ 
constructed in Section~\ref{SympSumPf_sec}.
The proof of Propositions~\ref{PhiExt_prp} below makes use of 
straightforward, though somewhat technical, Lemmas~\ref{TrivExt_lmm0} and~\ref{TrivExt_lmm1};
they are deferred to the end of this section.

\begin{prp}\label{PhiExt_prp}
Suppose $\X$ is as in Theorem~\ref{SympSum_thm},
$\Phi$ is a trivialization of the complex line bundle $\cO_{X_{\prt}}(X_{\eset})$
as in the first paragraph of Section~\ref{SympSumPf_sec},
$\ve$ and $\cC$ are $\R^+$-valued functions on~$X_{\eset}$ satisfying the
conditions of Propositions~\ref{sympform_prp}, and
$(\cZ,\wt\om_{\cC}^{(\ve)},\pi_{\cC,\ve})$ is the corresponding one-parameter family of smoothings.
Then there exists a complex line bundle $\cO_{\cZ}(\cZ_0)$
with the canonical section~$s_{\eset}$ and a trivialization~$\wt\Phi_{\cC,\ve}$
such~that the associated embedding~\eref{iocCvedfn_e} induces an isomorphism
\BE{ndioisom_e}\nd\io_{\cC,\ve}|_{X_{\prt}}\!:
\cO_{X_{\prt}}(X_{\eset})\lra \io_{\cC,\ve}^*\cO_{\cZ}(\cZ_0)|_{X_{\prt}}\EE
of complex line bundles over~$X_{\prt}$ and
\BE{PhiExtPrp_e}
\wt\Phi_{\cC,\ve}\!\circ\!s_{\eset}\!=\!\pi_{\cC,\ve}\!:\cZ\lra\C,\qquad
\wt\Phi_{\cC,\ve}\!\circ\!\nd\io_{\cC,\ve}|_{X_{\prt}}\!=\!\cC\ve^{-1}\Phi\!:
\cO_{X_{\prt}}(X_{\eset})\lra\C\,.\EE
\end{prp}

\vspace{.2in}

\noindent
We continue with the notation and setup of Section~\ref{SympSumPf_sec}.
For $i\!\in\![N]$, let $\ri{X}_i\!\subset\!X_i$ be as in~\eref{riXdfn_e},
$\cZ_i'\!\subset\!\cZ'$ be the image of~$X_i$ under the map~$\io_{\cC,\ve;i}$
in~\eref{ioZdfn_e}, and
$$\pi_i\!=\!\pi_{\{i\}}\!:
\cN_{X_{\eset}}X_{\{i\}}\!=\!\cN X_i\!\equiv\!X_i\!\times\!\C\lra X_i$$
be the projection to the first coordinate.
Denote by $\rho_{\{i\};i}$ and $\al_{\{i\};i}$ the pullbacks of 
the function~$r^2$ and the 1-form $\nd\th$, respectively,
by the projection $\cN X_i\!\lra\!\C$.\\

\noindent
If in addition $i\!\in\!I\!\subset\![N]$,  let
$$X_{I;i}=\cZ_I\!\cap\!\cN_{I;i}
\subset\Dom(\Psi_{I;i}^{(\ve)})\big|_{\ri{X}_I}\subset\cN_{I;i}
\qquad\hbox{and}\qquad
X_{i;I}=\Psi_{I;i}^{(\ve)}\big(X_{I;i}\big)\subset X_i$$ 
be as in~\eref{XIIprdfn_e}.
Denote~by
$$\pi_{I;i}\!:\cN_{I;i}\lra X_I \qquad\hbox{and}\qquad 
\pr_{I;i}\!:\cN X_I\!=\!\pi_{I;i}^*\cN_{X_{I-i}}X_I\lra \cN_{I;i}$$
the projection maps (if $I\!=\!\{i\}$, $\pr_{I;i}\!=\!\pi_i$).
Let $\fI_{I;i}$ be the complex structure on the oriented rank~2 vector bundle
$\cN_{X_{I-i}}X_I$ determined 
by~$\rho_{I;i}^{\,\R}$ if $|I|\!\ge\!2$ and the standard complex structure if $|I|\!=\!1$.\\

\noindent
For each $i\!\in\![I]$, we will construct a complex line bundle and 
a smooth~map,
$$\pi_{\cN\cZ_i}\!:\cN\cZ_i\lra X_i
\qquad\hbox{and}\qquad \Psi_i'\!:\cN'\cZ_i\lra\cZ',$$
respectively, so that the latter is a diffeomorphism from a neighborhood~$\cN'\cZ_i$
of~$X_i$ in~$\cN\cZ_i$ onto  an open neighborhood of $\cZ_i'$ in~$\cZ'$ 
and restricts to~$\io_{\cC,\ve;i}$ over~$X_i$.
Under the identification of~$\cN\cZ_i$ with the vertical tangent subbundle 
of $T\cN\cZ_i|_{X_i}$, the homomorphism
\BE{Thidfn_e0}\begin{split}
&\Th_i\!:\cN\cZ_i\lra 
\cN\!\io_{\cC,\ve;i}\!\equiv\!\frac{\io_{\cC,\ve;i}^{\,*}T\cZ'}{\nd\io_{\cC,\ve;i}(TX_i)}
=\io_{\cC,\ve;i}^{\,*}\cN_{\cZ'}\cZ_i'\!\equiv\!
\io_{\cC,\ve;i}^{\,*}\frac{T\cZ'|_{\cZ_i'}}{T\cZ_i'},\\
&\hspace{1.2in} \Th_i(v)=\big[\nd_{\pi_{\cN\cZ_i}(v)}\Psi_i'(v)\big],
\end{split}\EE
is an isomorphism of rank~2 real vector bundles over~$X_i$ such that 
\BE{Psiidfn_e}\Psi_i\!\equiv\!\Psi_i'\!\circ\!\Th_i^{-1}\!:
\cN_i'\!\equiv\!\Th_i\big(\cN'\cZ_i\big)\lra \cZ'\EE
is a regularization of~$\cZ_i'$ in~$\cZ'$ in the sense of Definition~\ref{smreg_dfn}.
We will show that the complex structure~$\fI_i$
in the fibers of~$\cN\cZ_i$ is $\Th_i^*(\wt\om_{\cC}^{(\ve)}|_{\cN_{\cZ'}\cZ_i'})$-compatible
and~that
\BE{Psiio_e}
\io_{\cC,\ve;j}(\Psi_{ij;j}(v))=
\Psi_i\big(\nd_{\pi_{ij;j}(v)}\io_{\cC,\ve;j}(v)\big)\in\cZ_j'
\quad
\forall\,v\!\in\!\cN_{ij;j}(2\ve^2),\,i,j\!\in\![N],\,i\!\neq\!j.\EE

\vspace{.2in}

\noindent 
The pairs $(\Psi_i',\fI_i)$  and $(\Psi_i,\Th_{i*}\fI_i)$ 
determine complex line bundles $\cO_{\cZ^*}'(\cZ_i')$ and $\cO_{\cZ^*}(\cZ_i')$,
respectively,
over a neighborhood~$\cZ^*$ of~$\cZ_0'$ in~$\cZ'$ as in~\eref{cOXVdfn_e}.
The maps~\eref{Thidfn_e0} with $i\!\in\![N]$ induce an isomorphism
\BE{ThProd_e}\Th\!:\cO_{\cZ^*}'(\cZ_0')\!\equiv\!\bigotimes_{i=1}^N\cO_{\cZ^*}'(\cZ_i')
\lra\cO_{\cZ^*}(\cZ_0')
\!\equiv\!\bigotimes_{i=1}^N\cO_{\cZ^*}(\cZ_i')\EE
of complex line bundles over~$\cZ^*$.
By~\eref{Psiio_e}, the differentials of the maps $\io_{\cC,\ve;j}$ induce an isomorphism
as in~\eref{ndioisom_e} with~$\cZ_0$ replaced by~$\cZ_0'$.
We will describe the line bundle $\cO_{\cZ^*}'(\cZ_0')$
explicitly over trivializing open sets and construct a section~$s_{\eset}'$
and a trivialization~$\wt\Phi_{\cC,\ve}'$ of this bundle so~that
\BE{wtPhis0_e}\wt\Phi_{\cC,\ve}'\!\circ\!s_{\eset}'\!=\!\pi_{\cC,\ve}\!:
\cZ^*\lra\C\EE
and $s_{\eset}\!\equiv\!\Th\!\circ\!s_{\eset}'$ is 
the canonical section of $\cO_{\cZ^*}(\cZ_0)$.
By~\eref{wtPhis0_e}, the trivialization 
\BE{wtPhicCvedfn_e}\wt\Phi_{\cC,\ve}\!\equiv\!\wt\Phi_{\cC,\ve}'\!\circ\!\Th^{-1}\!:
\cO_{\cZ^*}(\cZ_0)\lra\C\EE
satisfies the first equality in~\eref{PhiExtPrp_e}.
We conclude by establishing the second equality in~\eref{PhiExtPrp_e}.\\

\noindent 
For all $i\!\in\!I\!\subsetneq\!I'\!\subset\![N]$,
$$\pi_{I';i}^*\cN_{X_{I'-i}}X_{I'}\big|_{\cN_{I';i}(4^N)}
=\pi_{I';i}^*\cN_{I';I'-i}\big|_{\cN_{I';i}(4^N)}
\subset \pi_{I';I}^*\cN_{I';I'-I}\big|_{\cN_{I';I}(4^N)}\subset \cN X_{I'}$$
under identifications as in the second equality in~\eref{cNXIide_e}.
Thus, the bundle homomorphism in~\eref{fDeIIpr0_e} if $|I|\!\ge\!2$ and 
in~\eref{DPsiext_e} if $|I|\!=\!1$ with~$I$ and~$I'$ switched
 restricts to a bundle homomorphism
\BE{TrivEx0_e0}
\fD\Psi_{I';I}^{(\ve)}\!: 
\pi_{I';i}^*\cN_{X_{I'-i}}X_{I'}\big|_{\cN_{I';i}(4^N)}
\lra \pi_{I;i}^*\cN_{X_{I-i}}X_I
\big|_{\cN_{I;i}(4^N)|_{\Im(\Psi_{I';I}^{(\ve)})}}\,.\EE
If $|I|\!\ge\!2$ or $I'\!=\!\{i\}$, \eref{TrivEx0_e0} is an isomorphism of Hermitian line bundles.
If $I\!=\!\{i\}$ and \hbox{$|I'|\!\ge\!2$}, 
\eref{TrivEx0_e0} restricts to an isomorphism of rank~2 vector bundles
over the subspace~\eref{diffsub_e} with $I$ and~$I'$ interchanged;
this subspace contains~$X_{I';i}$.\\

\noindent
For each $i\!\in\!I$, define
\begin{gather*}
\cN\cZ_i=\bigg(\bigsqcup_{i\in I\subset[N]}\!\!\!\!\!
\pi_{I;i}^*\cN_{X_{I-i}}X_I\big|_{X_{I;i}}\bigg)\!\!\bigg/\!\!\!\sim\,,\\
\pi_{I';i}^*\cN_{X_{I'-i}}X_{I'}\big|_{\Psi_{I';i}^{(\ve)\,-1}(X_{i;I}\cap X_{i;I'})} \ni v
\sim \fD\Psi_{I';I}^{(\ve)}(v)\in  
\pi_{I;i}^*\cN_{X_{I-i}}X_I\big|_{\Psi_{I;i}^{(\ve)\,-1}(X_{i;I}\cap X_{i;I'})}
~~\forall\,i\!\in\!I\!\subset\!I'\!\subset\![N].
\end{gather*}
By  the first statement in~\eref{Psiover_e} and~\eref{SCCcons_e4}, 
$\sim$ is an equivalence relation.
By  the first statement in~\eref{SCCepcons_e2}, the~map
\begin{gather}
\label{cNZidfn_e}
\pi_{\cN\cZ_i}\!:\cN\cZ_i\lra X_i\!=\!
\bigcup_{i\in I\subset[N]}\!\!\!\!\!\Psi_{I;i}^{(\ve)}(X_{I;i}), \\
\notag
 \pi_{\cN\cZ_i}\big([v]\big)=\Psi_{I;i}^{(\ve)}\big(\pr_{I;i}(v)\big)
\quad\forall~v\!\in\!\pi_{I;i}^*\cN_{X_{I-i}}X_I\big|_{X_{I;i}}\,,~i\!\in\!I\!\subset\![N],
\end{gather}
is well-defined and determines a smooth rank~2 vector bundle.
By Lemma~\ref{TrivExt_lmm0} below, the complex structures~$\fI_{I;i}$ with $i\!\in\!I$ 
induce a complex structure~$\fI_i$ on the vector bundle~\eref{cNZidfn_e}.\\

\noindent
For $i\!\in\!I\!\subset\![N]$, define
\begin{gather*}
\de_I\!:X_I^{\circ}\lra\R^+, \qquad
\de_I(x)=\begin{cases}2,&\hbox{if}~|I|\!\ge\!2;\\
2\cC(x)^2\wt\ve_i(x)^2,&\hbox{if}~I\!=\!\{i\};
\end{cases}\\
\cN'\cZ_{i;I}=\big\{v\!\in\!\pi_{I;i}^*\cN_{X_{I-i}}X_I|_{X_{I;i}}\!:
\rho_{I;i}(v)\!<\!\de_I(\pi_I(v))\big\}.
\end{gather*}
By~\eref{cZequiv_e} and Proposition~\ref{cZtopol_prp}, the~map 
\begin{gather}
\notag
\Psi_i'\!:\cN'\cZ_i\!\equiv\!
\bigg(\bigsqcup_{i\in I\subset[N]}\!\!\!\!\!\cN'\cZ_{i;I}\!\!\bigg)\!\!\bigg/\!\!\!\sim\,\lra
\cZ'\!=\!\bigg(\bigsqcup_{I\in\cP^*(N)}\!\!\!\!\!\!\cZ_I\bigg)\!\!\bigg/\!\!\!\sim,\\
\label{Psiiprdfn_e}
\Psi_i'\big([v]\big)=q(v)\in q(\cZ_{I;i})
\quad\forall~v\!\in\!\cN'\cZ_{i;I}\,,~i\!\in\!I\!\subset\![N],
\end{gather}
is well-defined and smooth.
By~\eref{ioZdfn_e} and the first statement in~\eref{SCCepcons_e2}, 
\BE{ioPsipr_e}
\Psi_i'\big([v]\big)=
\io_{\cC,\ve;j}\big(\Psi_{ij;j}^{(\ve)}\big(\fD\Psi_{I;ij}^{(\ve)}(v)\big)\!\big)
\quad\forall~v\!\in\!\cN'\cZ_{i;I}
\big|_{X_{I;i}\cap X_{I;j}}\,,~i,j\!\in\!I\!\subset\![N],~i\!\neq\!j.\EE

\vspace{.1in}

\noindent
For $I_0\!\subset\!I\!\subset\![N]$ with $I_0\!\neq\!\eset$, let
\begin{gather*}
\cZ_{I;I_0}^*=\big\{v\!\in\!\cZ_I\!\!:\,\rho_{I;i}(v)\!<\!\de_I(\pi_I(v))
~\forall\,i\!\in\!I_0,~
\rho_{I;i}(v)\!\neq\!0~\forall\,i\!\in\!I\!-\!I_0\big\},\\
\pi_{I;I_0}^{\cO}\!:\cO_{I;I_0}'(\cZ_0')\!=\!
\bigotimes_{i\in I_0}\pi_I^*\cN_{X_{I-i}}X_I\big|_{\cZ_{I;I_0}^*}
\lra \cZ_{I;I_0}^*\,.
\end{gather*}
If $|I|\!=\!1$, then $I_0\!=\!I$ and $\cZ_{I;I_0}^*\!=\!\cZ_I$.
If $I_0',I\!\subset\!I'\!\subset\![N]$ with $I_0'\!\neq\!\eset$, then
$$\cZ_{I';I}\!\cap\!\cZ_{I';I_0'}^*=\eset~~\hbox{if}~I_0'\!\not\subset\!I,\quad
\fD\Psi_{I';I}^{(\ve)}\big(\cZ_{I';I}\!\cap\!\cZ_{I';I_0'}^*\!\cap\!\cZ_{I';I_0}^*\big)
=\cZ_{I;I'}\!\cap\!\cZ_{I;I_0'}^*\!\cap\!\cZ_{I;I_0}^*~~\hbox{if}~I_0'\!\subset\!I.$$
In the second case, the diffeomorphism
$$\fD\Psi_{I';I}^{(\ve)}\!: \cZ_{I';I}\!\cap\!\cZ_{I';I_0'}^*\!\cap\!\cZ_{I';I_0}^*
\lra \cZ_{I;I'}\!\cap\!\cZ_{I;I_0'}^*\!\cap\!\cZ_{I;I_0}^*$$
lifts to a bundle isomorphism
\begin{gather*}
\fD\Psi_{I';I}^{(\ve)}\!:
\cO_{I';I_0'}'(\cZ_0')\big|_{\cZ_{I';I}\cap\cZ_{I';I_0'}^*\cap\cZ_{I';I_0}^*}
\lra
\cO_{I;I_0}'(\cZ_0')\big|_{\cZ_{I;I'}\cap\cZ_{I;I_0'}^*\cap\cZ_{I;I_0}^*}
\qquad\hbox{s.t.}\\
\fD\Psi_{I';I}^{(\ve)}\bigg((v_i')_{i\in I'},
\bigotimes_{i\in I_0'-I_0}\!\!\!\!v_i'\otimes\!\bigotimes_{i\in I_0'\cap I_0}\!\!\!\!w_i'\!\bigg)
=\bigg((v_i)_{i\in I},\bigotimes_{i\in I_0-I_0'}\!\!\!\!v_i \otimes\!
\bigotimes_{i\in I_0'\cap I_0}\!\!\!\!w_i\!\bigg)
~~\hbox{if}\\
\fD\Psi_{I';I}^{(\ve)}\big((v_i')_{i\in I'}\big)=(v_i)_{i\in I},~~
\fD\Psi_{I';I}^{(\ve)}
\big((v_i')_{i\in I'-I},(w_i')_{i\in I_0'\cap I_0}\big)
=(w_i)_{i\in I_0'\cap I_0}\in\cN X_I\,.
\end{gather*}

\vspace{.2in}

\noindent
The union~$\cZ^*$ of the open subsets $q(\cZ_{I;I_0}^*)\!\subset\!\cZ'$ 
is a neighborhood of $\cZ_0'$ in~$\cZ'$.
Define
\begin{gather*}
\cO_{\cZ^*}'(\cZ_0')=
\bigg(\bigsqcup_{\begin{subarray}{c}I_0\subset I\subset[N]\\ I_0\neq\eset\end{subarray}}
\!\!\!\!\!\!
\cO_{I;I_0}'(\cZ_0')\bigg)\!\!\bigg/\!\!\!\sim,\\
\cO_{I';I_0'}'(\cZ_0')\big|_{\cZ_{I';I}\cap\cZ_{I';I_0'}^*\cap\cZ_{I';I_0}^*}
\ni u\sim
\fD\Psi_{I';I}^{(\ve)}(u)\in\cO_{I;I_0}'(\cZ_0')
\big|_{\cZ_{I;I}\cap\cZ_{I;I_0'}^*\cap\cZ_{I;I_0}^*}
~~\forall\,I_0,I_0'\subset\!I\!\subset\!I'.
\end{gather*}
By~the first statement in~\eref{Psiover_e} and~\eref{SCCcons_e4},
$\sim$ above is an equivalence relation.
Furthermore, the~map
\BE{cOcZdfn_e}\pi^{\cO}_{\cZ^*}\!:\cO_{\cZ^*}'(\cZ_0')\lra\cZ^*,\qquad
\pi^{\cO}_{\cZ^*}\big([u]\big)=q\big(\pi_{I;I_0}^{\cO}(u)\big),\EE
is well-defined and determines a smooth complex line bundle.\\

\noindent
For $I\!\in\!\cP^*(N)$, let $\Phi_{\cC,\ve;I}$ be as 
in~\eref{SympSum_e27} and~\eref{PhiSimpDfn_e}.
Define a smooth map
\begin{gather}\notag
\wt\Phi_{\cC,\ve}'\!:\cO_{\cZ^*}'(\cZ_0')\lra \C,\\
\label{wtPhidfn_e}
\wt\Phi_{\cC,\ve}'\Big(\big[(v_i)_{i\in I},\bigotimes_{i\in I_0}w_i\big]\Big)
=\Phi_{\cC,\ve;I}\big((v_i)_{i\in I-I_0},(w_i)_{i\in I_0}\big)
~~\forall\,\big((v_i)_{i\in I},\bigotimes_{i\in I_0}\!w_i\big)
\!\in\!\cO_{I;I_0}'(\cZ_0')\,.
\end{gather}
By the first statement in~\eref{Psiover_e} and~\eref{PhiPsi_e2},
this map is well-defined.
The restriction of~$\wt\Phi_{\cC,\ve}'$ to each fiber of \eref{cOcZdfn_e} 
is an isomorphism because
$$v_i\neq0 \qquad\forall~\big((v_i)_{i\in I},\bigotimes_{i\in I_0}\!w_i\big)
\!\in\!\cO_{I;I_0}'(\cZ_0'),\,i\!\in\!I\!-\!I_0,$$
and \eref{PhiIisom_e} is an isomorphism of complex line bundles over
$\ri{X}_I\!\subset\!X_I^{\star}$.
Define a smooth section of~\eref{cOcZdfn_e} by
\BE{wtPhis0_e0}s_{\eset}'\big(q\big((v_i)_{i\in I}\big)\!\big)=
\big[(v_i)_{i\in I},\bigotimes_{i\in I_0}\!v_i\big]\in 
\cO_{I;I_0}'(\cZ_0')
\qquad\forall~(v_i)_{i\in I}\in\cZ_{I;I_0}^*,~I_0\!\subset\!I\!\subset\![N],~I_0\!\neq\!\eset.\EE
By~\eref{wtPhidfn_e} and~\eref{piRdfn_e}, the section~$s_{\eset}'$
and the trivialization~$\wt\Phi_{\cC,\ve}'$ of~\eref{cOcZdfn_e} satisfy~\eref{wtPhis0_e}.\\

\noindent
For $i\!\in\!I\!\subset\![N]$, let
\BE{wtThIidfn_e}\wt\Th_{I;i}\!:\pi_{I;i}^*\cN_{X_{I-i}}X_I\big|_{X_{I;i}}\lra 
\pi_{I;i}^*\cN X_I\big|_{X_{I;i}}\lra  T\cZ_I\big|_{X_{I;i}}\EE
denote the inclusion of $\pi_{I;i}^*\cN_{X_{I-i}}X_I$ as a subbundle of 
the restriction of the vertical tangent bundle of the total space
of the fibration $\pi_I\!:\cN X_I\!\lra\!X_I$.
Since the maps~\eref{fDeIIpr0_e} and~\eref{DPsiext_e} are isomorphisms of split vector bundles,
\begin{gather}\label{wtThIidfn_e2}
\wt\Th_{I;i}\big(\fD\Psi_{I';I}^{(\ve)}(v)\big)
=\nd_{\Psi_{I';I}^{(\ve)}(\pi_{I';i}(v))}\fD\Psi_{I';I}^{(\ve)}\big(\wt\Th_{I';i}(v)\big)\\
\notag
\forall~v\!\in\!\pi_{I';i}^*\cN_{X_{I'-i}}X_{I'}
\big|_{\Psi_{I';i}^{(\ve)\,-1}(X_{i;I}\cap X_{i;I'})},~
i\!\in\!I\!\subset\!I'\!\subset\![N].
\end{gather}
By the first statement of Lemma~\ref{TrivExt_lmm1} below, the composition
\BE{TrivExt1_e0}\Th_{I;i}\!:\pi_{I;i}^*\cN_{X_{I-i}}X_I\big|_{X_{I;i}}\lra 
T\cZ_I\big|_{X_{I;i}}\lra  
\frac{T\cZ_I|_{X_{I;i}}}{TX_{I;i}}\!=\!\cN_{\cZ_I}X_{I;i}\EE
of $\wt\Th_{I;i}$ with the quotient projection
is an isomorphism of vector bundles over~$X_{I;i}$.
By~\eref{wtThIidfn_e2}, the isomorphisms~\eref{TrivExt1_e0} induce an isomorphism
\begin{gather}
\notag
\Th_i\!:\cN\cZ_i\lra \cN\!\io_{\cC,\ve;i}\!\equiv\!
\frac{\io_{\cC,\ve;i}^{\,*}T\cZ'}{\nd\io_{\cC,\ve;I}(TX_i)}\,,\\
\label{wtThidfn_e}
\Th_i\big([v]\big)=\big[\Psi_{I;i}^{(\ve)}\big(\pr_{I;i}(v)\big),
\nd_{\pr_{I;i}(v)}q\big(\wt\Th_{I;i}(v)\big)\big]
\quad\forall~v\!\in\!\pi_{I;i}^*\cN_{X_{I-i}}X_I\big|_{X_{I;i}}\,,~i\!\in\!I\!\subset\![N],
\end{gather}
of rank~2 real vector bundles over~$X_i$.
By~\eref{Psiiprdfn_e}, this isomorphism is described by the second line in~\eref{Thidfn_e0}
and thus the map~$\Psi_i$ in~\eref{Psiidfn_e} is a regularization of~$\cZ_i'$ in~$\cZ'$.
By the statement of Lemma~\ref{TrivExt_lmm1} below, the complex structure~$\fI_i$
on the fibers of~$\cN\cZ_i$
is $\Th_i^*(\wt\om_{\cC}^{(\ve)}|_{\cN_{\cZ'}\cZ_i'})$-compatible.\\

\noindent
The complex line bundles $\cO_{\cZ^*}(\cZ_i')$  over~$\cZ^*$ as in~\eref{cOXVdfn_e}
obtained from the identifications~$\Psi_i$ and the complex structures $\Th_{i*}\fI_i$
are given~by
\BE{cOZstardfn_e}\begin{split}
&\quad
\cO_{\cZ^*}(\cZ_i')= \big(\Psi_i^{-1\,*}\pi_{\cN_{\cZ'}\cZ_i'}^*\cN_{\cZ'}\cZ_i'
|_{\Psi_i'(\cN'_i)} \sqcup 
(\cZ^*\!-\!\cZ_i')\!\times\!\C\big)\big/\!\!\sim\,\lra \cZ^*,\\
&\Psi_i^{-1\,*}\pi_{\cN_{\cZ'}\cZ_i'}^*\cN_{\cZ'}\cZ_i'
|_{\Psi_i(\cN'_i)} 
  \ni\big(\Psi_i(v),v,cv\big)\sim\big(\Psi_i(v),c\big)\in \big(\cZ^*\!-\!\cZ_i'\big)\!\times\!\C.
\end{split}\EE
The isomorphisms~$\Th_i$ with $i\!\in\![N]$ induce an isomorphism~$\Th$
as in~\eref{ThProd_e}.
By~\eref{wtPhis0_e0} and~\eref{wtThidfn_e}, $s_{\eset}\!\equiv\!\Th\!\circ\!s_{\eset}'$
is the canonical section of $\cO_{\cZ^*}(\cZ_0')$.
We define a trivialization~$\Phi_{\cC,\ve}$ of $\cO_{\cZ^*}(\cZ_0')$ by~\eref{wtPhicCvedfn_e}.\\

\noindent
Let $i\!\in\![N]$. 
By~\eref{SympSum_e17} and the last condition in Definition~\ref{smreg_dfn}, 
$$\big[\nd_{\pi_{ij}(v)}\Psi_{ij;j}^{(\ve)}(v)]=\ve\big(\pi_{ij}(v)\big)v
\in\cN_{X_j}X_{ij}\!=\!\frac{TX_j|_{X_{ij}}}{TX_{ij}}
\quad\forall\,[v]\!\in\!\cN_{X_j}X_{ij}\,.$$
By~\eref{ioZdfn_e} and the first statement in~\eref{SCCepcons_e2}, 
$$\io_{\cC,\ve;j}\big(\Psi_{ij;j}^{(\ve)}\big(\fD\Psi_{I;ij}^{(\ve)}(v)\big)\!\big)=q(v)
\in\cZ_j' \quad\forall~v\!\in\!\cN'\cZ_{i;I}
\big|_{X_{I;i}\cap X_{I;j}}\,,~i,j\!\in\!I\!\subset\![N],~i\!\neq\!j.$$
Combining the last two statements with~\eref{wtThidfn_e}, we obtain
\begin{gather}\label{Thindio_e}
\Th_i\big([v]\big)=
\ve\big(\pi_{ij}\big(\fD\Psi_{I;ij}^{(\ve)}(v)\!\big)\!\big)\,
\nd_{\pi_{ij}(\fD\Psi_{I;ij}^{(\ve)}(v))}
\io_{\cC,\ve;j}\big(\fD\Psi_{I;ij}^{(\ve)}(v)\big)\in \cN_{\cZ'}\cZ_i'\\
\notag
\forall~
v\!\in\!\pi_{I;ij}^*\cN_{X_{I-i}}X_I\big|_{X_{I;i}\cap X_{I;j}}\,,
~i,j\!\in\!I\!\subset\![N],~i\!\neq\!j.
\end{gather}
Along with~\eref{Psiidfn_e}, \eref{ioPsipr_e}, and~\eref{SympSum_e17}, 
this implies~\eref{Psiio_e}.\\

\noindent
Let $j\!\in\![N]\!-\!\{i\}$ and $\cZ_{ij}'\!=\!\io_{\cC,\ve}(X_{ij})$.
By~\eref{cOZstardfn_e} and~\eref{Psiio_e},
\begin{gather*}
\io_{\cC,\ve}^*\cO_{\cZ^*}(\cZ_i')\big|_{X_j}
= \big(\Psi_{ij;j}^{-1\,*}\{\nd_{\pi_{ij;j}(v)}\io_{\cC,\ve;j}\}^*
\pi_{\cN_{\cZ'}\cZ_i'}^*\cN_{\cZ'}\cZ_i'
|_{\Psi_{ij;j}(\cN_{ij;j}(2\ve^2))} \sqcup 
(X_j\!-\!X_{ij})\!\times\!\C\big)\big/\!\!\sim,\\
\big(\Psi_{ij;j}(v),v,\nd\io_{\cC,\ve;j}(v),c\,\nd\io_{\cC,\ve;j}(v)\big)
\sim\big(\Psi_{ij;j}(v),c\big).
\end{gather*}
The line bundle $\cO_{X_j}(X_{ij})$ constructed using the identification 
$\Psi_{ij;j}|_{\cN_{ij;j}(2\ve^2)}$ is given~by
\begin{gather*}
\cO_{X_j}(X_{ij})= \big(\Psi_{ij;j}^{-1\,*}\pi_{ij;j}^*\cN_{ij;j}
|_{\Psi_{ij;j}(\cN_{ij;j}(2\ve^2))} \sqcup 
(X_j\!-\!X_{ij})\!\times\!\C\big)\big/\!\!\sim,\\
\big(\Psi_{ij;j}(v),v,cv\big)\sim\big(\Psi_{ij;j}(v),c\big).
\end{gather*}
Thus, the isomorphism
$$\nd\io_{\cC,\ve;j}\!:\cN_{ij;j}\!=\!\cN_{X_j}{X_{ij}}\lra 
\frac{T\cZ_j'|_{\cZ_{ij}'}}{T\cZ_{ij}'}=\frac{T\cZ'|_{\cZ_{ij}'}}{T\cZ_i'|_{\cZ_{ij}'}}
=\cN_{\cZ'}\cZ_i'$$
induces an isomorphism from $\cO_{X_j}(X_{ij})$ to 
$\io_{\cC,\ve}^*\cO_{\cZ^*}(\cZ_i')|_{X_j}$.
If $k\!\in\![N]\!-\!\{i\}$, then
$$\nd\io_{\cC,\ve;j}|_{X_{jk}}\!=\!\nd\io_{\cC,\ve;k}|_{X_{jk}}\!:
\cO_{X_{jk}}(X_{ijk})\!=\!
\cO_{X_j}(X_{ij})|_{X_{jk}}\!=\!\cO_{X_k}(X_{ik})|_{X_{jk}}\lra 
\io_{\cC,\ve}^*\cO_{\cZ^*}(\cZ_i')|_{X_{ij}}\,.$$
Thus, the differentials of the maps $\io_{\cC,\ve;j}$ induce an isomorphism
as in~\eref{ndioisom_e} with~$\cZ_0$ replaced by~$\cZ_0'$.\\

\vspace{.1in}

\noindent 
Let $i,j\!\in\![N]$ be distinct.
An element of $\cO_{X_{\prt}}(X_{\eset})|_x$ with $x\!\in\!X_{ij}$ is represented 
by a tensor product
\begin{gather*}
w_i'\!\otimes\!w_j'\!\otimes\!\bigotimes_{k\in I_0'}(x,v_k',w_k') \qquad\hbox{with}\\ 
I_0'\!\subset\![N]\!-\!\{i,j\},~
v_k'\!\in\!\cN_{ijk;ij}(2\ve^2),\,x\!=\!\Psi_{ijk;ij}(v_k'),~
(v_k',w_k')\!\in\!\pi_{ijk;ij}^*\cN_{ijk;ij}\,.
\end{gather*}
By~\eref{Psiover0_e} and the first statement in~\eref{SCCcons_e2},
such an element can be written~as
\begin{gather*} 
\big(x\!=\!\Psi_{I;ij}(v)\!=\!\Psi_{I;ij}^{(\ve)}(\ve(v)^{-1}v),v,
\bigotimes_{k\in I_0}\!w_k\big) 
\qquad\hbox{with}\\
I_0=\{i,j\}\!\cup\!I_0'\!\subset\!I,~
v\!\equiv\!(v_l)_{l\in I}
\!\in\!m_{\ve;I}(\cZ_{I;I_0}^*)\!\cap\!\cN_{I;ij},~(v,w_k)\!\in\!\pi_I^*\cN_{X_{I-k}}X_I,~\\
w_k'=\fD\Psi_{I;ijk}\big((v_l)_{l\in I-\{i,j,k\}},w_k\big)
=\fD\Psi_{I;ijk}^{(\ve)}\big((\ve(v)^{-1}v_l)_{l\in I-\{i,j,k\}},w_k\big)~\forall\,k\!\in\!I_0.
\end{gather*}
With these identifications, the condition~\eref{Phicompdfn_e} becomes 
\begin{gather}\label{Phicompdfn_e5}
\Phi\Big(\Psi_{I;ij}(v),v,\bigotimes_{k\in I_0}\!w_k\Big)
=\Phi\big(\Pi_{\fR;I}\big((v_k)_{k\in I-I_0},(w_k)_{k\in I_0}\big)\!\big)\\
\notag
\forall~\big(v,\bigotimes_{k\in I_0}\!w_k\big)
\in\cO_{I;I_0}'(\cZ_0')\big|_{m_{\ve;I}(\cZ_{I;I_0}^*)\cap\cN_{I;ij}},
~i,j\!\in\!I_0\!\subset\!I\!\subset\![N].
\end{gather}

\vspace{.1in}

\noindent
By~\eref{Thindio_e}, \eref{SympSum_e17}, 
and the second condition in~\eref{SympSum_e15}, 
the isomorphism~\eref{ndioisom_e} satisfies
\begin{gather*}
\nd\io_{\cC,\ve}|_{X_{ij}}\Big(
\big[\Psi_{I;ij}(v),v,\bigotimes_{k\in I_0}\!w_k\big]\Big)
=\Th\Big(\big[\ve(v)^{-1}v,\bigotimes_{k\in I_0}\!\big(\ve(v)^{-1}w_k\big)\big]\Big)\in
\cO_{\cZ^*}(\cZ_i')\\
\forall~\big(v,\bigotimes_{k\in I_0}\!w_k\big)
\in\cO_{I;I_0}'(\cZ_0')\big|_{m_{\ve;I}(\cZ_{I;I_0}^*)\cap\cN_{I;ij}},
~i,j\!\in\!I_0\!\subset\!I\!\subset\![N].
\end{gather*}
Along with~\eref{wtPhicCvedfn_e}, \eref{wtPhidfn_e}, and~\eref{SympSum_e27}, 
this implies that 
\begin{gather*}
\wt\Phi_{\cC,\ve}\Big(\nd\io_{\cC,\ve}|_{X_{ij}}\!\Big(
\big[\Phi_{I;ij}(v),v,\bigotimes_{k\in I_0}\!w_k\big]\!\Big)\!\!\Big)
=\cC\big(\pi_I(v)\big)\ve\big(\pi_I(v)\big)^{-1}
\Phi\big(\Pi_{\fR;I}\big((v_k)_{k\in I-I_0},(w_k)_{k\in I_0}\big)\!\big)\\
\forall~\big(v,\bigotimes_{k\in I_0}\!w_k\big)
\!\equiv\!\big((v_k)_{k\in I},\bigotimes_{k\in I_0}\!w_k\big)
\in\cO_{I;I_0}'(\cZ_0')\big|_{m_{\ve;I}(\cZ_{I;I_0}^*)\cap\cN_{I;ij}},
~i,j\!\in\!I_0\!\subset\!I\!\subset\![N].
\end{gather*}
Combining this with~\eref{Phicompdfn_e5}, we obtain the second equality in~\eref{PhiExtPrp_e}.\\

\noindent
It remains to establish the two lemmas used above.

\begin{lmm}\label{TrivExt_lmm0}
For all $i\!\in\!I\!\subsetneq\!I'\!\subset\![N]$, the bundle homomorphism~\eref{TrivEx0_e0}
is $\C$-linear with respect to the complex structures
$\pi_{I';i}^{\,*}\fI_{I';i}$ and $\pi_{I;i}^{\,*}\fI_{I;i}$.
\end{lmm}

\begin{proof}
If $|I|\!\ge\!2$, then
the homomorphism~\eref{TrivEx0_e0} is an isomorphism intertwining
the Hermitian structures $\pi_{I';i}^{\,*}(\rho_{I';i},\al_{I';i})$ and 
$\pi_{I;i}^{\,*}(\rho_{I;i},\al_{I;i})$.
If $I'\!=\!\{i\}$, \eref{TrivEx0_e0} is the identity.
If $I\!=\!\{i\}$ and $|I'|\!\ge\!2$, the homomorphism~\eref{TrivEx0_e0} is given~by
$$\fD\Psi_{I';I}^{(\ve)}\big((v_j)_{j\in I'-i},v_i\big)
=\cC\big(\pi_I\!\big(\!(v_j)_{j\in I'}\!\big)\!\big)
\ve\big(\pi_I\!\big((v_j)_{j\in I'-i})\!\big)^{|I|-1}
\Phi\big(\Pi_{\fR;I}((v_j)_{j\in I'})\!\big)\,.$$
Since $\Phi$ is a $\C$-linear isomorphism and $\Pi_{\fR;I}$ 
is a $\C$-linear homomorphism in the~$v_i$-input,
$\fD\Psi_{I';I}^{(\ve)}$ is a $\C$-linear homomorphism in this case as~well.
\end{proof}

\begin{lmm}\label{TrivExt_lmm1}
For all $i\!\in\!I\!\subset\![N]$, the homomorphism~\eref{TrivExt1_e0} 
is an isomorphism of vector bundles over~$X_{I;i}$.
Furthermore, the restriction of $\pi_{I;i}^*\fI_{I;i}$ to 
$\pi_{I;i}^*\cN_{X_{I-i}}X_I|_{X_{I;i}}$ is a 
$\Th_{I;i}^{\,*}(\wt\om_{\cC;I}^{(\ve)}|_{\cN_{\cZ_I}X_{I;i}})$-compatible 
complex structure.
\end{lmm}

\begin{proof} 
The first claim follows from the canonical decomposition 
\BE{TrivExt1_e1}T\big(\cN X_I\big)\big|_{\cN_{I;i}}=
T\cN_{I;i}\oplus\pi_{I;i}^{\,*}\cN_{X_{I-i}}X_I\EE
of vector bundles over~$\cN_{I;i}$.
By~\eref{wtomdfn_e} and~\eref{wtomdfn_e0}, 
\begin{equation*}\begin{split}
\wt\om_{\cC;I}^{(\ve)}|_{T\cZ_I|_{X_{I;i}}}
=\Bigg(\!\!\pi_{I;i}^*\om_I&+\frac12
\nd\!\!\!\!\!\!\sum_{j\in I-\{i\}}\!\!\!\!\!\big(\ve^2\rho_{I;j}\al_{I;j}\big)\\
&+\frac12
\nd\bigg(\!\bigg(1-\!\!\prod_{j\in I-\{i\}}\!\!\!\!\!\!\eta_{I;j}\!\!\bigg)
\ve^2\rho_{I;i}\al_{I;i}
+\bigg(
\prod_{j\in I-\{i\}}\!\!\!\!\!\!\eta_{I;j}\!\!\bigg)\ve^2\Phi_{\cC,\ve;I}^*\big(r^2\nd\th\big)
\!\!\bigg)\!\!\Bigg)\!\bigg|_{T\cZ_I|_{X_{I;i}}}\,.
\end{split}\end{equation*}
Since $\rho_{I;i}$ vanishes on~$\cN_{I;i}$ and $\nd\rho_{I;i}$ vanishes on $T(\cN X_I)|_{\cN_{I;i}}$,
Lemma~\ref{Phipull_lmm} implies that there is a smooth $\R^+$-valued 
function~$f_{\cC,\ve;I}$
on~$X_I^{\star}$ such~that 
\BE{TrivExt1_e5}\begin{split}
\wt\om_{\cC;I}^{(\ve)}|_{T\cZ_I|_{X_{I;i}}}
&=\Bigg(\!\pr_{I;I-i}^*\big(\wh\om_I^{(\ve)}|_{\cN_{I;i}}\big)\\
&\hspace{.5in}+\frac{\ve^2}{2}
\bigg(\!\!\bigg(1-\!\!\prod_{j\in I-\{i\}}\!\!\!\!\!\!\eta_{I;j}\!\!\bigg)
+(f_{\cC,\ve;I}\!\circ\!\pi_I)\!\!\!\!\!\!\prod_{j\in I-\{i\}}\!\!\!\!\!\!\eta_{I;j}\rho_{I;j}
\!\!\bigg)\nd\big(\rho_{I;i}\al_{I;i}\big)
\!\!\Bigg)\!\bigg|_{T\cZ_I|_{X_{I;i}}}
\end{split}\EE
if $|I|\!=\!2$.
The same conclusion holds if $I\!=\!\{i\}$.\\

\noindent
By~\eref{TrivExt1_e5}, the image of the inclusion~\eref{wtThIidfn_e}
is the $\wt\om_{\cC;I}^{(\ve)}$-orthogonal
complement of $TX_{I;i}$ in~$T\cZ_I|_{X_{I;i}}$. Thus,
\BE{TrivExt1_e9}\begin{split}
\Th_{I;i}^{\,*}(\wt\om_{\cC;I}^{(\ve)}|_{\cN_{\cZ_I}X_{I;i}})
&=\wt\Th_{I;i}^{\,*}\wt\om_{\cC;I}^{(\ve)}\\
&=\frac{\ve^2}{2}
\bigg(\!\!\bigg(1-\!\!\prod_{j\in I-\{i\}}\!\!\!\!\!\!\eta_{I;j}\!\!\bigg)
+(f_{\cC,\ve;I}\!\circ\!\pi_I)\!\!\!\!\!\!\prod_{j\in I-\{i\}}\!\!\!\!\!\!\eta_{I;j}\rho_{I;j}
\!\!\bigg)\pi_{I;i}^*\nd\big(\rho_{I;i}\al_{I;i}\big)
\end{split}\EE
under the canonical identification of $\cN_{X_{I-i}}X_I$ with the vertical 
tangent bundle of  $\cN_{X_{I-i}}X_{I}\!\lra\!X_I$ along~$X_I$.
By Definition~\ref{sympreg1_dfn}\ref{sympreg1_it} and~\eref{ombund_e},
the complex structure~$\fI_{I;i}$ is compatible with
\BE{omirestr_e}\om_i\big|_{\cN_{X_{I-i}X_I}}=
\frac12\nd\big(\rho_{I;i}\al_{I;i}\big)\big|_{\cN_{X_{I-i}}X_I}\EE
under the identification of $\cN_{X_{I-i}}X_I$ with the vertical tangent
bundle of~$T\cN_{X_{I-i}}X_I$ along~$X_I$ if \hbox{$|I|\!\ge\!2$}.
If $|I|\!=\!1$, the complex structure~$\fI_{I;i}$ is also compatible with~\eref{omirestr_e}.
Along with~\eref{TrivExt1_e9}, this implies the second claim.
\end{proof}

\appendix

\section{Connections in vector bundles}
\label{conn_app}

\noindent
This appendix contains a number of explicit computations involving connections
in complex line bundles.
All of these computations are straightforward;
we include them for the sake of completeness.\\

\noindent
Let $\pi\!:L\!\lra\!X$ be a complex line bundle and $\ze_L$ be 
the \sf{radial vector field} on the total space of~$L$,~i.e.
$$\ze_L(v)=(v,v)\in\pi^*L= TL^{\ver}\!\equiv\!\ker\nd\pi \lhra TL \qquad\forall\,v\!\in\!L.$$
A connection~$\na$ in~$L$ and determines a \sf{horizontal tangent subbundle} 
\hbox{$TL^{\hor}\!\subset\!TL$} and a splitting of the exact sequence
$$0\lra \pi^*L\lra TL\stackrel{\nd\pi}\lra\pi^*TX\lra0$$
of vector bundles over $L$, i.e.~a splitting
\BE{Lsplit_e}TL\approx TL^{\ver}\oplus TL^{\hor}\approx\pi^*L\oplus \pi^*TX\EE
such that the second isomorphism above is $(\id,\nd\pi)$;
see \cite[Lemma~1.1]{anal}.\\

\noindent
If $U\!\subset\!X$ is an open subset, a non-vanishing section $\xi\!\in\!\Ga(U;L)$ 
induces a trivialization $L|_U\!\approx\!U\!\times\!\C$ and 
\BE{naka_e}\na\xi=\ka_{\xi}\xi\EE
for some $\C$-valued 1-form $\ka_{\xi}\!\in\!\Ga(U;T^*X\!\otimes_{\R}\!\C)$.
If $\ze\!\in\!\Ga(U;L)$ is another non-vanishing section, then $\ze\!=\!f\xi$
for some $f\!\in\!C^{\i}(U;\C^*)$~and 
\BE{thch_e}\ka_{\ze}=\ka_{\xi}+f^{-1}\nd f\,.\EE
With respect to the trivialization of~$L|_U$ determined by~$\xi$,
$$T_{(x,v)}L^{\hor}=
\big\{(\dot{x},-\ka_{\xi}(\dot{x})v)\!:~\dot{x}\!\in\!T_xX\big\}
\subset T_xX\!\oplus\!\C
\qquad\forall~(x,v)\!\in\!U\!\times\!\C\,.$$
This is consistent with~\eref{thch_e} and shows that the $\C^*$-action on~$L$ preserves 
the splitting~\eref{Lsplit_e}.\\

\noindent
If $h\!:Y\!\lra\!X$ is a smooth map, a connection~$\na$ in~$L$ induces 
a connection~$\na^h$ in the line bundle $h^*L\!\lra\!Y$.
If $U\!\subset\!X$ is an open subset, $\xi$ is a non-vanishing section of~$L|_U$, 
and $\ka_{\xi}$ is as in~\eref{naka_e}, then 
$h^*\xi\!=\!\xi\!\circ\!h$ is a non-vanishing section of~$h^*L$
over the open subset $h^{-1}(U)\!\subset\!Y$ and
$$\na^h(h^*\xi)=\ka_{h^*\xi}(h^*\xi), \qquad\hbox{where}\quad
\ka_{h^*\xi}=h^*\ka_{\xi}=\ka_{\xi}\circ\nd h\,.$$

\vspace{.1in}

\noindent
Let $\xi\!\in\!\Ga(U;L)$ be as in~\eref{naka_e}.
Denote by $\cK$  the $\C$-valued 1-form on $L\!-\!X$ given in the corresponding 
trivialization of~$L$ by
\BE{aldfn_e}
\cK_{(x,v)}\big(\dot{x},\dot{v}\big)=-\fI\big(\ka_{\xi}(\dot{x})\!+\!v^{-1}\dot{v}\big)
\qquad\forall~(\dot{x},\dot{v})\!\in\!T_{(x,v)}(U\!\times\!\C),~(x,v)\!\in\!U\!\times\!\C^*.\EE
By~\eref{thch_e}, $\cK$ is well-defined (i.e.~independent of the choice of~$\xi$).
It is preserved by the $\C^*$-action and satisfies
\BE{cKprp_e}
\cK\big|_{TL^{\hor}|_{L-X}}=0, \quad
\cK\big(\ze_L(v)\big)=-\fI, ~~
\cK\bigg(\frac{\nd}{\nd\th}\ne^{\fI\th}v\Big|_{\th=0}\bigg)
=\cK\big(\fI\ze_L(v)\big)=1 ~~~\forall\,v\!\in\!L\!-\!X.\EE
The choice of the normalization for~$\cK$ is motivated by the last property above and
by~\eref{alcK_e} below.\\

\noindent
Given a Hermitian metric~$\rho$ on~$L$, let $SL\!\subset\!L$ 
denote the unit circle bundle.
If $\na$ is $\rho$-compatible, $\xi$ is a trivialization of~$L$ over an open subset
$U\!\subset\!X$, $|\xi|\!=\!1$, and $\ka_{\xi}$ is as in~\eref{naka_e}, then
\begin{gather}\notag
T_{(x,v)}(SL)=\big\{(\dot{x},cv)\!:
\dot{x}\!\in\!T_xX,\,c\!\in\!\fI \R\big\}
\subset T_xX\!\oplus\!\C \qquad\forall~(x,v)\!\in\!U\!\times\!S^1,\\
\label{kaIm_e}
\ka_{\xi}(\dot{x})+\ov{\ka_{\xi}(\dot{x})}=0~~\forall\,\dot{x}\!\in\!TU\,.
\end{gather}
Thus, $\ka_{\xi}$ take values in~$\fI\R$ and
the splitting~\eref{Lsplit_e} restricts to a splitting
\begin{gather*}
T(SL)\approx T(SL)^{\ver}\oplus T(SL)^{\hor}, 
\qquad\hbox{where}\\
\notag
T(SL)^{\ver}=\ker\nd\big\{\pi|_{SL}\big\}=TL^{\ver}\cap T(SL)\,, 
\qquad T(SL)^{\hor}=TL^{\hor}\big|_{SL}\,.
\end{gather*}
There is a unique $\R$-valued \sf{connection 1-form}~$\al$ on~$SL$ such~that 
$$\ker\al=T(SL)^{\hor}\,, \qquad 
\al\bigg(\frac{\nd}{\nd\th}\ne^{\fI\th}v\Big|_{\th=0}\bigg)=1~~~\forall\,v\!\in\!SL.$$
By the first and last statements in~\eref{cKprp_e},
it is given by the restriction of~$\cK$ to the tangent bundle of~$SL$.
Via the retraction 
$$\ri{L}\equiv L\!-\!V\lra SL, \qquad v\lra\frac{v}{|v|}\,,$$
the 1-form~$\al$ extends to a 1-form on $\ri{L}$; 
we denote the resulting extension by~$\al$ as~well.
By~\eref{cKprp_e},
\BE{alcK_e}\al=\Re\,\cK\,.\EE
As in the main part of the paper,
we will use the same notation~$\rho$ to denote the square of the norm function on~$L$
and the Hermitian form on~$L$.\\

\noindent
A connection $\na$ in a vector bundle \hbox{$\pi\!:L\!\lra\!X$} determines
an extension~$\Om_{\na}$ of a fiberwise 2-form~$\Om$
to a 2-form on the total space of~$L$.
If $\Om$ is a fiberwise symplectic form on an oriented vector bundle \hbox{$\pi\!:L\!\lra\!X$} 
of rank~2 and
$(\rho,\na)$ is an $\Om$-compatible Hermitian structure on~$L$,
then 
\BE{Omnaprp_e}\big\{\io_{\ze_L}\Om_{\na}\big\}(\fI\ze_L)=\Om(\ze_L,\fI\ze_L)
=\rho^{\R}(\ze_L,\ze_L)=\rho(\ze_L,\ze_L).\EE
Since $\io_{\ze_L}\Om_{\na}$ vanishes on~$TL^{\hor}$ and on~$\ze_L$,
\eref{alcK_e}, \eref{cKprp_e}, and~\eref{Omnaprp_e} give
\BE{Omvsaldfn_e} \rho\al=\io_{\ze_L}\Om_{\na}\,.\EE

\begin{lmm}\label{conn_lmm0}
Suppose $\pi\!:L\!\lra\!X$ is a complex line bundle, 
$(\rho,\na)$ is a Hermitian structure in~$L$, and
$\al$ is the connection 1-form on $\ri{L}$ determined by~$(\rho,\na)$.
\begin{enumerate}[label=(\arabic*),leftmargin=*]

\item\label{conn0_it1} If $(\rho',\na')$ is another Hermitian structure in~$L$,
there exist $f\!\in\!C^{\i}(X;\R^+)$ and  an $\R$-valued 1-form
$\mu_{\fI\R}$ on~$X$ such~that 
\BE{conncomp_e}  \rho'=f^{-2}\rho, \qquad \na'=\na-f^{-1}\nd f+\fI\mu_{\fI\R}\,.\EE 

\item\label{conn0_it2} If $\rho'$ and $\na'$ are given by~\eref{conncomp_e}, then 
$(\rho',\na')$ is a Hermitian structure in~$L$ and 
the connection 1-form on $\ri{L}$ determined by~$(\rho',\na')$ is given~by
\BE{conncomp_e2}\al'=\al+\pi^*\mu_{\fI\R}\,.\EE
\end{enumerate}
\end{lmm}

\begin{proof}
Since $L$ is a complex line bundle, the first half of~\ref{conn0_it1} is clear.
If~$\na'$ is another connection in~$L$, then
$$\na'=\na+\mu$$
for some $\C$-valued 1-form on~$X$.
Along with the compatibility condition on~$(\rho,\na)$ and~$(\rho',\na')$,
this implies the second half of~\ref{conn0_it1} and the 
first half of~\ref{conn0_it2}.
If  $\rho'$ and $\na'$ are given by~\eref{conncomp_e} and
$\xi\!\in\!\Ga(U;L)$ is a local section with $\rho(\xi)\!=\!1$,
then $\xi'\!\equiv\!f\xi$ is a local section with $\rho'(\xi')\!=\!1$ and
$$\ka_{\xi'}'=\big(\ka_{\xi}\!-\!f^{-1}\nd f\!+\!\fI\mu_{\fI\R}\big)+f^{-1}\nd f
=\ka_{\xi}\!+\!\fI\mu_{\fI\R}\,.$$
Along with~\eref{alcK_e} and~\eref{aldfn_e}, this implies the second half of~\ref{conn0_it2}.
\end{proof}

\noindent
If $N\!\in\!\Z^+$ and $L_i\!\lra\!X$ is a complex line bundle for each 
$i\!=\!1,\ldots,N$,
define
$$\Pi\!:\bigoplus_{i=1}^N L_i\lra  \bigotimes_{i=1}^N L_i
\qquad\hbox{by}\quad
(v_1,\ldots,v_N)\lra v_1\!\otimes\!\ldots\!\otimes\!v_N\,.$$

\begin{lmm}\label{conn_lmm}
Let $N\!\in\!\Z^+$ and $X$ be a manifold.
For each $i\!\in\![N]$, let $(L_i,\rho_i,\na^{(i)})$ be a Hermitian line
bundle over~$X$ with induced connection 1-form~$\al_i$
on~$\ri{L}_i$.
Suppose
$$(\pi,\Phi)\!:
\bigotimes_{i=1}^N(L_i,\rho_i,\na^{(i)})\lra 
(X\!\times\!\C,\rho_{\C},\na^{\C})$$
is an isomorphism with the trivial Hermitian line bundle, 
$$\pi_i\!: \ri{L}_1\!\times_X\!\ldots\!\times_X\!\ri{L}_N\lra \ri{L}_i$$
is the projection onto the $i$-th factor, and $\wt\Phi\!=\!\Phi\!\circ\!\Pi$. 
Then,
$$\sum_{i=1}^N\pi_i^*\al_i=\wt\Phi^*\nd\th \in 
\Ga\big(\ri{L}_1\!\times_X\!\ldots\!\times_X\!\ri{L}_N;
T^*(\ri{L}_1\!\times_X\!\ldots\!\times_X\!\ri{L}_N)\big)\,.$$
\end{lmm}

\begin{proof} For each $i\!\in\![N]$, let $\xi_i\!\in\!\Ga(U;L_i)$ be such that 
$|\xi_i|\!=\!1$.
By choosing $\xi_N$ suitably, we can assume that 
$$\Phi(\xi_1\!\otimes\!\ldots\!\otimes\!\xi_N)=1\,.$$
This implies that 
\begin{gather*}
0=\na\big(\xi_1\!\otimes\!\ldots\!\otimes\!\xi_N\big)
=\bigg(\sum_{i=1}^N\ka_{\xi_i}\!\!\bigg)\xi_1\!\otimes\!\ldots\!\otimes\!\xi_N\,,\\
\wt\Phi(x,v_1,\ldots,v_N)
= v_1\!\ldots\!v_N~~~\forall~(x,v_1,\ldots,v_N)\in U\!\times\!\C^N.
\end{gather*}
Combining this with~\eref{alcK_e} and~\eref{aldfn_e}, we find that 
\begin{equation*}\begin{split}
\sum_{i=1}^N\pi_i^*\al_i
\big|_{(x,v_1,\ldots,v_N)}\big(\dot{x},\dot{v}_1,\ldots,\dot{v}_N)
&=\Im\sum_{i=1}^N\big(\ka_{\xi_i}(\dot{x})\!+\!v_i^{-1}\dot{v}_i\big)
=\Im\sum_{i=1}^Nv_i^{-1}\dot{v}_i\\
&=\Im\big(\big\{\nd_{(x,v_1,\ldots,v_N)}\ln\wt\Phi\big\}
\big(\dot{x},\dot{v}_1,\ldots,\dot{v}_N)\big)\\
&=\big\{\wt\Phi^*\nd\th\big\}\big|_{(x,v_1,\ldots,v_N)}\big(\dot{x},\dot{v}_1,\ldots,\dot{v}_N),
\end{split}\end{equation*}
as claimed.
\end{proof}

\begin{crl}\label{conn_crl}
Let $N$, $X$, $(L_i,\rho_i,\na^{(i)})$, and~$\al_i$ be as in Lemma~\ref{conn_lmm}.
If 
\BE{conncrl_e0}(\pi,\Phi)\!: \bigotimes_{i=1}^N L_i\lra X\!\times\!\C\EE
is an isomorphism of complex line bundles,
there exist $f\!\in\!C^{\i}(X;\R^+)$ and an $\R$-valued 1-form 
$\mu_{\fI\R}$ on~$X$ such~that 
\BE{conncrl_e1}  
(\pi,\Phi)\!:
\bigotimes_{i=1}^N(L_i,\rho_i,\na^{(i)})\lra 
\big(X\!\times\!\C,f^2\rho_{\C},\na^{\C}\!+\!f^{-1}\nd f\!-\!\fI\mu_{\fI\R}\big)\EE
is an isomorphism of Hermitian line bundles.
If \eref{conncrl_e1} is an isomorphism of Hermitian line bundles, then 
$$\pi^*\mu_{\fI\R}+\sum_{i=1}^N\pi_i^*\al_i
=\wt\Phi^*\nd\th \in 
\Ga\big(\ri{L}_1\!\times_X\!\ldots\!\times_X\!\ri{L}_N;
T^*(\ri{L}_1\!\times_X\!\ldots\!\times_X\!\ri{L}_N)\big)\,.$$
\end{crl}

\begin{proof}
The first claim of this corollary follows from Lemma~\ref{conn_lmm0}\ref{conn0_it1}. 
The second claim is obtained by applying Lemma~\ref{conn_lmm}~with
$$(L_1,\rho_1,\na^{(1)}) \qquad\hbox{replaced by}\qquad 
(L_1,f^{-2}\rho_1,\na^{(1)}\!-\!f^{-1}\nd f\!+\!\fI\mu_{\fI\R})$$
and then using~\eref{conncomp_e2}; see Lemma~\ref{conn_lmm0}.
\end{proof}

\noindent
With $N$, $X$, $(L_i,\rho_i,\na^{(i)})$, $\al_i$, and~$\pi_i$ as in Lemma~\ref{conn_lmm}, let
$$\pi\!: \cN=\bigoplus_{i=1}^NL_i\lra X\,.$$
A splitting of the exact sequence
\BE{cNseq_e}0\lra \pi^*\cN\lra T\cN \stackrel{\nd\pi}{\lra} \pi^*TX\lra0\EE
over $\ri{L}_1\!\times_X\!\ldots\!\times_X\!\ri{L}_N$ is obtained by taking 
$$T\cN^{\hor} =\bigcap_{i=1}^N\big(\ker\pi_i^*\al_i\cap\ker\pi_i^*\nd\rho_i\big)\subset T\cN.$$
By~\eref{cKprp_e}, \eref{aldfn_e}, and~\eref{kaIm_e},
\BE{TcNhor_e}\begin{split}
T_{(x,v_1,\ldots,v_N)}\cN^{\hor}
=\big\{(\dot{x},-\ka_{\xi_1}^{(1)}(\dot{x})v_1,\ldots,-\ka_{\xi_N}^{(N)}(\dot{x})v_N)\!:~
\dot{x}\!\in\!T_xX\big\}\subset T_xX\!\times\!\C^N  \qquad&\\
\forall~(x,v_1,\ldots,v_N)\in U\!\times\!(\C^*)^N&
\end{split}\EE
in the trivialization induced by local sections $\xi_i\!\in\!\Ga(U;L_i)$ 
with $|\xi_i|\!=\!1$.
The above splitting thus extends to a splitting 
$$T\cN =T\cN^{\ver}\oplus T\cN^{\hor}\lra \cN$$
of~\eref{cNseq_e} over the entire total space~$\cN$;
the latter restricts to the canonical splitting over $X\!\subset\!\cN$.
Via this splitting,
the complex structure~$\fI$ on the fibers of $\pi$ and 
an almost complex structure~$J$ on~$X$ induce an almost complex structure~$\wt{J}$
on the total space of~$\cN$. 

\begin{lmm}\label{almostJ_lmm}
If $N$, $X$, $(L_i,\rho_i,\na^{(i)})$, $\al_i$, $\pi_i$, and~$\wt{J}$ are as
above and
$\Phi$ is as in~\eref{conncrl_e0}, then there exists a continuous function 
$\cC_{\Phi}\!:X\!\lra\!\R^+$ with the following property.
For every $\la\!\in\!\C^*$, $v\!\in\!\wt\Phi^{-1}(\la)$, and 
$\dot{v}\!\in\!T_v\wt\Phi^{-1}(\la)$, there exists 
$\dot{w}\!\in\!T_v\cN^{\ver}$ such~that 
\BE{almostJ_e} \wt{J}\dot{v}+\dot{w}\in T_v\wt\Phi^{-1}(\la),\qquad
|\dot{w}|\le \cC_{\Phi}\big(\pi(v)\big)|\la|^{1/N}\big|\nd\pi_v(\dot{v})\big|.\EE
\end{lmm}

\begin{proof}
Let $v\!\in\!\wt\Phi^{-1}(\la)$,  $\dot{v}\!\in\!T_v\cN$, and $x\!=\!\pi(v)$.
In a trivialization as in~\eref{TcNhor_e},
$$v=(x,v_1,\ldots,v_N), \qquad 
\dot{v}=\big(\dot{x},\dot{v}_1\!-\!\ka_{\xi_1}^{(1)}(\dot{x})v_1,\ldots,
\dot{v}_N-\ka_{\xi_N}^{(N)}(\dot{x})v_N\big)$$
for some $v_i,\dot{v}_i\!\in\!\C$. 
Furthermore,
$$\wt\Phi(v)=f(x)v_1\ldots v_N, \qquad
\wt{J}\dot{v}=\big(J\dot{x},\fI\dot{v}_1\!-\!\ka_{\xi_1}^{(1)}(J\dot{x})v_1,\ldots,
\fI\dot{v}_N-\ka_{\xi_N}^{(N)}(J\dot{x})v_N\big)$$
for some $f\!\in\!C^{\i}(U;\C^*)$ determined by the trivialization.
Thus,
\BE{PhiNormBd_e} 
\big|\wt\Phi(v)\big|=\big|f(x)\big|\cdot |v_1|\ldots |v_N| \,.\EE
If $\la\!\in\!\C^*$ and $(x,v)\!\in\!\wt\Phi^{-1}(\la)$, 
then $v_i\!\neq\!0$ for all~$i$.
By symmetry and~\eref{PhiNormBd_e}, we can assume~that
\BE{PhiNormBd_e2} |v_1|\le \big|\wt\Phi(v)\big/f(x)\big|^{1/N}\,.\EE
Define
\begin{gather*}
\dot{w}_1=\fI\bigg(\frac{\nd_xf(\dot{x})\!+\!\fI\,\nd_xf(J\dot{x})}{f(x)}
-\sum_{i=1}^N\big(\ka_{\xi_i}^{(i)}(\dot{x})\!+\!\fI\ka_{\xi_i}^{(i)}(J\dot{x})\big)\bigg)v_1
\in\C,\\
\dot{w}=(0,\dot{w}_1,0,\ldots,0)\in T_x\cN.
\end{gather*}
By~\eref{PhiNormBd_e2}, $\dot{w}$ satisfies the second condition in~\eref{almostJ_e}.
If $\dot{v}\!\in\!T_v\wt\Phi^{-1}(\la)$, then
\begin{gather*}
\frac{\nd_xf(\dot{x})}{f(x)}+\sum_{i=1}^N\frac{\dot{v}_i\!-\!\ka_{\xi_i}^{(i)}(\dot{x})v_i}{v_i}
=\nd_{(x,v)}\wt\Phi(\dot{v})=0,\\
\nd_{(x,v)}\wt\Phi\big(\wt{J}\dot{v}\!+\!\dot{w}\big)=
\frac{\nd_xf(J\dot{x})}{f(x)}
+\frac{\fI\dot{v}_1\!-\!\ka_{\xi_1}^{(1)}(J\dot{x})v_1\!+\!\dot{w}_1}{v_1}
+\sum_{i=2}^N\frac{\fI\dot{v}_i\!-\!\ka_{\xi_i}^{(i)}(J\dot{x})v_i}{v_i}=0.
\end{gather*}
Thus, $\dot{w}$ also satisfies the first condition in~\eref{almostJ_e}.
\end{proof}

\begin{crl}\label{almostJ_crl}
Suppose $N$, $X$, $(L_i,\rho_i,\na^{(i)})$, $\al_i$, $\pi_i$, $\wt{J}$, and
$\Phi$ are as in Lemma~\ref{almostJ_lmm} and 
$\wt\om$ is a nondegenerate 2-form on a neighborhood~$\cN'$ of~$X$ in~$\cN$ taming~$\wt{J}$. 
For every compact subset~$K$ of~$\cN'$, there exists $r_K\!\in\!\R^+$ such that 
the restriction of~$\wt\om$ to $T_v\wt\Phi^{-1}(\la)$ is nondegenerate for 
all $v\!\in\!\wt\Phi^{-1}(\la)\!\cap\!K$ and $\la\!\in\!\C^*$ with $|\la|\!<\!r_K$.
\end{crl}

\begin{proof} Given $\dot{v}\!\in\!T_v\wt\Phi^{-1}(\la)$, let 
$\dot{w}\!\in\!T_v\cN^{\ver}$ be as in Lemma~\ref{almostJ_crl}. 
Thus,
$$\big|\wt\om\big(\dot{v},\wt{J}\dot{v}\!+\!\dot{w}\big)\!-\!
\wt\om\big(\dot{v},\wt{J}\dot{v}\big)\big|
\le\cC_{\Phi}\big(\pi(v)\big)|\la|^{1/N}\big|\nd\pi_v(\dot{v})\big||\dot{v}|
\le\cC_{\Phi}'\big(\pi(v)\big)r_K^{1/N}|\dot{v}|^2.$$
Since $\wt{J}$ tames~$\wt\om$ over the compact set~$K$, it follows that 
$\wt\om(\dot{v},\wt{J}\dot{v}\!+\!\dot{w})$ is nonzero if 
$r_K$ is sufficiently small.
\end{proof}

\section{The smoothability criterion revisited}
\label{RemainPf_sec}

\noindent
The smoothability condition~\eref{SympSumCond_e} is equivalent to the condition
$$\sum_{i=1}^N c_1\big(\cO_{X_i^c}(X_i)\big)\big|_{X_{\prt}}=0\in H^2\big(X_{\prt};\Z\big).$$
Proposition~\ref{SimpCr_prp} below provides a different description of 
the cohomology classes
$$\PD_{X^c_i}\big(X_i\big)\equiv c_1\big(\cO_{X_i^c}(X_i)\big)\in H^2(X_i^c;\Z),
\qquad i\!\in\![N].$$
It is more conceptual and less useful, but is more intrinsic.
It is also more indicative of being an obstruction to the existence
of the smoothing, since the singular fiber~$X_{\eset}$ 
of $\pi\!:\cZ\!\lra\!\De$  is homologous to a smooth one 
and the normal bundle to the latter is trivial. 

\begin{prp}\label{SimpCr_prp}
Let $(X_{\eset},(\om_i)_{i\in[N]})$ be as SC symplectic variety as 
in Definition~\ref{SCC_dfn}.
\begin{enumerate}[label=(\arabic*),leftmargin=*]

\item \label{label:poincaredualcondition}
For each $i\!\in\![N]$ and each connected component $X'_{\prt;i}$ of $X_{\prt}\!\cap\!X_i$, 
there exists a unique element 
$\PD_{X^c_i}(X'_{\prt;i})\!\in\!H^2(X_i^c;\Z)$ such~that 
\begin{alignat}{1}
\label{PDcond_e}
&\PD_{X^c_i}\big(X'_{\prt;i}\big)\big|_{X_j}
=\PD_{X_j}\big(X'_{\prt;i}\!\cap\!X_j\big)\in H^2(X_j;\Z)
\qquad\forall~j\in[N]\!-\!\{i\},\\
\label{PDcond_e1}
&\PD_{X^c_i}\big(X'_{\prt;i}\big)|_{X^c_i-X'_{\prt;i}}= 0 \in H^2(X^c_i\!-\!X'_{\prt;i};\Z).\
\end{alignat}

\item \label{label:linebundlecondition} 
For each $i\!\in\![N]$, 
$$c_1\big(\cO_{X_i^c}(X_i)\big)=\sum_{X'_{\prt;i}}\PD_{X^c_i}\big(X'_{\prt;i}\big)
\in  H^2(X_i^c;\Z),$$ 
where the sum is taken over the connected components of $X_{\prt}\!\cap\!X_i$.

\end{enumerate}
\end{prp}

\begin{proof}
For all $i,j\!\in\![N]$ distinct, let $\Psi_{ij;j}$ and $\fI_{ij;j}$ be 
as in~\eref{restrisom_e3a} and~\eref{restrisom_e3b}, respectively.
Restricting the construction of~\eref{cOXXidfn_e} to each connected component $X'_{\prt;i}$ 
of $X_{\prt}\!\cap\!X_i$, we obtain  a complex line bundle $\cO_{X_i^c}(X'_{\prt;i})$
over~$X_i^c$ such~that 
$$c_1\big(\cO_{X_i^c}(X'_{\prt;i})\big)\big|_{X_j}
= c_1\big( \cO_{X_j}(X'_{\prt;i}\!\cap\!X_j)\big) =
\PD_{X_j}\big(X'_{\prt;i}\!\cap\!X_j\big)
\qquad\forall~j\in[N]\!-\!\{i\}.$$
Thus, the cohomology class
$$\PD_{X^c_i}\big(X'_{\prt;i}\big)\equiv
c_1\big(\cO_{X_i^c}(X'_{\prt;i})\big)\in H^2(X_i^c;\Z)$$
satisfies~\eref{PDcond_e}.
Since the restriction of $\cO_{X_i^c}(X'_{\prt;i})$ to $X^c_i\!-\!X'_{\prt;i}$ 
is a trivial line bundle, it also satisfies~\eref{PDcond_e1}.
Along with Lemma~\ref{Xinject_lmm} below, this completes the proof of the first claim.
It is immediate~that
$$\cO_{X_i^c}(X_i)\equiv\bigotimes_{X'_{\prt;i}}
\cO_{X_i^c}(X'_{\prt;i})\lra X_i^c\,,$$
where the tensor product is taken over the connected components of $X_{\prt}\!\cap\!X_i$.
This implies the second claim.
\end{proof}

\begin{lmm}\label{Xinject_lmm}
Let $(X_{\eset},(\om_i)_{i\in[N]})$ be as as Proposition~\ref{SimpCr_prp}
and $i,j\!\in\![N]$ be distinct.
If $X_{\prt;i}'\!\subset\!X_{\prt}\!\cap\!X_i$ is a connected component such that 
$X_{\prt;i}'\!\cap\!X_j\!\neq\!\eset$, then the homomorphism
\BE{Xinject_e}H^2(X_i^c;\Z)\lra H^2(X_i^c\!-\!X_{\prt;i}';\Z)\oplus  H^2(X_j;\Z),\EE
induced by the restriction homomorphisms, is injective.
\end{lmm}

\begin{proof}
The kernel of the first homomorphism in~\eref{Xinject_e} is the image of the restriction
homomorphism
$$H^2\big(X_i^c,X_i^c\!-\!X_{\prt;i}';\Z\big)\lra H^2(X_i^c;\Z)\,.$$
Thus, it is sufficient to show that the composition 
\BE{Xinject_e2}H^2\big(X_i^c,X_i^c\!-\!X_{\prt;i}';\Z\big)\lra  H^2(X_i^c;\Z)\lra H^2(X_j;\Z)\EE
of the two restriction homomorphisms is injective.\\

\noindent
With notation as in~\eref{restrisom_e3a}, let
$$D_{X_k}(X_{\prt;i}')=\cN_{ik;i}'\big|_{X_{\prt;i}'\cap X_k}
~~~\forall~k\!\in\![N]\!-\!i,\qquad
D_{X_i^c}(X_{\prt;i}')=
\bigcup_{k\in[N]-i}\!\!\!\!\!\!D_{X_k}(X_{\prt;i}')\,.$$
We use the maps $\Psi_{ik;k}$ to identify these disk bundles with neighborhoods
of~$X_{\prt;i}'\!\cap\!X_k$ in~$X_k$ and of 
$X_{\prt;i}'$ in~$X_i^c$.
Since these disk bundles are oriented,
there is a commutative diagram
$$\xymatrix{H^0(X_{\prt;i}')\ar[r]^>>>>>>>>>{\approx}\ar[d]&
H^2\big(D_{X_i^c}(X'_{\prt;i}),D_{X_i^c}(X'_{\prt;i})\!-\!X'_{\prt;i}\big) 
\ar[d]& H^2\big(X_i^c,X_i^c\!-\!X_{\prt;i}'\big)\ar[d]\ar[l]_>>>>>>>>>{\approx}\\
H^0\big(X_{\prt;i}'\!\cap\!X_j\big)\ar[r]^>>>>>{\approx}&
H^2\big(D_{X_j}(X_{\prt;i}'),D_{X_j}(X_{\prt;i}')\!-\!X'_{\prt;i}\!\cap\!X_j\big) 
&H^2\big(X_j,X_j\!-\!X_{\prt;i}'\!\cap\!X_j\big)\ar[l]_>>>>>>{\approx}}$$
where the vertical arrows are the restriction homomorphisms,
the right horizontal arrows are the excision isomorphisms \cite[Corollary~4.6.5]{Spanier},
and the left horizontal arrows are the Thom isomorphisms \cite[Theorem~10.4]{MiSt}
for the disk bundle $D_{X_i^c}(X'_{\prt;i})\!\lra\!X'_{\prt;i}$ and 
its restriction to~$X_{\prt;i}'\!\cap\!X_j$.
They send the unit~1 in $H^0(X_{\prt;i}';\Z)$ and $H^0\big(X_{\prt;i}'\!\cap\!X_j;\Z\big)$
to the Thom class~$u_i$ for $D_{X_i^c}(X'_{\prt;i})$
and its restriction $u_i|_{X_{\prt;i}'\cap X_j}$, respectively.
Let 
$$u_i'|_{X_{\prt;i}'\cap X_j}\in H^2\big(X_j,X_j\!-\!X_{\prt;i}'\!\cap\!X_j;\Z\big)$$
denote the element corresponding to $u_i|_{X_{\prt;i}'\cap X_j}$ under the excision
isomorphism.
By~\cite[Exercise~11-C]{MiSt},  the restriction of $u_i'|_{X_{\prt;i}'\cap X_j}$ to~$X_j$
is $\PD_{X_j}(X_{\prt;i}'\!\cap\!X_j)$.
Since $X_{\prt;i}'\!\cap\!X_j$ is a symplectic submanifold of~$X_j$,
$$\blr{\om^{n-1}\PD_{X_j}(X_{\prt;i}'\!\cap\!X_j),X_j}
=\blr{\om^{n-1},X_{\prt;i}'\!\cap\!X_j}\neq0$$ 
if $2n\!=\!\dim_{\R}X_j$.
Thus, $\PD_{X_j}(X_{\prt;i}'\!\cap\!X_j)\!\neq\!0$ and the composition 
$$H^2\big(X_i^c,X_i^c\!-\!X_{\prt;i}';\Z\big)\lra 
H^2\big(X_j,X_j\!-\!X_{\prt;i}'\!\cap\!X_j;\Z\big) \lra H^2(X_j;\Z)$$
of the two restriction homomorphisms is nonzero even after tensoring with~$\Q$.
Since \hbox{$H^0(X_{\prt;i}';\Z)\!=\!\Z$}, it follows that this composition is
injective.
Therefore, the composition~\eref{Xinject_e2} is also injective,
as needed.
\end{proof}

\vspace{.3in}

\noindent
{\it Simons Center for Geometry and Physics, Stony Brook University, Stony Brook, NY 11794\\
mtehrani@scgp.stonybrook.edu}\\

\noindent
{\it Department of Mathematics, Stony Brook University, Stony Brook, NY 11794\\
markmclean@math.stonybrook.edu, azinger@math.stonybrook.edu}\\

\end{document}